\documentclass[10pt]{amsart}
\usepackage{mathrsfs}
\usepackage{amssymb}
\usepackage{amsmath, bm}

\usepackage{titletoc}
\usepackage{amscd}
\usepackage{amsmath, amssymb}
\usepackage{amsthm}
\usepackage{amsfonts}
\usepackage[colorlinks,linkcolor=blue,citecolor=blue, pdfstartview=FitH]{hyperref}

\usepackage[all,cmtip]{xy}

\usepackage{stmaryrd}

\usepackage{tikz-cd} 

\usepackage{extarrows} 

\numberwithin{equation}{section}

\theoremstyle{plain}
\newtheorem{thm}{Theorem}[section]

\newtheorem{lem}[thm]{Lemma}
\newtheorem{corollary}[thm]{Corollary}
\newtheorem{prop}[thm]{Proposition}

\newtheorem{thmx}{Theorem}

\theoremstyle{definition}

\newtheorem{rmk}[thm]{Remark}

\newtheorem{claim}[thm]{Claim}

\newtheorem{defn-thm}[thm]{Definition-Theorem}
\newtheorem{defn-pro}[thm]{Definition-Proposition}

\theoremstyle{remark}

\usepackage{fancyhdr}
\newcommand{\bA}{{\mathbf C}}

\newcommand{\C}{{\mathbb C}}
\newcommand{\sC}{{\mathcal C}}

\newcommand{\sD}{{\mathcal D}}

\newcommand{\F}{{\mathbb F}}
\newcommand{\bF}{{\textbf F}}
\newcommand{\sF}{{\mathcal F}}

\newcommand{\I}{{\mathrm I}}
\newcommand{\sI}{{\mathcal I}}

\newcommand{\K}{{\mathbb K}}

\newcommand{\m}{{\mathfrak m}}

\newcommand{\N}{{\mathbb N}}
\newcommand{\sN}{{\mathcal N}}

\newcommand{\sO}{{\mathcal O}}

\newcommand{\sP}{{\mathcal P}}
\newcommand{\Q}{{\mathbb Q}} 
\newcommand{\rQ}{{\mathrm Q}}
\newcommand{\sQ}{{\mathcal Q}}
\newcommand{\fQ}{{\mathfrak Q}}

\newcommand{\sT}{{\mathcal T}}
\newcommand{\sfT}{{\mathsf T}}

\newcommand{\fU}{{\mathfrak U}}

\newcommand{\sX}{{\mathfrak X}}

\newcommand{\Z}{{\mathbb Z}}

\newcommand{\bb}{\mathbf{b}} 
\newcommand{\bc}{\mathbf{c}} 
\renewcommand{\k}{\C} 

\newcommand{\id}{\mathrm{id}}

\newcommand{\pr}{\mathrm{pr}} 
\newcommand{\tr}{{\mathrm{tr}}} 
\newcommand{\Tr}{{\mathrm{Tr}}} 
 
\newcommand{\xs}{{\ \xrightarrow{\sim} \ }}

\newcommand{\ch}{\mathrm{ch}} 
\newcommand{\ev}{\mathrm{ev}} 
 
\newcommand{\Stab}{\mathrm{Stab}} 
\newcommand{\rad}{\mathrm{rad}} 
\newcommand{\Hom}{\mathrm{Hom}}
\newcommand{\End}{\mathrm{End}}

\newcommand{\Ext}{\mathrm{Ext}}

\newcommand{\op}{{\mathrm{op}}}

\newcommand{\Fun}{\mathrm{Fun}}

\newcommand{\Mod}{\text{-}\mathrm{Mod}}
\renewcommand{\mod}{\text{-}\mathrm{mod}}

\newcommand{\rep}{\mathrm{rep}}
\newcommand{\Vect}{\mathrm{Vect}}
\newcommand{\Coh}{\mathrm{Coh}}
\newcommand{\QCoh}{\mathrm{QCoh}}
 
\newcommand{\Lim}[1]{\mathop{\mathrm{lim}}\limits_{#1}}
\newcommand{\cLim}[1]{\mathop{\mathrm{colim}}\limits_{#1}}
\newcommand{\adj}{\mathrm{adj}} 
\newcommand{\Db}{D^\mathrm{b}} 

\renewcommand{\for}{\mathrm{for}}

\newcommand{\rR}{\mathrm{R}} 
\newcommand{\rI}{\mathrm{I}} 

\newcommand{\Spec}{\mathrm{Spec}}

\newcommand{\rk}{\text{rk}} 
 
\newcommand{\pt}{\mathrm{pt}} 
\newcommand{\red}{\mathrm{red}}

\newcommand{\Gm}{\mathbb{G}_{\mathrm{m}}}

\newcommand{\Lie}{\mathrm{Lie}} 
\renewcommand{\sc}{\mathrm{sc}} 

\newcommand{\g}{{\mathfrak{g}}}
\renewcommand{\b}{{\mathfrak{b}}}
\renewcommand{\t}{{\mathfrak{t}}}
\newcommand{\n}{{\mathfrak{n}}}

\newcommand{\Fr}{\mathrm{Fr}} 
\newcommand{\HC}{\mathrm{HC}} 
\newcommand{\hc}{\mathrm{hc}} 
 
\newcommand{\af}{\mathrm{af}} 
 
\newcommand{\ex}{\mathrm{ex}}

\newcommand{\Gr}{\mathcal{G}\mathfrak{r}} 
\newcommand{\Fl}{\mathcal{F}{l}} 
\newcommand{\hb}{\mathrm{hb}} 
\newcommand{\mx}{\mathrm{mix}} 
\newcommand{\pur}{\mathrm{pure}} 
\newcommand{\reg}{\mathrm{reg}}

\newcommand{\bL}{\mathbf{\Lambda}} 
\newcommand{\bP}{\mathbf{P}}

\begin{document}
\title{Center of the category $\sO$ for a hybrid quantum group} 

\makeatletter
\let\MakeUppercase\relax% disables author uppercase
\makeatother

\author{Quan Situ} 
\address{Yau Mathematical Sciences Center\\
Tsinghua University\\
Beijing 100084, P.~R.~China}
\email{stq19@tsinghua.org.cn, quan.situ@uca.fr}
\date{}
\begin{abstract} 
We establish an algebra isomorphism between the center of the category $\sO$ for a hybrid quantum group at a root of unity $\zeta$ and the cohomology of $\zeta$-fixed locus on affine Grassmannian. 
A deformed version of this isomorphism was established in the previous paper \cite{Situ1}. 
For the Steinberg block of $\sO$, we construct an abelian equivalence to the category of equivariant sheaves on the Springer resolution. 
\end{abstract}

\maketitle
\setcounter{tocdepth}{1} \tableofcontents 

%%%%%%%%%%%%%%%%%%%%%%%%%
\section{Introduction and notations} 
%%%%%%%%%%%%%%%%%%%%%%%%%
%---------------------------------------------------------------
\subsection{Main results} 
%---------------------------------------------------------------
Let $G$ be a complex connected and simply-connected semisimple algebraic group, with a Borel subgroup $B$ and a Cartan subgroup $T\subset B$. 
Let $\zeta\in \C$ be a primitive $l$-th root of unity, where $l$ is an odd integer greater than the Coxeter number and coprime to the determinant of the Cartan matrix for $G$. 
The \textit{hybrid} (or ``\textit{mixed}") quantum group $U^\hb_\zeta$ 
%(at root of unity $\zeta$) 
was firstly introduced by Gaitsgory \cite{Gai18} with the perspective of generalizing the Kazhdan--Lusztig equivalence. 
The positive part of $U^\hb_\zeta$ is given by the positive part of Lusztig's quantum group and the negative part of $U^\hb_\zeta$ is given by the one of De Concini--Kac quantum group. 
The category $\sO$ for $U^\hb_\zeta$ can be viewed as a quantum analogue of the BGG category $\sO$. 

In the work of Bezrukavnikov--Boixeda-Alvarez--Shan--Vasserot \cite{BBASV}, the category $\sO$ for $U^\hb_\zeta$ is used to study the representations of small quantum group and its center. 
In particular, they constructed a homomorphism from the cohomology of ring of $\zeta$-fixed locus on affine Grassmannian to the center $Z(\sO)$ of the category $\sO$ for $U^\hb_\zeta$. 
%, and predicted that it should be an isomorphism. 
In this paper we prove that it is an isomorphism.

%=================================
\subsubsection{Center of category $\sO$ for $U^\hb_\zeta$}
%=================================
Let $\check{G}$ be the Langlands dual group of $G$, with the corresponding Borel subgroup $\check{B}$ and the Cartan subgroup $\check{T}$. 
Denote the affine Grassmannian for $\check{G}$ by $\Gr=\check{G}\big(\C(\!(t)\!)\big)/\check{G}\big(\C[\![t]\!]\big)$. 
There is a $\Gm$-action on $\Gr$ by loop rotations. 
We consider the set of $\zeta$-fixed points $\Gr^\zeta$. 
There is a decomposition of this ind-variety, 
\begin{equation}\label{equ 1.00}
\Gr^\zeta=\bigsqcup_{\omega} \Fl^\omega, 
\end{equation}
where $\Fl^\omega$ are partial affine flag varieties of parahoric type $\omega$ associated with the loop group $\check{G}\big(\C(\!(t^l)\!)\big)$. 
Let $H^\bullet(\Gr^\zeta)$ be the singular cohomology of $\Gr^\zeta$ with complex coefficients. 
In \cite{BBASV} an algebra homomorphism $H^\bullet(\Gr^{\zeta})\rightarrow Z(\sO)$ is constructed, and it is extended to a homomorphism $\overline{\bb}: H^\bullet(\Gr^{\zeta})^\wedge \rightarrow Z(\sO)$ from a completion $H^\bullet(\Gr^{\zeta})^\wedge$ of $H^\bullet(\Gr^\zeta)$ in \cite{Situ1}. 
Our main result is the following. 

\begin{thmx}[Theorems \ref{prop 4.12}, \ref{thm 5.10} and \ref{thm 5.17}] \label{thm A} 
There is an algebra isomorphism 
$$\overline{\bb}: H^\bullet(\Gr^{\zeta})^\wedge\xs Z(\sO).$$ 
Under the map $\overline{\bb}$, the decomposition (\ref{equ 1.00}) is compatible with the block decomposition $\sO=\bigoplus\limits_\omega \sO^{[\omega]}$ 
labelled by the singular type $\omega$. 
In other words, $\overline{\bb}$ restricts to isomorphisms 
$$\overline{\bb}_\omega: H^\bullet(\Fl^\omega)^\wedge \xs Z(\sO^{[\omega]})$$ 
for each parahoric/singular type $\omega$. 
\end{thmx} 
Let $S=H_{\check{T}}^\bullet(\pt)_{\widehat{0}}$ be the completion of $H_{\check{T}}^\bullet(\pt)$ at the augmentation ideal. 
Recall that in \cite{Situ1} we established a deformed version of the isomorphism above, 
$$\bb: H_{\check{T}}^\bullet(\Gr^\zeta)^\wedge_S \xs Z(\sO_{S}),$$ 
where $H_{\check{T}}^\bullet(\Gr^\zeta)^\wedge_S$ is a completion of $H_{\check{T}}^\bullet(\Gr^\zeta)$, and $\sO_{S}$ is the deformation category $\sO$ for $U^\hb_\zeta$ defined over $S$. 
We have $\sO_{S}\otimes_S \C=\sO$. 
The isomorphisms $\bb$ and $\overline{\bb}$ satisfy $\bb\otimes_S \C=\overline{\bb}$. 
Note however that, although $H_{\check{T}}^\bullet(\Gr^\zeta)\otimes_{H_{\check{T}}^\bullet(\pt)}\C=H^\bullet(\Gr^\zeta)$, we do not have $Z(\sO_S)\otimes_S \C=Z(\sO)$ a priori, since taking center does not commute with specialization in general. 
To prove that $\overline{\bb}$ is an isomorphism, we need to study further properties of the category $\sO$. 

%=================================
\subsubsection{Equivalence for Steinberg block}
%=================================
We consider the Steinberg block $\sO^{[-\rho]}$ of $\sO$, i.e. the block containing Steinberg representations of $U^\hb_\zeta$. 
We establish the following 

\begin{thmx}[Theorem \ref{thm 4.7} and Corollary \ref{cor 4.8}] \label{thm B} 
There is an equivalence of abelian categories 
$$\sO^{[-\rho]} \xs \Coh^G(T^*(G/B)).$$ 
\end{thmx} 

We explain some ideas about this equivalence. 
In fact, Lusztig's version of Borel and De Concini--Kac's version of negative Borel are dual Hopf algebras (in the category of $X^*(T)$-graded vector spaces with braiding structure deformed by $q=\zeta$), so that $U^\hb_\zeta$ is realized as a Drinfeld double. 
So the analogue of $U^\hb_\zeta$ at $q=1$ corresponds to the semi-product 
$$U^\hb_1:=\k[B]\rtimes U\n,$$ 
where $\n$ is the Lie algebra of the unipotent radical $N$ of $B$. 
Recall that the analogue of Lusztig's quantum group $U_\zeta$ at $q=1$ is the enveloping algebra $U_1$ of $\Lie(G)$, and there is a quantum Frobenius map  $U_\zeta\twoheadrightarrow U_1$. 
Here for $U^\hb_\zeta$, we find a sub-quotient algebra which is isomorphic to $U^\hb_1$. 
Its construction is based on the quantum coadjoint action \cite{DeCKP92}, see \textsection \ref{subsect 4.1} for details. 
\iffalse 
To be more precise, the sub-quotient algebra is defined as follows. 
We consider the adjoint actions of elements in $U^+_\zeta$ on the Frobenius center $Z_\Fr^\leq$ of $\fU^\leq_\zeta$. 
The Chevalley generators $E_i$ commute with $Z_\Fr^\leq$. 
Their divided powers $E^{(l)}_i$ act on $Z_\Fr^\leq$ by the ``positive half" of the \textit{quantum coadjoint action} \cite{DeCKP92}, which coincide with the conjugations of the corresponding root vectors in $\n$ on $\k[B]$. 
Therefore, we let $\fU^\hb_\zeta$ be the subalgebra generated by $U^+_\zeta$ and $Z_\Fr^\leq$, then there is an isomorphism $\fU^\hb_\zeta/\langle E_i\rangle\xs U^\hb_1$ (in fact, up to a finite central extension). 
\fi 
Using this sub-quotient algebra and this isomorphism, we establish an equivalence between the Steinberg block $\sO^{[-\rho]}$ and the category $\sO$ for $U^\hb_1$ (denoted by $\sO_1$). 
Note that a similar equivalence for the Steinberg block (which is called the ``special block") of the category $\sO$ for $U_\zeta$ was established in \cite[Thm 3.11]{AM15}. 
Finally, note that any $U^\hb_1$-module in $\sO_1$ is naturally a coherent sheave on $B$ supported on $N$, and that the conditions of $X^*(T)$-grading and locally unipotent $U\n$-action amount to gives a $B$-equivariant structure on it. 
Using the two observations above, we obtain identifications 
$$\sO_1= \Coh^B(N)= \Coh^G(T^*(G/B)),$$ 
where the last equality is by induction, since $T^*(G/B)=G\times^B N$. 

%---------------------------------------------------------------
\subsection{Main steps of proof} 
%---------------------------------------------------------------
The proof of Theorem~\ref{thm A} is separated in blocks. 
We sketch the main steps as follows. 

%=================================
\subsubsection{Center of Steinberg block} 
%=================================
We firstly consider the Steinberg block $\sO^{[-\rho]}$. 
In this case, $\omega=-\rho$ is a maximal parahoric type, and in particular $\Fl^{-\rho}$ is the affine Grassmannian associated with the loop group $\check{G}\big(\C(\!(t^l)\!)\big)$. 

Using the equivalence in Theorem \ref{thm B}, we compute the center of $\sO^{[-\rho]}$ in two ways. 
The first one is algebraic, namely we identify $Z(\sO^{[-\rho]})$ with the center of the degree zero part of a completion of $\k[N]\rtimes U\n$. 
The second one is geometric, based on the equivalence between the derived category of $G\times \Gm$-equivariant sheaves on $T^*(G/B)$ and the triangulated category of mixed complexes on affine Grassmannian $\Fl^{-\rho}$, due to Arkhipov--Bezrukavnikov--Ginzburg \cite{ABG04}. 
\iffalse
Recall that this equivalence exchanges the grading shift of $\Gm$-action on the coherent side, with the cohomology shift composed with the half of the Tate twist on the constructible side. 
So there is a homomorphism from $H^\bullet(\Fl^{-\rho})$ to the graded center of $\Coh^{G\times \Gm}(T^*(G/B))$, which by forgetting the grading provides a map to the center of $\Coh^{G}(T^*(G/B))$. 
\fi
Using this equivalence we obtain a homomorphism 
\begin{equation}\label{equ 1.0} 
H^\bullet(\Fl^{-\rho})^\wedge \rightarrow Z(\Coh^{G}(T^*(G/B)))\xlongequal{\text{Thm \ref{thm B}}} Z(\sO^{[-\rho]}). 
\end{equation} 
We prove that (\ref{equ 1.0}) is an isomorphism, by using Ginzburg's description \cite{Gin95} of the cohomology of affine Grassmannian. 
We further show that (\ref{equ 1.0}) coincides with the map $\overline{\bb}_{-\rho}$, whose proof relies on a deformed version of the constructions above. 

%=================================
\subsubsection{Center of principal block} 
%=================================
Next, we treat the case of the principal block $\sO^{[0]}$. 
Now $\Fl^0$ is the affine flag variety associated with $\check{G}\big(\C(\!(t^l)\!)\big)$. 
We explain the idea of comparing $Z(\sO^{[0]})$ with the cohomology of $\Fl^0$ as follows. 

Recall in the classical case, computation of the center of the category $\sO^0_{\mathrm{BGG}}$ (the principal block of the category $\sO$ for $\Lie(G)$) follows from the three steps, see \cite{Soe90}: 
\begin{enumerate}
\item[(a)] construct an algebra homomorphism $H^\bullet(\check{G}/\check{B})\rightarrow Z(\sO^0_{\mathrm{BGG}})$; 
\item[(b)] consider the anti-dominant projective module $Q$, and show the composition 
\begin{equation}\label{equ 1.a} 
H^\bullet(\check{G}/\check{B})\rightarrow Z(\sO^0_{\mathrm{BGG}})\rightarrow \End(Q)
\end{equation}
is an isomorphism; 
\item[(c)] prove that the last map in (\ref{equ 1.a}) is an injection. 
\end{enumerate} 
Moreover, $Q$ can be obtained by applying translation functor to the Verma module of highest weight $-\rho$. 

Now turn to our case. 
The step (a) is fulfilled by the map $\overline{\bb}_0$. 
For step (b), we consider a family of projective covers $\{Q^{\leq \mu}\}_{\mu\geq 0}$ of the Verma module $M(-\rho)$ in different truncated categories of $\sO^{[-\rho]}$, and when $\mu=0$ we have $Q^{\leq 0}=M(-\rho)$. 
Their endomorphism algebra form a limit $\Lim{\mu\geq 0} \End(Q^{\leq\mu})$, such that the natural restrictions $Z(\sO^{[-\rho]})\rightarrow \End(Q^{\leq\mu})$ yield a homomorphism $Z(\sO^{[-\rho]})\rightarrow \Lim{\mu\geq 0} \End(Q^{\leq\mu})$. 
By using the algebraic description of $Z(\sO^{[-\rho]})$, we show that it yields an isomorphism 
\begin{equation}\label{equ 1.b} 
Z(\sO^{[-\rho]})\xs Z\big(\Lim{\mu\geq 0}\End(Q^{\leq\mu})\big). 
\end{equation}
We then apply the the translation functor $\sfT^0_{-\rho}:\sO^{[-\rho]}\rightarrow \sO^{[0]}$ to the family $\{Q^{\leq \mu}\}_{\mu\geq 0}$, and similarly obtain an algebra homomorphism 
\begin{equation}\label{equ 1.c} 
Z(\sO^{[0]})\rightarrow Z\big(\Lim{\mu\geq 0} \End(\sfT^0_{-\rho}Q^{\leq \mu})\big). 
\end{equation} 
Moreover, $\sfT_{-\rho}^0M(-\rho)$ is the anti-dominant projective module in a sub-quotient category of $\sO$, whose endomorphism ring is isomorphic to $H^\bullet(\check{G}/\check{B})$ as in the classical case. 
We show that there is a commutative diagram 
$$\begin{tikzcd}
H^\bullet(\Fl^{-\rho})^\wedge \arrow[r]\arrow[d,"(\ref{equ 1.b})\circ \overline{\bb}_{-\rho}","\simeq"'] & 
H^\bullet(\Fl^{0})^\wedge \arrow[r]\arrow[d,"(\ref{equ 1.c})\circ\overline{\bb}_{0}"] & 
H^\bullet(\check{G}/\check{B})\arrow[d,"\simeq"] \\ 
Z\big(\Lim{\mu\geq 0}\End(Q^{\leq\mu})\big)\arrow[r,"\sfT^0_{-\rho}"] & 
Z\big(\Lim{\mu\geq 0} \End(\sfT^0_{-\rho}Q^{\leq \mu})\big)\arrow[r] &\End(\sfT_{-\rho}^0M(-\rho)),
\end{tikzcd}$$
where the upper horizontal maps are induced by the fibration $\check{G}/\check{B}\rightarrow \Fl^0\rightarrow \Fl^{-\rho}$. 
In fact, both the horizontal maps above are short exact sequence of algebras, in the sense that the third algebra is the quotient of the second one by the ideal generated by the augmentation ideal of the first one. 
It implies the bijectivity of the middle vertical map. 
For step (c), we show (\ref{equ 1.c}) is injective. 

%=================================
\subsubsection{Center of singular blocks} 
%=================================
Finally, we deduce the case of arbitrary block $\sO^{[\omega]}$ as follows. 
Consider the translation functors $\sfT^\omega_0:\sO^{[0]}\rightarrow \sO^{[\omega]}$ and $\sfT_\omega^0:\sO^{[0]}\rightarrow \sO^{[\omega]}$, which form a biadjoint pair. 
As recalled in the Appendix \ref{app D}, the traces of such functors (in the sense of \cite{B90}) provide linear maps between the centers 
$$\tr_{\sfT^\omega_0}: Z(\sO^{[\omega]})\rightarrow Z(\sO^{[0]}),\quad 
\tr_{\sfT_\omega^0}: Z(\sO^{[0]})\rightarrow Z(\sO^{[\omega]}).$$ 
We show that there exist compatible linear maps between $H^\bullet(\Fl^{\omega})^\wedge$ and $H^\bullet(\Fl)^\wedge$ fitting into a commutative diagram 
$$\begin{tikzcd}
H^\bullet(\Fl^{\omega})^\wedge \arrow[d,"\overline{\bb}_\omega"']\arrow[r] 
& H^\bullet(\Fl)^\wedge \arrow[d,"\overline{\bb}_0"]\arrow[r] & H^\bullet(\Fl^{\omega})^\wedge \arrow[d,"\overline{\bb}_\omega"] \\ 
Z(\sO^{[\omega]})\arrow[r,"\tr_{\sfT^\omega_0}"]& Z(\sO^{[0]})\arrow[r,"\tr_{\sfT_\omega^0}"] & Z(\sO^{[\omega]}).
\end{tikzcd}$$
We show that the horizontal compositions are linear isomorphisms, and hence the bijectivity of $\overline{\bb}_\omega$ follows from the one for $\overline{\bb}_0$. 

%---------------------------------------------------------------
\subsection{Arrangement of the paper} 
%---------------------------------------------------------------
In Section~\ref{sect 2} we recall some facts about quantum groups and basic properties of their category $\sO$ studied in \cite{Situ1}. 

Section~\ref{sect 3} is devoted to the study of the Steinberg block $\sO^{[-\rho]}$. 
In \textsection\ref{subsect 4.1} we construct a sub-quotient algebra of $U^\hb_\zeta$ that is isomorphic to a central extension of $U^\hb_1$; in \textsection\ref{subsect 3.2} we prove the equivalence in Theorem \ref{thm B}; in \textsection\ref{subsect 3.3} we study the center $Z(\sO^{[-\rho]})$ and prove Theorem \ref{thm A} in this case, where the algebraic interpretation of $Z(\sO^{[-\rho]})$ will be discussed in \textsection\ref{subsect 3.3.4}. 

We study the principal block $\sO^{[0]}$ in Section~\ref{sect 4}. 
In \textsection\ref{subsect 4.0} we recall the translation functors; in \textsection\ref{subsect 4.2} we construct a new truncation of category $\sO$ that is useful to compare the center of $\sO^{[-\rho]}$ and $\sO^{[0]}$; in \textsection\ref{subsect 4.3} we prove Theorem \ref{thm A} for principal block. 

In Section~\ref{sect 5} we consider an arbitrary block $\sO^{[\omega]}$, where in \textsection\ref{subsect 5.1}-\ref{subsect 5.3.4} we study the trace of translation functors between $\sO^{[0]}$ and $\sO^{[\omega]}$, and in \textsection\ref{subsect 5.3} we complete the rest of proof of Theorem \ref{thm A}. 

In Appendix~\ref{app A} we discuss some general facts on center of a category. 
Appendix~\ref{app D} is about the trace of translation functors, where in \textsection\ref{subsect D.1} we prove the Bernstein's formula for quantum groups, and in \textsection\ref{subsect B.2} we discuss the the compatibility of trace map and pushforward of cohomology. 

\subsection{Notations and conventions} 
\subsubsection{Notations} 
%Let $\k$ be an algebraically closed field of characteristic $0$, and without loss of generality, we assume $\k=\C$. 
For a complex variety $X$ with an action of a complex linear group $G$, we denote by $H^\bullet_G(X)$ the $G$-equivariant cohomology with coefficients in $\k$. 
For a Lie algebra $\g$ over $\k$, we denote by $U\g$ its enveloping algebra. 

\subsubsection{Root data} 
Let $G$ be a complex connected and simply-connected semisimple algebraic group, with a Borel subgroup $B$ and a maximal torus $T$ contained in $B$. 
Let $B^-$ be the opposite Borel subgroup, and let $N$, $N^-$ be the unipotent radical of $B$, $B^-$. 
Denote their Lie algebras by 
$$\g=\Lie(G),\quad \b=\Lie(B), \quad \n=\Lie(N),\quad \n^-=\Lie(N^-),\quad \t=\Lie(T).$$ 
Let $W=N_G(T)/T$ be the Weyl group for $G$, with longest element $w_0$. 
Let $\check{G}$ be the Langlands dual group of $G$, with the dual torus $\check{T}$ and the corresponding Borel subgroup $\check{B}$. 

Let $(X^*(T),X_*(T), \Phi, \check{\Phi})$ be the root datum associated with $G$. 
Let $\Phi^+$ and $\Sigma=\{\alpha_i\}_{i\in \I}$ be the subsets of $\Phi$ consisting of positive roots and simple roots. 
We set $\check{\Sigma}=\{\check{\alpha}_i\}_{i\in \I}$. 
We abbreviate $\Lambda:= X^*(T)$ and $\check{\Lambda}:= X_*(T)$, and let $\langle-,-\rangle:\check{\Lambda} \times \Lambda \rightarrow \Z$ be the canonical pairing. 
Let $\rQ\subset \Lambda$ be the root lattices. 
There is a partial order $\leq$ on $\Lambda$ given by $\lambda\leq \mu$ if $\lambda-\mu\in \rQ_{\leq}:=\Z_{\leq}\Sigma$. 
Recall that the fundamental group of $\check{G}$ is $\pi_1:=\pi_1(\check{G})=\Lambda/\rQ$.  
Let $a_{ij}:=\langle \check{\alpha}_i, \alpha_j\rangle$ be the $(i,j)$-th entry of the Cartan matrix of $G$. 
Let $(d_i)_{i\in \I}\in \N^\I$ be a tuple of relatively prime positive integers such that $(d_ia_{ij})_{i,j\in \I}$ is symmetry and positive definite. 
It defines a pairing $(-,-):\rQ \times \rQ \rightarrow \Z$ by $(\alpha_i,\alpha_j):=d_ia_{ij}$ and extents to 
$$(-,-):\Lambda \times \Lambda \rightarrow \frac{1}{e} \Z, \quad e:=|\pi_1|.$$ 
Let $\{e_{\alpha}\}_{\alpha\in \Phi^+}$ and $\{f_{\alpha}\}_{\alpha\in \Phi^+}$ be the Chevalley basis for $\n$ and $\n^-$. 

\subsubsection{Rings} 
Let $h$ be the Coxeter number of $G$. 
Let $l\geq h$ be an odd positive integer which is prime to $e$, and to $3$ if $G$ contains a component of type $G_2$. 
Let $\zeta_e\in \k$ be a primitive $l$-th root of unity, and let $\zeta=(\zeta_e)^e$. 
Let $q$ be a formal variable, and set $q_e=q^{\frac{1}{e}}$. 
We set $\bA=\k[q_e^{\pm1}]$ and $\bF=\k(q_e)$. 

We set $S'=\k[\t]$. 
Consider the $W$-invariant isomorphism $\t^*\xs \t$ such that $\alpha_i\mapsto d_i \check{\alpha}_i$ for each $i\in \rI$. 
It yields an isomorphism $S'\xs \k[\t^*] =H_{\check{T}}^\bullet(\pt)$. 
Let $S=\k[\t]_{\widehat{0}}$ be the completion at $0\in \t$, and $\k[T]_{\widehat{1}}$ be the completion at $1\in T$. 
We have $S=\k[T]_{\widehat{1}}$ via the exponential map $\exp:\t \rightarrow T$. 

\subsubsection{Affine Weyl group} 
Let $W_\af:= W\ltimes \rQ$ and $W_\ex:= W\ltimes \Lambda \simeq W_\af \rtimes \pi_1$ be the affine Weyl group and the extended Weyl group. 
Denote by $\tau_\mu\in W_\ex$ the translation by $\mu\in \Lambda$. 
Denote by $\ell(-):W_\ex \rightarrow \Z_{\geq 0}$ the length function. 
For any subset $J$ of affine simple roots, we denote by $W_J$ the parabolic subgroup in $W_\af$ generated by the reflections associated with $J$, and we identify 
$$W^J_\af=W_\af/W_J=\{ x\in W_\af|\ \ell(x)\leq \ell(y),\ \forall y\in xW_J\}.$$ 
Let $W^J_\ex=W_\ex/W_J=\pi_1\times W^J_\af$. 

Let $W_{l,\af}:= W\ltimes l\rQ$ and $W_{l,\ex}:= W\ltimes l\Lambda$ be the $l$-affine Weyl group and the $l$-extended Weyl group. 
There is a shifted action of $W_{l,\ex}$ on $\Lambda$, given by $w\bullet \lambda:=w(\lambda+\rho)-\rho$ for any $w\in W_{l,\ex}$ and $\lambda\in \Lambda$, where $\rho=\frac{1}{2}\sum_{\alpha\in \Phi^+}\alpha$. 
Set 
$$\Xi_\sc:= \Lambda/ (W_{l,\af}, \bullet), \quad \Xi:=\Xi_\sc/\pi_1= \Lambda/ (W_{l,\ex}, \bullet) ,$$ 
where $\bullet$ represents the $\bullet$-action above. 
We may identify 
$$\Xi_\sc=\{\omega\in \Lambda\ |\ 0\leq\langle \omega+\rho, \check{\alpha} \rangle\leq l,\ \forall \alpha\in \Phi^+ \},$$ 
since any coset in $\Xi_\sc$ is uniquely determined by an element in the RHS. 
For $\omega\in \Xi_\sc$, we denote by $W_{l,\omega}=\Stab_{(W_{l,\af},\bullet)}(\omega)$, and set $W^\omega_{l,\af}=W_{\l,\af}/W_{l,\omega}$, $W^\omega_{l,\ex}=W_{\l,\ex}/W_{l,\omega}$. 

\subsubsection{Affine flag varieties} 
Let $\check{G}(\!(t)\!)$ and $\check{G}[\![t]\!]$ be the loop group and the arc group of $\check{G}$. 
For any parabolic subgroup $W_J\subset W_\ex$, we denote by $P^J$ the corresponding standard parahoric subgroup of $\check{G}(\!(t)\!)$. 
The partial affine flag variety of type $J$ is the fpqc quotient $\Fl^J=\check{G}(\!(t)\!)/P^J$. 
Recall the $\check{T}$-fixed point set 
$$(\Fl^J)^{\check{T}}=\{\delta_x^J|\ x\in W^J_\ex\}= W^J_\ex, \quad \delta_x^J:=xP^J/P^J .$$ 
If $J=\emptyset$, then $\sI=P^\emptyset$ is the Iwahoric subgroup and $\Fl=\Fl^\emptyset$ is the affine flag variety. 
If $J=\check{\Sigma}$, we have natural identifications 
$$W^{\check{\Sigma}}_\af=\rQ, \quad W^{\check{\Sigma}}_\ex=\Lambda ,\quad P^{\check{\Sigma}}=\check{G}[\![t]\!],$$ 
and $\Gr=\Fl^{\check{\Sigma}}$ is the affine Grassmannian. 
We abbreviate $\delta_x=\delta^J_x$ if without confusion, and write $\delta_\mu=\delta_{\tau_\mu}$ for $\mu\in \Lambda$. 
Recall that $\pi_0(\Fl^J)=\pi_1$. 
Denote by $\Fl^{J,\circ}$ the connected component for $\Fl^J$ containing $P^J/P^J$. 
Then we have $(\Fl^{J,\circ})^{\check{T}}=\{\delta_x^J|\ x\in W^J_\af\}$ and an isomorphism 
\begin{equation}\label{equ 1.6pi_1}
\Fl^J \simeq \pi_1 \times \Fl^{J,\circ}. 
\end{equation}
There is a $\Gm$-action on $\C(\!(t)\!)$ by rotating $t$, which induces a $\Gm$-action on each $\Fl^J$. 

Let $\Fl_l^J$ be the partial affine flag variety given by replacing $t$ by $t^l$. 
We may regard $\Fl^J=\Fl^J_l$ if without ambiguity. 
For $\omega\in \Xi_\sc$, we let $J_\omega$ be the subset of $l$-affine simple roots corresponding to $W_{l,\omega}$. 
We abbreviate $\Fl^{\omega}=\Fl^{J_\omega}_l$ and $\Fl^{\omega,\circ}=\Fl^{J_\omega,\circ}_l$. 
Consider the fixed locus $\Gr^\zeta$ of $\zeta\in \Gm$ on $\Gr$. 
By \cite[\textsection 4]{RW22}, there is an isomorphism 
\begin{equation}\label{equ 1.1} 
\bigsqcup_{\omega\in \Xi} \Fl^{\omega}=\Gr^\zeta, \quad g\delta^{J_\omega}_e\mapsto g\delta_{\omega+\rho}, \quad \forall g\in \check{G}(\!(t^l)\!). 
\end{equation} 
At the level of $\check{T}$-fixed points, the isomorphism (\ref{equ 1.1}) is compatible with the isomorphism $\bigsqcup_{\omega\in \Xi} W_{l,\ex}^\omega=\Lambda$, $xW_l^\omega \mapsto x(\omega+\rho)$. 
Combining \eqref{equ 1.6pi_1} with \eqref{equ 1.1}, we obtain a decomposition 
\begin{equation}\label{equ 1.8new}
\bigsqcup_{\omega\in \Xi_\sc} \Fl^{\omega,\circ}=\Gr^\zeta
\end{equation}

Let $T'=\check{T}\times \Gm$, $\check{T}$, $\Gm$ or the trivial group, the cohomology $H^\bullet_{T'}(\Fl^J)$ is freely generated as a $H^\bullet_{T'}(\pt)$-module by the fundamental classes $[\Fl^{J,x}]_{T'}$ of the finite codimensional Schubert varieties $\Fl^{J,x}$ labelled by $x\in W^J_\ex$. 
For any $H^\bullet_{T'}(\pt)$-algebra $R$, we denote by $H^\bullet_{T'}(\Fl^J)^\wedge_R$ the space of formal series of $[\Fl^{J,x}]_{T'}$ with coefficients in $R$. 
We will drop the subscript $R$ if $R=H^\bullet_{T'}(\pt)$. 

\subsubsection{Conventions} 
Categories and functors are additive and $\k$-linear. 
A \textit{block} in a category means an additive full subcategory that is a direct summand (not necessarily indecomposable). 
We will write ``$\Lim{}$" and ``$\cLim{}$" for the limits and colimits in a category (if exist). 

Let $\sC$ be an $R$-linear category. 
The \textit{center} $Z(\sC)$ of $\sC$ is the ring of $R$-linear endomorphism of the identity functor of $\sC$, i.e. 
$$Z(\sC)=\{\big(z_M\in \End_{\sC}(M)\big)_{M\in \sC}|\ f\circ z_{M_1}=z_{M_2}\circ f,\ \forall M_1,M_2\in \sC, \forall f\in \Hom_{\sC}(M_1,M_2)\}.$$
We may abbreviate $\Hom(M_1,M_2)=\Hom_{\sC}(M_1,M_2)$ if there is no ambiguity. 
For a set $\sX$, we denote by $\Fun(\sX,R)$ the space of $R$-valued functions on $\sX$, which is naturally endowed with an $R$-algebra structure. 

For a scheme $X$, we denote by $\sO_{X}$ its structure sheaf and by $\sT_{X}$ its tangent sheaf if $X$ is smooth. 
For any algebraic group $K$, we will denote by $\rep(K)$ the category of finite dimensional rational representations of $K$. 

\subsection{Acknowledgments} 
The author sincerely thanks his supervisor Professor Peng Shan for suggesting this problem and her patient guidance. 
Without her help this article could not be finished by the author alone. 
The author also thanks Professor Nicolai Reshetikhin for enlightening discussions. 
The author thanks Tamas Hausel for pointing out a mistake on Harish-Chandra center in an earlier version.

%%%%%%%%%%%%%%%%%%%%%%%%%
\section{Quantum groups and their representations}\label{sect 2} 
%%%%%%%%%%%%%%%%%%%%%%%%%
%---------------------------------------------------------------
\subsection{Quantum groups} 
%---------------------------------------------------------------
The quantum group $\mathscr{U}_q$ associated with $G$ is the $\bF$-algebra generated by the standard generators $E_i, F_i, K_\lambda \ (i\in \I, \lambda\in \Lambda)$. 
We abbreviate $K_i=K_{\alpha_i}$. 

The Lusztig's integral form $U_q$ is a $\bA$-subalgebra generated by $E_i^{(n)}, F_i^{(n)}, K_\lambda$; the De Concini--Kac's integral form $\fU_q$ is generated by $E_i, F_i, K_\lambda$. 
We define the \textit{hybrid quantum group} $U^\hb_q$ to be the $\bA$-subalgebra generated by $E_i^{(n)}, F_i, K_\lambda$. 
There are algebra inclusions 
\begin{equation}\label{equ QGs} 
\fU_q \subset U_q^{\hb} \subset U_q. 
\end{equation} 
We denote the $\bF$-subalgebras $\mathscr{U}^+_q:=\langle E_i\rangle_{i\in \I}$, $\mathscr{U}^-_q:=\langle F_i\rangle_{i\in \I}$ and $\mathscr{U}^0_q:=\langle K_\lambda\rangle_{\lambda\in \Lambda}$, and denote by $\fU_q^+, \fU_q^-, \fU^0_q$ and $U_q^+, U_q^-, U^0_q$ their intersections with $\fU_q$ and $U_q$. 
There is a triangular decomposition 
$$U_q^{\hb}=\fU_q^-\otimes \fU_q^0 \otimes U_q^+.$$ 
We will identify $\mathscr{U}^0_q=\bF[\Lambda]=\bF[T]$. 

For any integral form $A_q$ above, we let the $\k$-algebra $A_\zeta:=A_q\otimes_{\bA} \k$ be the specialization at $q_e= \zeta_e$. 
The specialization yields a chain of maps 
\begin{equation}\label{equ QGspe} 
\fU_\zeta \rightarrow U_\zeta^{\hb} \rightarrow U_\zeta .
\end{equation} 
Let $u_\zeta$ be the \textit{small quantum group} in $U_\zeta$, which coincides with the image of $\fU_\zeta\rightarrow U_\zeta$. 
Denote by $\fU^b_\zeta$ the image of $\fU_\zeta \rightarrow U_\zeta^{\hb}$. 
Then there are triangular decompositions 
$$u_\zeta=u_\zeta^-\otimes u_\zeta^0 \otimes u_\zeta^+, \quad \fU^b_\zeta= \fU^-_\zeta\otimes \fU^0_\zeta \otimes u_\zeta^+.$$ 

Fix a convex order on $\Phi^+$, let $E_\beta\in U_q^+$ and $F_\beta\in U_q^-$ be the \textit{root vectors} associated with $\beta\in \Phi^+$. 
Lusztig \cite[\textsection 8]{Lus90} defined the \textit{quantum Frobenius homomorphism} 
\begin{equation}
\Fr: U_\zeta \rightarrow U\g, \quad 
\text{by}\quad E_\beta^{(n)},\ F_\beta^{(n)}\mapsto 
\begin{cases} 
		\frac{e_\beta^{n/l}}{(n/l)!},\ \frac{f_\beta^{n/l}}{(n/l)!} & \text{if}\ l|n , \\ 
		 0 & \text{if else} 
\end{cases}, \quad K_\lambda\mapsto 1. 
\end{equation} 
It restricts to the homomorphisms $\Fr : U_\zeta^+ \rightarrow U\n$ and $\Fr: U_\zeta^-\rightarrow U\n^-$. 

%---------------------------------------------------------------
\subsection{Centers of quantum groups} 
%---------------------------------------------------------------
%=================================
\subsubsection{Harish-Chandra center} 
%=================================
We set $\mathscr{U}_q^{0,ev}:=\bF\langle K_{2\lambda}\rangle_{\lambda\in \Lambda}$. 
There is an algebra isomorphism 
$$Z(\mathscr{U}_q) \xs (\mathscr{U}_q^{0,ev})^{(W,\bullet)}$$ 
given by projecting $Z(\mathscr{U}_q)$ to $\mathscr{U}_q^0$ under the triangular decomposition, where $(W,\bullet)$ is the shifted action by $w\bullet K_\lambda =q^{(w\lambda -\lambda, \rho)}K_{w\lambda}$, for any $w\in W$, $\lambda\in\Lambda$. 
By the natural identifications 
$$(\mathscr{U}_q^{0,ev})^{(W,\bullet)}= \bF[T]^{(W,\bullet)}= \bF[T/W],\quad f(K_{2\lambda})\mapsto f(K_\lambda)\mapsto f(q^{-2(\lambda,\rho)}K_\lambda),$$ 
we have an isomorphism $\hc:\ Z(\mathscr{U}_q) \xs \bF[T/W]$. 
The centers of the integral forms are given by 
$$Z(U_q)= Z(\mathscr{U}_q) \cap U_q, \quad Z(\fU_q)= Z(\mathscr{U}_q) \cap \fU_q.$$ 
The isomorphism $\hc$ induces isomorphisms 
$$ \hc: \ Z(\fU_q) \xs \bA[T/W] ,\quad 
\hc: \ Z_\HC:= Z(\fU_q)/(q_e-\zeta_e)Z(\fU_q) \xs \k[T/W],$$ 
where $Z_\HC$ is called the \textit{Harish-Chandra center} of $\fU_\zeta$. 
The natural inclusion $Z(\fU_q)\subset Z(U_q)$ induces homomorphisms 
$$\hc^{-1}: \bA[T/W]\rightarrow Z(U_q), \quad 
\hc^{-1}: \C[T/W]\rightarrow Z(U_\zeta).$$ 

%=================================
\subsubsection{Frobenius center}\label{subset Frocen} 
%=================================
The \textit{Frobenius center} of $\fU_\zeta$ is the $\k$-subalgebra 
$$Z_\Fr:=\langle K^l_{\lambda}, F^l_\beta , E^l_\beta \rangle_{\lambda\in \Lambda, \beta\in \Phi^+}. $$ 
We abbreviate $Z_\Fr^{\flat}:= Z_\Fr \cap \fU_\zeta^{\flat}$ for $\flat=-,+,0,\leq$ and $\geq$. 
By \cite[\textsection 0]{DeCKP92}, there are isomorphisms of $\k$-algebras, $Z_\Fr^{-}\xs \k[N^-]$, $Z_\Fr^{+}\xs \k[N]$ and $Z_\Fr^{0}\xs \k[T]$, which give an isomorphism  
\begin{equation}\label{equ 2.7} 
\Spec Z_\Fr \xs G^*, 
\end{equation} 
where $G^*=N^-\times T \times N$ is the Poisson dual group of $G$. 
We have the following isomorphisms, see \cite[p128]{DeCP92}, 
\begin{equation}\label{equ 2.-1} 
Z(\fU_\zeta)=Z_\Fr\otimes_{Z_\Fr\cap Z_\HC}Z_\HC=\k[G^*\times_{T/W}T/W], 
\end{equation} 
where $G^*\rightarrow T/W$ is by sending $(n_1,t,n_2)\in N^-\times T \times N$ to the $W$-orbit of the semisimple part of $n_1t^2n_2^{-1}$, and $T/W\rightarrow T/W$ is by $W(t)\mapsto W(t^l)$ for any $t\in T$.

Note that $\fU_\zeta^b$ coincides with the algebra $\fU_\zeta\otimes_{\k[N]} \k$ by evaluating $\k[N]$ at $1\in N$. 
We may view $Z_\Fr^{\leq}$ as a subalgebra in $U_\zeta^\hb$, which is central in $\fU_\zeta^b\subset U_\zeta^\hb$. 

%---------------------------------------------------------------
\subsection{Representations}\label{subsect 2.3} 
%---------------------------------------------------------------
There is a $\Lambda$-action on $\mathscr{U}_q^0$ such that any $\mu\in \Lambda$ corresponds to the $\bF$-algebra automorphism  
$$\tau_\mu: K_\lambda \mapsto q^{(\mu,\lambda)}K_\lambda, \quad \forall\lambda\in\Lambda.$$ 
The $\Lambda$-action on $\mathscr{U}_q^0$ preserves the integral forms $\fU^0_q$, $U^0_q$, and specializes to an action on $\fU^0_\zeta$, $U^0_\zeta$. 
Let $A^0$ be a finitely generated $\k$-subalgebra of $\fU_\zeta^0$ or $U_\zeta^0$ that is preserved under the $\Lambda$-action. 
Let $A=\bigoplus_{\lambda\in \rQ}A_\lambda$ be a $\rQ$-graded $\k$-algebra with $A^0\subset A_0$ such that 
$$fm=m\tau_\lambda(f),\quad \forall f\in A^0,\ \forall m\in A_\lambda,$$ 
Let $A^-,A^0,A^+$ be subalgebras of $A$ with triangular decomposition $A= A^-\otimes A^0\otimes A^+$ and satisfy further conditions as in \cite[\textsection 2.3]{Situ1}. 
We abbreviate $A^\leq=A^-A^0$ and $A^\geq=A^+A^0$. 
A \textit{deformation ring} $R$ for $A$ is a commutative and Noetherian $A^0$-algebra. 
Let $\pi: A^0\rightarrow R$ be the structure map. 

%=================================
\subsubsection{Module categories} 
%=================================
We define $A\Mod_R^{\Lambda}$ to be the category consisting of $A\otimes R$-modules $M$ endowed with a decomposition $M=\bigoplus\limits_{\mu\in \Lambda}M_\mu$ of $R$-modules (called the \textit{weight spaces}), such that $M_\mu$ is killed by the elements in $A\otimes R$ of the form 
$$f\otimes 1-1\otimes \pi({\tau_\mu}(f)), \quad f\in A^0.$$ 
Let $A\mod_R^{\Lambda}$ be the full subcategory of $A\Mod_R^{\Lambda}$ consisting of finitely generated $A\otimes R$-modules whose weight spaces are finitely generated $R$-modules. 
Define the \textit{category $\sO$ for $A$} to be the full subcategory $\sO^A_R$ of $A\mod_R^{\Lambda}$ of modules that are locally unipotent for the action of $A^+$. 
It is an abelian subcategory of $A\Mod_R^{\Lambda}$. 

Define the \textit{Verma module} 
$$M^A(\lambda)_R:= A \otimes_{A^{\geq}} R_\lambda \ \in \sO^A_R$$ 
where $R_\lambda$ is an $A^{\geq}$-module via $A^{\geq}\twoheadrightarrow A^0\xrightarrow{\pi\circ\tau_{\lambda}}R_\lambda$. 
If $R=\F$ is a field, $M^A(\lambda)_\F$ has a unique simple quotient $L^A(\lambda)_\F$. 

\subsubsection{$\pi_1$-grading} 
Since $A$ is $\rQ$-graded, any $\Lambda$-graded module of $A$ decomposes into submodules whose weights belong to the same class in $\pi_1=\Lambda/\rQ$. 
It yields a block decomposition 
\begin{equation}\label{equ 3.2} 
\sO^A_{R}=\bigoplus_{\gamma\in \pi_1}\sO^{A,\gamma}_{R}. 
\end{equation} 

%=================================
\subsubsection{Truncations and base changes} 
\label{subsect 2.3.3new}
%=================================
For any $\nu\in \Lambda$, there is a \textit{truncated category} $A\Mod_R^{\Lambda,\leq \nu}$ consists of the module $M$ in $A\Mod_R^{\Lambda}$ such that $M_\mu=0$ unless $\mu \leq \nu$. 
The category $\sO^{A,\leq \nu}_R:=\sO^{A}_R\cap A\Mod_R^{\Lambda,\leq \nu}$ always admits enough projective objects, in contrast to $\sO^{A}_R$. 

The \textit{truncation functor} 
$$\tau^{\leq \nu}: A\Mod_R^{\Lambda} \rightarrow A\Mod_R^{\Lambda,\leq \nu} ,\quad M \mapsto M \big/ A. \big( \bigoplus_{\mu\nleq \nu} M_\mu \big),$$ 
is by taking the maximal quotient in $A\Mod_R^{\Lambda,\leq \nu}$. 
Note that $\tau^{\leq \nu}$ is left adjoint to the natural inclusion $A\Mod_R^{\Lambda,\leq \nu} \hookrightarrow A\Mod_R^{\Lambda}$. 
We denote the counit by $\epsilon^{\leq \nu}:\id \rightarrow \tau^{\leq \nu}$. 

Let $R'$ be a commutative Noetherian $R$-algebra. 
There is a base change functor $-\otimes_R R': \sO^A_R \rightarrow \sO^A_{R'}$. 
Denote by $\sP^{A,\leq \nu}_R$ the full subcategory of projective modules in $\sO^{A,\leq \nu}_R$.  
The base change functor yields a natural equivalence, see \cite[Prop 2.4]{Fie03}, 
\begin{equation}\label{equ 2.10} 
\sP^{A,\leq \nu}_R\otimes_R R'\xs \sP^{A,\leq \nu}_{R'}. 
\end{equation}
By \cite[(2.8)]{Situ1}, it induces an $R$-algebra homomorphism 
\begin{equation}\label{equ 2.8} 
-\otimes_R R':\ Z(\sO^A_R) \rightarrow Z(\sO^A_{R'}). 
\end{equation} 

%---------------------------------------------------------------
\subsection{Category $\sO$ for hybrid quantum group} 
%---------------------------------------------------------------
We view $S$ as a deformation ring for $U^\hb_\zeta$ by the inclusion $\fU_\zeta^0=\k[\Lambda]=\k[T]\subset S$. 
Let $R$ be a commutative $S$-algebra that is a local Noetherian domain with residue field $\F$. 
For $A=U^\hb_\zeta$, we abbreviate $E(\lambda)_\F=L^A(\lambda)_\F$, $M(\lambda)_R=M^A(\lambda)_R$ and $\sO_{R}=\sO^A_R$. 
We denote by $Q(\mu)^{\leq \nu}_R$ the projective cover for $E(\mu)_\F$ in $\sO^{\leq \nu}_{R}$. 
In this subsection, we recall some basic properties for the category $\sO_{R}$ shown in \cite[\textsection 3]{Situ1}. 

%=================================
\subsubsection{Projective and simple modules} 
%=================================
%Suppose $G$ is simply connected. 
Denote the set of $l$-restricted dominant weights by 
$$\Lambda_l^+=\{\mu\in \Lambda|\ 0\leq \langle \mu, \check{\alpha}_i \rangle <l,\ \forall i\in \I \}.$$ 
Recall that for any $\lambda^0\in \Lambda_l^+$, the simple module $L(\lambda^0)_\k$ of $u_\zeta$ of highest weight $\lambda^0$ can be extend to a $U_\zeta$-module. 
We view $L(\lambda^0)_\k$ as a $U^\hb_\zeta$-module via $U^\hb_\zeta\rightarrow U_\zeta$. 
\begin{lem}[{\cite[Lem 3.1]{Situ1}}]\label{lem simple} 
We have $E(\lambda)_\k=L(\lambda^0)_\k \otimes \k_{l\lambda^1}$ for any $\lambda\in \Lambda$, where $\lambda=\lambda^0+l\lambda^1$ is the unique decomposition such that $\lambda^0\in \Lambda_l^+$. 
\end{lem} 

In \cite[\textsection 3.3.9]{BBASV}, the authors define a module $Q(\lambda)_R$ in $U^\hb_{\zeta}\Mod^{\Lambda}_R$ by 
\begin{equation}\label{equ 3.1} 
Q(\lambda)_R:= U^\hb_\zeta\otimes_{\fU^b_\zeta} P^b(\lambda)_R , \quad \lambda\in \Lambda, 
\end{equation} 
where $P^b(\lambda)_R$ is the projective cover for the simple module $L^b(\lambda)_\F$ of highest weight $\lambda$ in $\fU^b_\zeta\mod^{\Lambda}_R$. 
If $\lambda\in -\rho+l\Lambda$, we have $P^b(\lambda)_R=M(\lambda)_R$ as $\fU^b_\zeta$-modules, hence 
\begin{equation}\label{equ 3.20} 
\begin{aligned} 
Q(\lambda)_R&=U^\hb_\zeta\otimes_{\fU^b_\zeta} M(\lambda)_R
=U^\hb_\zeta \otimes_{U^{\hb,\geq}_\zeta} (U^{\hb,\geq}_\zeta\otimes_{\fU^{b,\geq}_\zeta} R_{\lambda})\\ 
&=U^\hb_\zeta \otimes_{U^{\hb,\geq}_\zeta} (U\n\otimes R_{\lambda}) , 
\end{aligned} 
\end{equation} 
where $U\n\otimes R_{\lambda}$ is viewed as a $U^{\hb,\geq}_\zeta$-module by $U^{\hb,\geq}_\zeta= U_\zeta^+\otimes \fU_\zeta^0\xrightarrow{\Fr\otimes \pi\circ \tau_{\lambda}} U\n \otimes R_\lambda$. 

\begin{lem}[{\cite[Lem 3.7]{BBASV} and \cite[Lem 3.2, Prop 3.4]{Situ1}}]\label{lem 3.2} \ 
\begin{enumerate} 
\item The functor $\Hom_{U^\hb_{\zeta}\Mod^{\Lambda}_R}(Q(\lambda)_R,-)$ on $\sO_{R}$ is exact; 
\item The projective cover for $E(\lambda)_\k$ in $\sO^{\leq \nu}_{S}$ is $Q(\lambda)_S^{\leq \nu}\simeq \tau^{\leq \nu}Q(\lambda)_S$. 
In particular, for any $\lambda\leq \nu$, we have 
$$Q(-\rho+l\lambda)^{\leq -\rho+l\nu}_S= U^\hb_\zeta \otimes_{U^{\hb,\geq}_\zeta} \big((U\n/ \bigoplus_{\mu\nleq \nu-\lambda} (U\n)_{\mu}) \otimes S_{-\rho+l\lambda}\big).$$ 
\item (BGG reciprocity) For any $\lambda \leq \mu \leq \nu$, we have an equality 
$$(Q(\lambda)^{\leq \nu}_R: M(\mu)_R)=[M(\mu)_\F: E(\lambda)_\F].$$ 
\end{enumerate} 
\end{lem} 

%=================================
\subsubsection{Blocks decomposition}\label{subsect 2.3.2} 
%=================================
In \cite[\textsection3.3]{Situ1} (see also \cite[II \textsection6]{Jan03}) we introduce a partial order $\uparrow$ on $\Lambda$ generated by 
\[
s\bullet \lambda \uparrow \lambda \quad \text{if}\quad s\bullet \lambda \leq \lambda
\]
where $s\in W_{l,\af}$ is a reflection. 
Note that the order $\uparrow$ is invariant under $l\Lambda$-translation, i.e. if $\mu \uparrow \lambda$ then $\mu+l\nu \uparrow \lambda+l\nu$ for any $\nu\in \Lambda$. 
Moreover, two weights in different $(W_{l,\af},\bullet)$-orbits in $\Lambda$ are incomparable under $\uparrow$.

\begin{lem}[{\cite[Prop 3.7]{Situ1}}]\label{lem LP} 
We have the following linkage principle: 
\begin{equation}\label{equ LP}
[M(\lambda)_\k:E(\mu)_\k]\neq 0 \quad \text{if and only if}\quad \mu \uparrow \lambda .
\end{equation}
%where $\uparrow$ is the partial order induced by the $\bullet$-action of $W_{l,\af}$ on $\Lambda$. 
In particular, there is a block decomposition 
\begin{equation}\label{equ 2.12} 
\sO_{R}=\bigoplus_{\omega\in \Xi_\sc} \sO^{ \omega}_{R}, 
\end{equation} 
such that the Verma module $M(\lambda)_R$ is contained in $\sO^{\omega}_{R}$ if and only if $\lambda\in W_{l,\af}\bullet \omega$. 
Moreover, the block $\sO^\omega_\k$ is indecomposable. 
\end{lem} 
\noindent 
We will abbreviate $\sO^{\omega,\leq \nu}_{R}=\sO^{\omega}_{R}\cap \sO^{\leq \nu}_{R}$. 

We give another construction of (\ref{equ 2.12}) here. 
Since the image of $Z_\HC \rightarrow \fU_\zeta \rightarrow U^\hb_\zeta$ is central in $U^\hb_\zeta$, we have homomorphisms $\k[T/W] \xrightarrow{\hc^{-1}} Z_\HC\rightarrow Z(\sO_{\k})$. 
By (\ref{equ 2.-1}), the composition factors through the quotient 
$$\k[T/W]\rightarrow \k[\Omega],$$ 
where $\Omega$ is the scheme-theoretic preimage of $W(1)$ of the morphism $T/W\rightarrow T/W$, given by $W(t) \mapsto W(t^l)$, for any $t\in T$. 
Consider the map 
$$\Lambda \rightarrow T/W ,\quad \lambda\mapsto W(\zeta^{2(\lambda+\rho)}),$$ 
which induces a bijection $\Lambda/(W_{l,\ex},\bullet)=\Xi \xs \Omega^{\red}$. 
It yields a decomposition of schemes 
\begin{equation}\label{equ 3.3.1} 
\Omega= \bigsqcup_{[\omega]\in \Xi} \Omega_{[\omega]}. 
\end{equation} 
Note that the character $Z_\HC \rightarrow \End_{U^\hb_\zeta}(M(\lambda)_\k)= \k$ corresponds to the point $W(\zeta^{2(\lambda+\rho)})\in T/W$, for any $\lambda\in \Lambda$. 
We denote by $[\lambda]$ the class of $\lambda$ in $\Xi$. 
Then there is an (extended) block decomposition 
\begin{equation}\label{equ 3.3} 
\sO_{\k} = \bigoplus_{[\omega]\in \Xi} \sO^{ [\omega]}_{\k} 
\end{equation} 
compatible with (\ref{equ 3.3.1}), such that $M(\lambda)_\k$ lies in $\sO^{ [\omega]}_{\k}$ if and only if $[\lambda]=[\omega]$. 
Refining (\ref{equ 3.3}) by (\ref{equ 3.2}), we get a block decomposition 
\begin{equation}\label{equ 3.4} 
\sO_{\k} = \bigoplus_{\omega\in \Xi_\sc} \sO^{\omega}_{\k}.  
\end{equation} 
Since $S$ is local with residue field $\k$, (\ref{equ 3.4}) can be lifted to a decomposition for $\sO_{S}$, and then extends to the one (\ref{equ 2.12}). 

%=================================
\subsubsection{The $l\Lambda$-symmetry} 
%=================================
Note that for any $\nu\in \Lambda$, there is a trivial $U_\zeta^\hb$-module $\k_{l\nu}$ supported on the weight $l\nu$. 
It gives an auto-equivalence of $\sO_{R}$, 
$$-\otimes \k_{l\nu}: \sO_{R}\xs \sO_{R}.$$ 
If $\omega_1 ,\omega_2 \in \Xi_\sc$ satisfy $[\omega_1]=[\omega_2]$, then there are $\lambda_i\in W_{l,\af}\bullet \omega_i$, $i=1,2$ such that $\lambda_1-\lambda_2\in l\Lambda$, which gives an equivalence $$-\otimes \k_{\lambda_2-\lambda_1}: \sO^{\omega_1}_{R}\xs \sO^{\omega_2}_{R}.$$ 
Therefore, the refinement of $\sO^{[\omega]}_{R}$ by (\ref{equ 3.2}) yields an equivalence 
\[
\sO^{[\omega]}_{R}\simeq (\sO^{\omega}_{R})^{\oplus \pi_1}. 
\]

%=================================
\subsubsection{The maps from cohomology to center} 
%=================================
Consider the central characters associated to Verma modules, 
$$\chi_R:\ Z(\sO_{R})\rightarrow \prod_{\lambda\in \Lambda}\End_{\sO_{R}} (M(\lambda)_R)=\Fun(\Lambda,R).$$ 
\begin{thm}[{\cite[Thm 5.4]{Situ1}}]\label{thm 3.11} 
There is a commutative diagram of algebra homomorphisms 
$$\begin{tikzcd}
 H_{\check{T}}^\bullet(\Gr^\zeta)^\wedge_S \arrow[r,"{\bb}","\sim"'] \arrow[d]
 & Z(\sO_{S})\arrow[d]\\ 
H^\bullet(\Gr^\zeta)^\wedge \arrow[r,"\overline{\bb}"] & Z(\sO_{\k}),
\end{tikzcd}$$
such that the composition 
$$\chi_{S}\circ {\bb}: H_{\check{T}}^\bullet(\Gr^\zeta)^\wedge_S\rightarrow \Fun(\Lambda,S)$$ 
coincides with the restrictions on the $\check{T}$-fixed points $\{\delta^{\check{\Sigma}}_{\lambda+\rho}\}_{\lambda\in \Lambda}$. 
In particular, the isomorphism $\bb$ is compatible with the decompositions (\ref{equ 1.8new}) and (\ref{equ 2.12}). 
\end{thm} 
\noindent 
By restriction on the direct summands corresponding to $\omega$, we obtain homomorphisms 
$$\bb_{[\omega]}:H_{\check{T}}^\bullet(\Fl^{\omega})^\wedge_S\rightarrow Z(\sO^{[\omega]}_{S}), \quad 
\bb_{\omega}:H_{\check{T}}^\bullet(\Fl^{\omega,\circ})^\wedge_S\rightarrow Z(\sO^{\omega}_{S}),$$ 
and 
$$\overline{\bb}_{[\omega]}:H^\bullet(\Fl^{\omega})^\wedge\rightarrow Z(\sO^{[\omega]}_{\k}), \quad 
\overline{\bb}_{\omega}:H^\bullet(\Fl^{\omega,\circ})^\wedge\rightarrow Z(\sO^{\omega}_{\k}).$$ 

%%%%%%%%%%%%%%%%%%%%%%%%%
\section{The Steinberg block}\label{sect 3} 
%%%%%%%%%%%%%%%%%%%%%%%%%
In this section we establish an equivalence between the \textit{Steinberg block} $\sO^{[-\rho]}_{\k}$ and the category $\Coh^B(\n)$, by relating them to the category $\sO$ for $U^\hb_1$. 
Next, we apply this equivalence to show that the algebra homomorphism 
$$\overline{\bb}_{[-\rho]}:H^\bullet(\Gr)^\wedge \rightarrow Z(\sO^{[-\rho]}_{\k})$$ 
is an isomorphism. 

%---------------------------------------------------------------
\subsection{Hybrid quantum Frobenius map}\label{subsect 4.1} 
%---------------------------------------------------------------
In this subsection, we construct a sub-quotient algebra $U^{\hb'}_1$ of $U^\hb_\zeta$ which is isomorphic to a central extension of $U^\hb_1$. 

%=================================
\subsubsection{Quantum coadjoint actions}
%=================================
Denote by $\underline{X}_{i}, \underline{Y}_{i}$ the adjoint operators on $U_q$ associated with the elements $E^{(l)}_{i}, F^{(l)}_{i}$, for any $i\in \I$, i.e. 
$$\underline{X}_{i} : U_q \rightarrow U_q, \  x \mapsto [E^{(l)}_{i},x] , \quad  \underline{Y}_{i} : U_q \rightarrow U_q, \  x \mapsto  [F^{(l)}_{i},x] . $$ 
By \cite[\textsection 3.4]{DeCK90}, they preserve the subalgebra $\fU_q\subset U_q$, and yield operators $\underline{X}_{i}, \underline{Y}_{i}$ on its specialization $\fU_\zeta$. 
Furthermore, the derivations $\underline{X}_{i}$, $\underline{Y}_{i}$ preserve $Z_\Fr\subset \fU_\zeta$, and induce tangent fields $\underline{X}_{i}, \underline{Y}_{i} \in \Gamma( G^*, \sT_{G^*})$ via the isomorphism (\ref{equ 2.7}). 
There is a unramified covering  from $G^*$ onto the big open cell $N^- T N$ in $G$, given by 
$$\kappa: G^* \rightarrow G, \quad (n_-,t ,n_+)\mapsto n_- t^2 n_+^{-1}, \quad \forall n_{-}\in N^{-},\ n_+\in N,\ t\in T .$$ 
Consider the pullback of tangent fields $\kappa^*: \Gamma(G,\sT_{G})\rightarrow \Gamma(G^*,\sT_{G^*})$. 
In the theorem below, we view $\g \subset \Gamma( G, \sT_{G})$ as the subspace of Killing vector fields on $G$ (i.e. the tangent fields induced by the conjugate action). 

\begin{thm}[{\cite[\textsection 5]{DeCKP92}}]\label{thm 4.1} 
The tangent fields $e_{\alpha_i}$, $f_{\alpha_i}$ on $G$ and $\underline{X}_{i}$, $\underline{Y}_{i}$ on $G^*$ are related by 
$$ \kappa^*(f_{\alpha_i})= -K_i^l \underline{X}_{i}, \quad \kappa^*(e_{\alpha_i})= K_i^l \underline{Y}_{i} , \quad \forall i\in \I. $$ 
\end{thm} 

\begin{lem}\label{lem 4.2} 
Let $i\in \I$. 
\begin{enumerate} 
\item The operator $[E_{i}^{(l)}, -]$ on $U^\hb_\zeta$ preserves the subalgebras $Z_\Fr^{\leq}$ and $Z_\Fr^{-}\otimes \fU^0_\zeta$. 
\item It induces a $U\n$-action on $Z_\Fr^{-}\otimes \fU^0_\zeta$ such that $e_{\alpha_i}$ acts via $-K_i^l [E_{i}^{(l)}, -]$ for any $i\in \I$. 
    There is a $U\n$-isomorphism of algebras 
	\begin{equation}\label{equ 4.1} 
	Z_\Fr^{-}\otimes \fU^0_\zeta \xs \k[B\times_T T], 
	\end{equation} 
	where the base change $T\rightarrow T$ is by $t\mapsto t^{2l}$ for any $t\in T$, and the $U\n$-action on $\k[B\times_T T]$ is induced by the $N$-conjugation on $B$. 
\end{enumerate} 
\end{lem} 
\begin{proof} 
    (1) Note that $[E_{i}^{(l)},-]$ is trivial on $\fU^0_\zeta\subset U^\hb_\zeta$. 
    We show $[E_{i}^{(l)}, F_\beta^l]\in Z_\Fr^{\leq}$ for any $\beta\in \Phi^+$. 
	Since $\underline{X}_{i}$ preserves the subspace $\fU_q\subset U_q$ and its specialization on $\fU_\zeta$ preserves $Z_\Fr$, it follows that in the integral form $U_q^\hb$, 
	$$ [E_{i}^{(l)}, F_\beta^l] \in \sum_{\wp, \wp'\in (l\N)^{\Phi^+}, \lambda\in \Lambda} \k \ F^\wp K^l_\lambda E^{\wp'} \quad \mathrm{mod} \ (q_e-\zeta_e)\fU_q .$$ 
	Hence in $U^\hb_\zeta$, we have $[E_{i}^{(l)}, F_\beta^l]\in Z_\Fr^{\leq}$. 
	
	(2) Denote by $(Z_\Fr^+)_+$ the augmentation ideal of $Z_\Fr^+$. 
	Since the operator $\underline{X}_{i}$ on $Z_\Fr$ stables $(Z_\Fr^+)_+$, it induces an operator (still denoted by $\underline{X}_{i}$) on 
	$$Z_\Fr^{\leq}= Z_\Fr/ Z_\Fr(Z_\Fr^+)_+ ,$$ 
	which coincides with the action $[E_{i}^{(l)}, -]$ on $Z_\Fr^{\leq}$ in (1). 
	Consider the Cartesian diagram 
	$$\begin{tikzcd} \Spec Z_\Fr^{\leq} \arrow[r,two heads,"\kappa"] \arrow[d,hook] & B^- \arrow[d,hook] \\ 
    \Spec Z_\Fr \arrow[r,two heads,"\kappa"] & N^- T N.
    \end{tikzcd}$$
	Note that the Killing vector field $f_{\alpha_i}$ on $G$ is tangent to $B^-$. 
	By Theorem \ref{thm 4.1}, we have $\kappa'^*(f_{\alpha_i})=-K_i^l \underline{X}_{i}$ as tangent fields on $\Spec Z_\Fr^{\leq}$, and $\kappa'$ induces a $U\n^{-}$-isomorphism 
	$$Z_\Fr^{\leq} \xs \k[B^-\times_T T],$$ 
	where $T\rightarrow T$ is by $t\mapsto t^{2}$. 
	By the inclusion $Z_\Fr^0=\k\langle K_\lambda^l\rangle_{\lambda\in \Lambda}\hookrightarrow \fU^0_\zeta=\k\langle K_\lambda\rangle_{\lambda\in \Lambda}$, it extends to a $U\n^{-}$-isomorphism $Z_\Fr^{-}\otimes \fU^0_\zeta \xs \k[B^-\times_T T]$, where $T\rightarrow T$ is by $t\mapsto t^{2l}$. 
	Finally, replacing $B^-,\n^-$ by $B,\n$ via the Chevalley involution, we get the desired isomorphism (\ref{equ 4.1}). 
\end{proof} 

%=================================
\subsubsection{Sub-quotient algebra} 
%=================================
Consider the subspace $\fU_\zeta^{\hb}:=Z_{\Fr}^- \otimes U_\zeta^{\hb,\geq}$ of $U_\zeta^\hb$. 
Denote by $(u_\zeta^+)_+$ the augmentation ideal of $u_\zeta^+$. 
\begin{prop}\label{prop 4.3} 
The subspace $\fU_\zeta^{\hb}$ is a subalgebra of $U_\zeta^\hb$. 
There is an algebra isomorphism 
\begin{equation}\label{equ 4.3} 
\fU_\zeta^{\hb}/\langle (u_\zeta^+)_+\rangle \xs \k[B\times_T T]\rtimes U\n, 
\end{equation} 
which is compatible with (\ref{equ 4.1}), and sends $-K_\beta^lE^{(l)}_\beta$ to $1\otimes e_{\beta}$ for any $\beta\in \Phi^+$. 
\end{prop} 
\begin{proof} 
We show the subspace $\fU_\zeta^{\hb}=Z_{\Fr}^- \otimes U_\zeta^{\hb,\geq}$ is closed under multiplications. 
Indeed, the left multiplications by the elements $E_i$, $F^l_\beta$, $K_\lambda$ stabilize $\fU_\zeta^{\hb}$, since they commute with $Z_{\Fr}^-$. 
The same holds for $E^{(l)}_i$ by Lemma \ref{lem 4.2}(1). 
Hence $\fU_\zeta^{\hb}$ is closed under the left multiplications. 
Similar arguments apply to the right multiplications. 

For the second assertion, we denote by $\langle u_\zeta^+\rangle_+$ the ideal in $U^+_\zeta$ generated by $(u_\zeta^+)_+$, which is also the kernel of the quantum Frobenius map $\Fr: U_\zeta^+ \twoheadrightarrow U\n$. 
Since $u_\zeta^+$ commutes with $Z_\Fr^{-}$, the ideal $\langle (u_\zeta^+)_+\rangle$ coincides with the subspace $Z_\Fr^{-}\otimes \fU^0_\zeta\otimes \langle u_\zeta^+\rangle_+$ in $\fU^\hb_\zeta$. 
It gives an isomorphism of $\k$-vector spaces $Z_\Fr^{-}\otimes \fU^0_\zeta\otimes U\n \xs \fU_\zeta^{\hb}/\langle (u_\zeta^+)_+\rangle$.  
Combining with (\ref{equ 4.1}), there is an isomorphism of $\k$-vector spaces 
\begin{equation}\label{equ 4.4} 
\k[B\times_T T] \rtimes U\n \xs Z_\Fr^{-}\otimes \fU^0_\zeta\otimes U\n \xs \fU_\zeta^{\hb}/\langle (u_\zeta^+)_+\rangle, 
\end{equation} 
such that $1\otimes e_\beta \mapsto -K_\beta^lE^{(l)}_\beta \text{ mod } \langle (u_\zeta^+)_+\rangle$ for any $\beta\in \Phi^+$. 
It remains to show (\ref{equ 4.4}) is an isomorphism of $\k$-algebras. 
Indeed, (\ref{equ 4.4}) realizes $\k[B\times_T T]$ and $U\n$ as subalgebras in $\fU_\zeta^{\hb}/\langle (u_\zeta^+)_+\rangle$, and Lemma \ref{lem 4.2}(2) shows that these subalgebras glue together in the way we want. 
\end{proof} 

We set $U^{\hb'}_1:=\k[B\times_T T]\rtimes U\n$, then there is natural inclusion $U^\hb_1\hookrightarrow U^{\hb'}_1$, and (\ref{equ 4.3}) gives an algebra surjection 
$$\Fr^\hb: \fU^{\hb}_\zeta \rightarrow U^{\hb'}_1.$$ 

%---------------------------------------------------------------
\subsection{Equivalence for the Steinberg block}\label{subsect 3.2} 
%---------------------------------------------------------------
In this subsection, we use the sub-quotient structure shown in \textsection \ref{subsect 4.1} to construct an equivalence between $\sO^{[-\rho]}_{\k}$ and the category $\sO$ for $U^\hb_1=\k[B]\rtimes U\n$, and then interpret the latter as $\Coh^B(\n)$. 
A deformed version is also discussed. 
\subsubsection{Module categories for $U^{\hb}_1$} 
The $\rQ$-grading on $U^\hb_\zeta$ induces an $l\rQ$-grading on $U^{\hb}_1$ and $U^{\hb'}_1$. 
We set  
$$U^{\hb,0}_1=\k\langle K^{2l}_{\lambda}\rangle_{\lambda\in \Lambda}\subset U^{\hb',0}_1=\fU^0_\zeta,$$ 
then $U^{\hb}_1$ and $U^{\hb'}_1$ naturally fit into the settings of \textsection \ref{subsect 2.3}. 
Let $A$ be either $U^\hb_1$ or $U^{\hb'}_1$, and $R$ be a deformation ring for $A$. 
For any lattice $l\rQ\subset \Lambda' \subset \Lambda$, we define $A\Mod_R^{\Lambda'}$ to be the full subcategory of $A\Mod_R^{\Lambda}$ of the $\Lambda'$-graded modules. 
The action of $l\Lambda$ on $\fU^0_\zeta$ is trivial, hence if $R$ is a $\fU^0_\zeta$-algebra, then we have $U^{\hb}_1\Mod^{l\Lambda}_R=U^{\hb'}_1\Mod^{l\Lambda}_R$. 

We define $\sO_{1,R}$ to be the full subcategory of $U^{\hb}_1\Mod^{l\Lambda}_R$ consisting of the modules that are finitely generated over $U^{\hb}_1\otimes R$ and are locally unipotent under the action of $U\n$. 
Define the Verma module 
$$M(\lambda)_{1,R}:= U^\hb_1\otimes_{\k[T]\otimes U\n}R_{l\lambda}\in \ \sO_{1,R} , \quad \lambda\in \Lambda.$$ 

%=================================
\subsubsection{An invariant functor} 
%=================================
Let $\langle u_\zeta^+\rangle_+$ be the ideal of $U_\zeta^+$ generated by $(u_\zeta^+)_+$. 
For any $\fU_\zeta^{\hb}$-module $M$, we denote by $M^{\langle u_\zeta^+\rangle_+}$ the subspace of $M$ where $\langle u_\zeta^+\rangle_+$ acts by zero. 
By the isomorphism (\ref{equ 4.3}), $M^{\langle u_\zeta^+\rangle_+}$ is naturally a $U^{\hb'}_1$-module. 
Let $R$ be a deformation ring for $U^\hb_\zeta$. 
Consider the functor 
$$(-)^{\langle u_\zeta^+\rangle_+}: 
U_\zeta^\hb\Mod^{\Lambda}_{R}\xrightarrow{\text{forget}} \fU_\zeta^{\hb}\Mod^{\Lambda}_{R}\rightarrow U^{\hb'}_1\Mod^{\Lambda}_R.$$ 
\begin{lem}\label{lem 4.4} 
Suppose $R=S$ or $\k$. The functor $(-)^{\langle u_\zeta^+\rangle_+}$ is exact on $\sO^{[-\rho]}_{R}$. 
\end{lem} 
\begin{proof} 
We define the following module in $U_\zeta^\hb\Mod^{\Lambda}_{R}$ as (\ref{equ 3.20}), 
$$Q_1(\mu)_R= U_\zeta^\hb\otimes_{U_\zeta^{\hb,\geq}}(U\n\otimes R_\mu), \quad \mu\in \Lambda.$$ 
For any $M\in U_\zeta^\hb\Mod^{\Lambda}_{R}$, we have an isomorphism of $\Lambda$-graded $R$-modules 
$$\bigoplus_{\mu\in \Lambda} \Hom_{U_\zeta^\hb\Mod^{\Lambda}_{R}}(Q_1(\mu)_R,M)\xs M^{\langle u_\zeta^+\rangle_+},$$ 
given by evaluating $(f_\mu)_{\mu}$ at the elements $1\otimes 1\otimes 1\in Q_1(\mu)_R$. 
It yields a natural isomorphism 
$$\bigoplus_{\mu\in \Lambda} \Hom_{U_\zeta^\hb\Mod^{\Lambda}_{R}}(Q_1(\mu)_R,-)\xs (-)^{\langle u_\zeta^+\rangle_+},$$ 
as functors from $U_\zeta^\hb\Mod^{\Lambda}_{R}$ to the category of $\Lambda$-graded $R$-modules. 
If $M\in \sO^{[-\rho]}_{R}$, then by the block decomposition (\ref{equ 2.12}), we have $\Hom(Q_1(\mu)_R,M)\neq 0$ only if $\mu\in -\rho+l\Lambda$, and in this case $Q_1(\mu)_R=Q(\mu)_R$ by (\ref{equ 3.20}). 
Therefore, by Lemma \ref{lem 4.2}(1), the functor 
\begin{equation}\label{equ 4.40} 
(-)^{\langle u_\zeta^+\rangle_+}= 
\bigoplus_{\mu\in \Lambda} \Hom_{U_\zeta^\hb\mod^{\Lambda}_{R}}(Q(-\rho+l\mu)_R,-)
\end{equation}
on $\sO^{[-\rho]}_{R}$ is exact. 
\end{proof} 

Recall that the Verma module $M(\lambda)_R=\fU^-_\zeta \otimes R_{\lambda}$ as $\fU^{\leq}_\zeta\otimes R$-modules. 
\begin{lem}\label{lem 4.6} 
Suppose $R=S$ or $\k$. 
If $\lambda\in -\rho+l\Lambda$, then we have $M(\lambda)_R^{\langle u_\zeta^+\rangle_+}= Z_\Fr^-\otimes R_{\lambda}$.  
\end{lem} 
\begin{proof} 
Without loss of generality, we may assume $\lambda=-\rho$. 
By (\ref{equ 4.40}), there are isomorphisms of $R$-modules 
\begin{equation}\label{equ 3.5} 
\begin{aligned} 
M(-\rho)_R^{\langle u_\zeta^+\rangle_+}
&=\bigoplus_{\mu\in \rQ} \Hom_{U_\zeta^\hb\Mod^{\Lambda}_{R}}(Q(-\rho+l\mu)_R,M(-\rho)_R)\\
&=\Hom_{U_\zeta^\hb\otimes {R}}(Q(-\rho)_R,M(-\rho)_R)\\ 
&=\End_{\fU_\zeta^b\otimes {R}}(M(-\rho)_R). 
\end{aligned} 
\end{equation}
Since $Z_\Fr^-$ is central in $\fU_\zeta^b$, the ring $\End_{\fU_\zeta^b\otimes {R}}(M(-\rho)_R)$ is naturally a $Z_\Fr^-$-algebra, and the isomorphism (\ref{equ 3.5}) is a homomorphism of $Z_\Fr^-\otimes R$-modules. 
Consider the $Z_\Fr^-\otimes R$-homomorphism 
$$\phi:\ Z_\Fr^-\otimes R\rightarrow \End_{\fU_\zeta^b\otimes {R}}(M(-\rho)_R), \quad 1\otimes 1\mapsto \id_{M(-\rho)_R}.$$ 
Note that $Z_\Fr^-\otimes R_{\lambda}\subseteq M(-\rho)_R^{\langle u_\zeta^+\rangle_+}$ coincides with the image of $\phi$ under the isomorphism (\ref{equ 3.5}). 
It remains to show that $\phi$ is a surjection. 
Since $\phi$ preserves the $\Lambda$-grading of both sides, and the weight spaces of $Z_\Fr^-\otimes R$ and $M(-\rho)_R^{\langle u_\zeta^+\rangle_+}$ are finitely generated over the local ring $R$, by Nakayama's Lemma, it is enough to consider the case $R=\k$. 
Denote by $(Z_\Fr^-)_+$ the augmentation ideal for $Z_\Fr^-$, and $\overline{\phi}$ the specialization of $\phi$ on $\k=Z_\Fr^-/(Z_\Fr^-)_+$. 
By \cite[Lem 6.3]{AJS94}, there is a short exact sequence of $\fU^b_\zeta$-modules 
$$0\rightarrow (Z_\Fr^-)_+.M(-\rho)_\k \rightarrow M(-\rho)_\k \rightarrow L^b(-\rho)_\k\rightarrow 0.$$ 
Since $M(-\rho)_\k$ is the projective cover of $L^b(-\rho)_\k$ in $\fU^b_\zeta\mod^\Lambda_\k$, the specialization $\overline{\phi}$ is by 
$$\overline{\phi}:\ Z_\Fr^-/(Z_\Fr^-)_+=\k \xs \Hom_{\fU_\zeta^b\otimes {R}}(M(-\rho)_\k,L^b(-\rho)_\k)=\k.$$ 
Note that $\End_{\fU_\zeta^b\otimes {R}}(M(-\rho)_R)$ is a $\rQ_{\leq}$-graded module of the $l\rQ_{\leq}$-graded algebra $Z_\Fr^-$. 
By graded Nakayama's Lemma, $\phi$ is a surjection. 
\end{proof} 

%=================================
\subsubsection{Equivalence for the Steinberg block} 
%=================================
Consider the short exact sequence 
\begin{equation}\label{equ 4.7} 
 0\rightarrow \langle F^l_\beta, K^l_{\lambda}-1, E_i^{(n)}\rangle \rightarrow \fU_\zeta^{\hb} \rightarrow \k\langle K_\lambda\rangle/\langle K^l_\lambda-1\rangle \rightarrow 0. 
\end{equation} 
For any $\mu\in \Lambda$, there is a one dimensional module $\k_\mu$ of $\k\langle K_\lambda\rangle/\langle K^l_\lambda-1\rangle$ such that $K_\lambda$ acts by $\zeta^{(\lambda,\mu)}$. 
It gives a module of $\fU_\zeta^{\hb}$ by pullback via the third map in (\ref{equ 4.7}). 
Consider the following functors 
$$\rR^\hb_{1,R}:\ 
U_\zeta^\hb\Mod^{\Lambda}_{R}\rightarrow U^{\hb'}_1\Mod^{\Lambda}_R, 
\quad M \mapsto  (M\otimes \k_{\rho})^{\langle u_\zeta^+\rangle_+},$$ 
and 
$$\rI^\hb_{1,R}:\ 
U^{\hb'}_1\Mod^{\Lambda}_R \rightarrow U_\zeta^\hb\Mod^{\Lambda}_{R}, \quad V\mapsto U_\zeta^\hb\otimes_{\fU_\zeta^{\hb}} (\Fr^{\hb,*}(V)\otimes \k_{-\rho}).$$ 
Note that $(\rI^\hb_{1,R},\rR^\hb_{1,R})$ forms an adjoint pair: for any $V\in U_1^{\hb'}\Mod^{\Lambda}_R$ and $M\in U_\zeta^\hb\Mod^{\Lambda}_{R}$, there are natural isomorphisms 
\begin{align*} 
	\Hom_{U_\zeta^\hb\Mod^{\Lambda}_{R}}(\rI^\hb_{1,R}(V),M)&=\Hom_{\fU_\zeta^{\hb}\Mod^{\Lambda}_R}(\Fr^{\hb,*}(V)\otimes \k_{-\rho} ,M) \\ 
	&=\Hom_{\fU_\zeta^{\hb}\Mod^{\Lambda}_R}(\Fr^{\hb,*}(V),M\otimes \k_{\rho} ) \\ 
	&=\Hom_{U_1^{\hb'}\Mod^{ \Lambda}_R}(V,\rR^\hb_{1,R}(M)). 
\end{align*} 

\begin{thm}\label{thm 4.7} 
Suppose $R=S$ or $\k$. 
The functors $\rR^{\hb}_{1,R}$ and $\rI^{\hb}_{1,R}$ restrict to the full subcategories 
$$\rI^{\hb}_{1,R}:\ \sO_{1,R} \rightleftarrows \sO^{[-\rho]}_{R}\ : \rR^{\hb}_{1,R},$$ 
and induce an equivalence of categories, matching $M(\lambda)_{1,R}$ with $M(-\rho+l\lambda)_R$ for any $\lambda\in \Lambda$. 
\end{thm} 
\begin{proof} 
\textit{Step 1. Show the exactness for $\rR^{\hb}_{1,R}$ and $\rI^{\hb}_{1,R}$.} 
Note that for any module $M\in \fU_\zeta^{\hb}\Mod^{\Lambda}_{R}$, we have $(M\otimes \k_{\rho})^{\langle u_\zeta^+\rangle_+}=M^{\langle u_\zeta^+\rangle_+}\otimes \k_{\rho}$ as $R$-modules. 
By Lemma \ref{lem 4.4}, $\rR^{\hb}_{1,R}$ is exact on $\sO^{[-\rho]}_{R}$. 
Note that $\rI^{\hb}_{1,R}$ is the composition of the functors $\Fr^{\hb,*}$, $-\otimes \k_{\rho}$ and $U_\zeta^\hb\otimes_{\fU_\zeta^{\hb}}-$, where the first two are clearly exact. 
Since $U_\zeta^\hb\otimes_{\fU_\zeta^{\hb}}M=\fU^-_\zeta \otimes_{Z^-_\Fr}M$ as $R$-modules, and $\fU^-_\zeta$ is a free $Z^-_\Fr$-module, the functor $U_\zeta^\hb\otimes_{\fU_\zeta^{\hb}}-$ is also exact. 
It shows the exactness of $\rI^{\hb}_{1,R}$. 
	
\textit{Step 2. Show that $\rR^{\hb}_{1,R}$ and $\rI^{\hb}_{1,R}$ can be restricted to the category $\sO$'s.} 
We show that $\rI^{\hb}_{1,R}$ and $\rR^{\hb}_{1,R}$ send Verma modules to Verma modules, then the result follows by exactness. 
Indeed, for any $\lambda\in \Lambda$, by Lemma \ref{lem 4.6} and (\ref{equ 4.1}) we have 
\begin{equation}\label{equ 4.8} 
\rR^{\hb}_{1,R}(M(-\rho+l\lambda)_R)= M(-\rho+l\lambda)^{{\langle u_\zeta^+\rangle_+}}_R \otimes \k_{\rho} = Z_\Fr^-\otimes R_{l\lambda} = M(\lambda)_{1,R}.  
\end{equation}
On the other hand, 
\begin{equation}\label{equ 4.9}
\begin{aligned}
 \rI^{\hb}_{1,R}( M(\lambda)_{1,R}) 
 &=U_\zeta^\hb\otimes_{\fU_\zeta^{\hb}} \big((\fU_\zeta^{\hb}\otimes_{U_\zeta^{\hb,\geq}} R_{l\lambda}) \otimes \k_{-\rho}\big) \\ 
 &= U_\zeta^\hb\otimes_{\fU_\zeta^{\hb}} (\fU_\zeta^{\hb}\otimes_{U_\zeta^{\hb,\geq}} R_{-\rho+l\lambda})= M(-\rho+l\lambda)_R. 
\end{aligned} 
\end{equation} 

\textit{Step 3. Show that $\rR^{\hb}_{1,R}$ and $\rI^{\hb}_{1,R}$ give an equivalence of categories.} 
By exactness, it is enough to show that the unit and counit morphisms associated with the adjoint pair $(\rR^{\hb}_{1,R},\rI^{\hb}_{1,R})$ are isomorphisms on projective modules (in truncated categories), which only needs to be verified for Verma modules since projective modules admit Verma flags. 
This follows from (\ref{equ 4.8}) and (\ref{equ 4.9}). 
\end{proof}

%=================================
\subsubsection{Equivariant coherent sheaves}\label{subsect 4.2.4} 
%=================================
Let $R$ be a deformation ring for $U^\hb_1$. 
Note that any module in $\sO_{1,R}$ is naturally a coherent sheaf on $B\times_T \Spec R$, and the requirements of locally unipotent $U\n$-actions and the $l\Lambda$-gradings amount to give a $B$-equivariant structure on it, where $B$ acts on $B\times_T \Spec R$ by the conjugation (on $B$). 
Thus, there is a tautological identification 
\begin{equation}\label{equ 4.9.1} 
\sO_{1,R}=\Coh^B(B\times_T \Spec R). 
\end{equation} 
Under (\ref{equ 4.9.1}), the Verma module $M(\lambda)_{1,R}$ in $\sO_{1,R}$ corresponds to the sheave $\sO_{B\times_T \Spec R}\otimes \k_\lambda$, where $\k_\lambda$ is the $1$-dimensional representation for $B$ associated with $\lambda\in \Lambda$. 

Denote by $\t_{\hat{0}}$ and $T_{\hat{1}}$ the completion at $0\in \t$ and $1\in T$. 
Let $R$ be the $\fU^0_\zeta$-algebra $S$, then we can identify 
$$B\times_{T} \Spec S= N\times T_{\hat{1}} =\n\times \t_{\hat{0}}$$ 
via the following well-known lemma 
\begin{lem} 
The exponential map $\exp:\b \rightarrow B$ induces an isomorphism of $B$-schemes $\exp : \n\times \t_{\hat{0}} \xs N\times T_{\hat{1}}$. 
\end{lem} 
So the equivalence in Theorem \ref{thm 4.7} can be reformulated as follows. 
\begin{corollary}\label{cor 4.8} 
There are compatible equivalences of abelian categories 
$$ \sO^{[-\rho]}_{S}\xs \Coh^B(\n\times\t_{\hat{0}}),\quad \sO^{[-\rho]}_{\k}\xs \Coh^B(\n) ,$$ 
sending $M(-\rho+l\lambda)_S$, $M(-\rho+l\lambda)_\k$ to $\sO_{\n\times\t_{\hat{0}}}\otimes \k_\lambda$, $\sO_{\n}\otimes \k_\lambda$, for any $\lambda\in \Lambda$. 
\end{corollary} 

%---------------------------------------------------------------
\subsection{Center of the Steinberg block}\label{subsect 3.3} 
%---------------------------------------------------------------
In this subsection, we give two descriptions of $Z(\sO^{[-\rho]}_\k)$. 
The first one in Theorem \ref{prop 4.12} is geometric, obtained by using the equivalence in Theorem \ref{thm ABG}. 
 The second one in Corollary \ref{cor 3.15} is algebraic, obtained by analyzing a limit of endomorphism rings of projective modules in truncated categories. 

%=================================
\subsubsection{Arkhipov--Bezrukavnikov--Ginzburg equivalence} 
%=================================
Recall the Bruhat decomposition for the affine Grassmannian $\Gr$ into $\sI$-orbits, 
$$\Gr= \bigsqcup_{\lambda\in \Lambda} \Gr_\lambda, \quad \Gr_\lambda:= \sI \delta_\lambda.$$  
The triangulate category of $\sI$-monodromic, resp. $\sI$-equivariant $\ell$-adic mixed complexes on $\Gr$ is defined as the colimit 
$$ D^{\mathrm{b},\mx}_{(\sI)}(\Gr):= \cLim{\lambda} D^{\mathrm{b},\mx}_{(\sI)}(\overline{\Gr_\lambda}) \quad \text{resp.} \quad D^{\mathrm{b},\mx}_{\sI}(\Gr):= \cLim{\lambda} D^{\mathrm{b},\mx}_{\sI}(\overline{\Gr_\lambda}) .$$ 
Denote by $\langle 1\rangle$ the half of the Tate twist. 
Let $i_\lambda : \Gr_\lambda \hookrightarrow \Gr$ be the locally closed inclusion, and 
denote the costandard sheaves by $\nabla_\lambda:= i_{\lambda*} {\overline{\Q}_\ell}_{{\Gr_\lambda}}[\dim \Gr_\lambda]$. 

On the other hand, consider the adjoint action of $B$ and the dilation action of $\C^{\times}$ on the varieties $\b$, $\n$. 
There is a $B\times \C^{\times}$-equivariant embedding $i : \n=\n \times \{0\} \hookrightarrow \b =\n \times \t$. 
We also denote by $\langle 1\rangle$ the $\C^{\times}$-grading shift on $\Coh^{\C^{\times}}(\b)$ and $\Coh^{\C^{\times}}(\n)$. 
We view $(\Coh^{B\times \C^{\times}}(\b),\langle 1\rangle)$ and $(\Coh^{B\times \C^{\times}}(\n),\langle 1\rangle)$ as graded categories in the sense of Appendix \ref{app A3}. 

Let $\Lambda^{++} \subset \Lambda$ be the set of regular dominant characters, i.e. 
$$\Lambda^{++}:=\{\lambda\in \Lambda|\ \langle \check{\alpha}, \lambda\rangle > 0,\ \forall \alpha\in \Phi^+ \}.$$ 
The following equivalence $\Psi$ is proved in {\cite[Thm 9.4.3]{ABG04}}, see also {\cite[Thm 1.4, Prop 6.5]{MR18}} for the deformation equivalence $\widetilde{\Psi}$ and their compatibility (in positive characteristic). 
\begin{thm}[{\cite{ABG04}}, {\cite{MR18}}]\label{thm ABG} 
There are compatible equivalences of triangulate categories 
$$\begin{tikzcd} \Db\Coh^{B\times \C^{\times}}(\b) \arrow[r,"\simeq"',"\widetilde{\Psi}"] \arrow[d,"Li^*"'] & D^{\mathrm{b},\mx}_\sI(\Gr) \arrow[d,"\for"] \\ 
\Db\Coh^{B\times \C^{\times}}(\n) \arrow[r,"\simeq"',"\Psi"] & D^{\mathrm{b},\mx}_{(\sI)}(\Gr) ,
\end{tikzcd}$$ 
such that $\widetilde{\Psi} \circ \langle 1 \rangle=  \langle 1 \rangle [1] \circ \widetilde{\Psi}$ and $\Psi \circ \langle 1 \rangle=  \langle 1 \rangle [1] \circ \Psi$. 
Moreover, $\widetilde{\Psi}(\sO_\b\otimes \k_\lambda)=\nabla_\lambda$ and $\Psi(\sO_\n\otimes \k_\lambda)=\nabla_\lambda$ for any $\lambda\in \Lambda^{++}$. 
\end{thm} 

\subsubsection{Base change}\label{subsect 4.3.2} 
Let $R$ be a commutative Noetherian $S'$-algebra. 
Consider the natural base change functor 
$$-\otimes_{S'}R:\ \Coh^B(\b) \rightarrow \Coh^B(\b\times_\t \Spec R).$$ 
By the discussions in \textsection \ref{subsect 4.2.4}, $\Coh^B(\b\times_\t \Spec R)$ can be viewed as the category $\sO$ for $\k[\n]\rtimes U\n$ with deformation ring $R$. 
By (\ref{equ 2.8}), the base change induces a homomorphism of the centers 
\begin{equation}\label{equ 4.10} 
-\otimes_{S'}R:\ Z(\Coh^B(\b)) \rightarrow Z(\Coh^B(\b\times_\t \Spec R)), 
\end{equation} 
which is an inclusion if so is $S'\rightarrow R$. 
If $\Spec R$ admits a $\C^{\times}$-action that is compatible with the one on $\t$, then similar statements hold for the category $\Coh^{B\times \C^{\times}}(\b\times_\t \Spec R)$. 

We denote by $\sP^{B,\leq \nu}_{R}$ the additive full subcategory generated by direct summands of  the $B$-equivariant sheaves 
$$\sO_{\b\times_\t \Spec R} \otimes \big(U\n/\bigoplus_{\lambda\nleq \nu-\mu} (U\n)_{\lambda}\big)\otimes \k_\mu ,\quad \mu\leq \nu$$ 
By \cite[\textsection 2.3.4]{Situ1}, $\sP^{B,\leq \nu}_{R}$ consists of projective objects in a truncation of $\Coh^B(\b\times_\t \Spec R)$. 

Note that the forgetful functor $\Coh^{B\times \C^{\times}}(\b) \rightarrow \Coh^{B}(\b)$ is a degrading functor in the sense of Appendix \ref{app A3}. 
The objects in $\sP^{B,\leq \nu}_{S'}$ admit liftings in $\Coh^{B\times \C^{\times}}(\b)$ that generate the full subcategory of projective objects in the corresponding truncated category of $\Coh^{B\times \C^{\times}}(\b)$. 
Hence by Lemma \ref{lem A.21}, there is a natural isomorphism 
\begin{equation}\label{equ 4.11} 
Z^\bullet(\Coh^{B\times \C^{\times}}(\b)) \xs Z(\Coh^{B}(\b)). 
\end{equation} 
We have a chain of algebra homomorphisms (see the notations in Appendix \ref{app A}) 
\begin{equation}\label{equ 4.12} 
H_{\check{T}}^\bullet(\Gr)^{\wedge}\rightarrow Z_{\sI}^{\pur}(\Gr)
\xrightarrow[\sim]{\text{Thm \ref{thm ABG}}} Z^\bullet(\Db\Coh^{B\times \C^{\times}}(\b))
\rightarrow Z^\bullet(\Coh^{B\times \C^{\times}}(\b)) ,
\end{equation} 
where the first arrow is because that the cohomology of $\Gr$ is pure, and the second isomorphism is by the equivalence $\widetilde{\Psi}$ in Theorem \ref{thm ABG} which exchanges $\langle 1 \rangle$ and $\langle 1 \rangle[1]$. 

\iffalse
\begin{lem} 
The inclusion in (\ref{equ 4.12}) is an equality. In particular, combining with (\ref{equ 4.11}), there is an isomorphism 
\begin{equation}\label{equ 4.13} 
Z(\Coh^{B}(\b)) \xs H_{\check{T}}^\bullet(\Gr)^{\wedge}.
\end{equation} 
\end{lem}
\begin{proof}
Note that the restriction $Z^\bullet(\Db\Coh^{B\times \C^{\times}}(\b)) \rightarrow Z^\bullet(\Coh^{B\times \C^{\times}}(\b))$ provides a retraction of the first inclusion in (\ref{equ 4.12}). 
Hence it is enough to show that the composition 
$$H_{\check{T}}^\bullet(\Gr)^{\wedge}\xleftarrow{\sim}Z^\bullet(\Db\Coh^{B\times \C^{\times}}(\b))\rightarrow Z^\bullet(\Coh^{B\times \C^{\times}}(\b))$$ 
is an injection. 
Indeed, by Theorem \ref{thm ABG} the equivalence $\widetilde{\Psi}$ sends $\sO_\b\otimes \k_\lambda$ to the standard sheaf $\nabla_\lambda$ for any $\lambda\in \Lambda^{++}$. 
Hence the composition 
$$H_{\check{T}}^\bullet(\Gr)^{\wedge}\rightarrow Z(\Coh^{B}(\b))\rightarrow \prod_{\lambda\in \Lambda^{++}}\End_{\Coh^{B}(\b)}(\sO_\b\otimes \k_\lambda)$$ 
coincides with 
$$H_{\check{T}}^\bullet(\Gr)^{\wedge}\xs Z_{\sI}^{\pur}(\Gr)\rightarrow 
\prod_{\lambda\in \Lambda^{++}}\Ext^\bullet_{D^{\mathrm{b},\mx}_\sI(\Gr)}(\nabla_\lambda) 
=\prod_{\lambda\in \Lambda^{++}} S',$$ 
which is by restriction on the $\check{T}$-fixed points $\{\delta_\lambda\}_{\lambda\in \Lambda^{++}}$ and is injective. 
\end{proof}
\fi

\begin{prop}\label{prop 4.11} 
The composition 
$$H_{\check{T}}^\bullet(\Gr)^{\wedge}\xrightarrow{(\ref{equ 4.11})\circ (\ref{equ 4.12})} Z(\Coh^{B}(\b)) \xrightarrow{-\otimes_{S'} S} Z(\Coh^B(\n\times \t_{\hat{0}})) 
\xlongequal{\text{Cor \ref{cor 4.8}}} Z(\sO^{[-\rho]}_{S})$$ 
coincides with the map $\bb_{[-\rho]}: H_{\check{T}}^\bullet(\Gr)^{\wedge} \rightarrow Z(\sO^{[-\rho]}_{S})$. 
\end{prop}
\begin{proof} 
Note that the restriction on the subfamily of fix points 
$$H_{\check{T}}^\bullet(\Gr)^\wedge_S\rightarrow 
\prod_{\lambda\in \Lambda^{++}} H_{\check{T}}^\bullet(\delta_{\lambda})\otimes_{H_{\check{T}}^\bullet(\pt)}S = \prod_{\lambda\in \Lambda^{++}}S$$ 
is an inclusion. 
It follows from Theorem \ref{thm 3.11} that the (partial) restriction 
$$\chi_{S}: Z(\sO^{[-\rho]}_{S}) \rightarrow 
\prod_{\lambda\in \Lambda^{++}}\End_{\sO_{S}}( M(-\rho+l\lambda)_S)= \prod_{\lambda\in \Lambda^{++}}S$$ 
is already an inclusion. 
Consider the following diagram, where the right square commutes, 
$$\begin{tikzcd}
H_{\check{T}}^\bullet(\Gr)^{\wedge} \arrow[r,hook,"\bb_{[-\rho]}"] \arrow[dr,"(\ref{equ 4.11})\circ (\ref{equ 4.12})"'] & Z(\sO^{[-\rho]}_{S}) 
\arrow[r,hook,"\chi_{S}"] & \prod_{\lambda\in \Lambda^{++}}\End_{\sO_{S}}( M(-\rho+l\lambda)_S )= \prod_{\lambda\in \Lambda^{++}}S \\ 
 & Z(\Coh^{B}(\b)) \ar[u,hook,"-\otimes_{S'}S"'] \arrow[r,hook] & \prod_{\lambda\in \Lambda^{++}}\End_{\Coh^{B}(\b)}(\sO_\b\otimes \k_\lambda) = \prod_{\lambda\in \Lambda^{++}} S' .\arrow[u,hook,"-\otimes_{S'}S"']
 \end{tikzcd}$$ 
By Theorem \ref{thm 3.11} and Theorem \ref{thm ABG}, the compositions $H_{\check{T}}^\bullet(\Gr)^{\wedge}\rightarrow \prod_{\lambda\in \Lambda^{++}}S$ in two ways above are both by restrictions on the $\check{T}$-fixed points. 
It follows that the left triangle commutes. 
\end{proof} 

%=================================
\subsubsection{Center of the Steinberg block}\label{subsect 4.3.3} 
%=================================
Similarly as (\ref{equ 4.11}) and (\ref{equ 4.12}), we have a chain of algebra homomorphisms  
\begin{equation}\label{equ 4.14}
H^\bullet(\Gr)^{\wedge}\rightarrow Z_{(\sI)}^{\pur}(\Gr) 
\xrightarrow[\sim]{\text{Thm \ref{thm ABG}}} Z^\bullet(\Db\Coh^{B\times \C^{\times}}(\n))
\rightarrow Z^\bullet(\Coh^{B\times \C^{\times}}(\n)) \xs  Z(\Coh^{B}(\n)). 
\end{equation} 

\begin{prop}\label{prop 4.12'}
The composition 
$$H^\bullet(\Gr)^{\wedge} \xrightarrow{(\ref{equ 4.14})} Z(\Coh^{B}(\n)) 
\xlongequal{\text{Cor \ref{cor 4.8}}} Z(\sO^{[-\rho]}_{\k})$$ 
coincides with the homomorphism $\overline{\bb}_{[-\rho]}: H^\bullet(\Gr)^{\wedge} \rightarrow Z(\sO^{[-\rho]}_{\k})$. 
\end{prop} 
\begin{proof} 
We show that the diagram 
\begin{equation}\label{equ 4.16} 
    \begin{tikzcd}[column sep=large]
    H_{\check{T}}^\bullet(\Gr)^{\wedge} \arrow[r,"(\ref{equ 4.11})\circ (\ref{equ 4.12})"] \arrow[d] 
    & Z(\Coh^{B}(\b)) \arrow[d,"i^*"] \arrow[r,"-\otimes_{S'}S"] 
    & Z(\sO^{[-\rho]}_{S}) \arrow[d,"-\otimes_{S}\k"] \\ 
    H^\bullet(\Gr)^{\wedge} \arrow[r,"(\ref{equ 4.14})"] & Z(\Coh^{B}(\n)) \arrow[r,"="]
    & Z(\sO^{[-\rho]}_{\k})
    \end{tikzcd}
\end{equation} 
commutes. 
Then the assertion will follows from the commutative diagram in Theorem \ref{thm 3.11} and Proposition \ref{prop 4.11}. 
It is clear that the right square of (\ref{equ 4.16}) commutes, so we have to show the commutativity of the left square. 
To that end, let $Q$ be any module in $\sP^{B,\leq \nu}_{S'}$ and fix a lifting (still denoted by $Q$) of it in $\Coh^{B\times \C^{\times}}(\b)$, then there is commutative diagram 
$$\begin{tikzcd}
    Z(\Coh^{B}(\b))\arrow[r]\arrow[d,"i^*"] &\bigoplus_{d} \Hom_{\Db\Coh^{B\times \C^{\times}}(\b)}(Q , Q\langle d \rangle) \arrow[d,"i^*=Li^*"] \\ 
    Z(\Coh^{B}(\n))\arrow[r] &\bigoplus_{d} \Hom_{\Db\Coh^{B\times \C^{\times}}(\n)}(Li^*Q , Li^*Q\langle d \rangle).\end{tikzcd}$$
On the other hand, there is a commutative diagram 
$$\begin{tikzcd}
    H_{\check{T}}^\bullet(\Gr)^{\wedge} \arrow[r]\arrow[d]  &\bigoplus_{d} \Hom_{D^{\mathrm{b},\mx}_{\sI}(\Gr)}(\widetilde{\Psi}(Q),\widetilde{\Psi}(Q)\langle d \rangle [d]) \arrow[d,"\for"] \\ 
    H^\bullet(\Gr)^{\wedge} \arrow[r] &\bigoplus_{d} \Hom_{D^{\mathrm{b},\mx}_{(\sI)}(\Gr)}(\Psi(Li^*Q),\Psi(Li^*Q)\langle d \rangle [d]) .
\end{tikzcd}$$
Any element in $Z(\Coh^{B}(\b))$, resp. in $Z(\Coh^{B}(\n))$, is determined by its restriction to the full subcategories $\sP^{B,\leq \nu}_{S'}$, resp. $\sP^{B,\leq \nu}_{\k}$, for all $\nu\in \Lambda$. 
It follows that the right square of (\ref{equ 4.16}) commutes. 
\end{proof} 

\begin{thm}\label{prop 4.12} 
There is an algebra isomorphism 
\begin{equation}\label{equ 4.15} 
\overline{\bb}_{[-\rho]}: H^\bullet(\Gr)^{\wedge} \xs Z(\sO^{[-\rho]}_{\k}). 
\end{equation}
\end{thm}
\begin{proof} 
By Proposition \ref{prop 4.12'}, it is equivalent to show that the map 
$$b:H^\bullet(\Gr)^{\wedge}\xrightarrow{(\ref{equ 4.14})} Z(\Coh^B(\n))=Z(\Coh^G(\widetilde{\sN}))$$ 
is an isomorphism. 
\iffalse
We prove this assertion by constructing a retraction $\mathbf{a}: Z(\Coh^B(\n))\rightarrow H^\bullet(\Gr)^{\wedge}$ of (\ref{equ 4.14}) and show that $\mathbf{a}$ is an injection. 
\fi
To that end, we consider $\sN^{\reg}$ the set of regular nilpotent elements in $\g$, and let $j:\sN^{\reg}\hookrightarrow \widetilde{\sN}$ be the natural inclusion. 
Consider the functor $j_*:\Coh^G(\sN^{\reg})\rightarrow \QCoh^G(\widetilde{\sN})$, which is a full embedding since $j^*j_*$ is the identity. 
It yields a homomorphism of centers 
$$z_j: Z(\QCoh^G(\widetilde{\sN}))\rightarrow Z(\Coh^G(\sN^{\reg})).$$ 
Note that any sheaf in $\QCoh^G(\widetilde{\sN})$ is the union of its coherent subsheaves, hence the center of $\QCoh^G(\widetilde{\sN})$ is uniquely determined by its restriction on $\Coh^G(\widetilde{\sN})$, namely we can identify  $Z(\QCoh^G(\widetilde{\sN}))=Z(\Coh^G(\widetilde{\sN}))$. 
We claim that the map $z_j$ is an injection. 
Indeed, for any $\sF\in \Coh^G(\widetilde{\sN})$ we have a commutative diagram 
\begin{equation}\label{equ 3.17'}
	\begin{tikzcd}
    Z(\Coh^G(\widetilde{\sN})) \arrow[r,"z_j"]\arrow[d] 
    &Z(\Coh^G(\sN^{\reg})) \arrow[d] \\ 
    \End(\sF) \arrow[r,"j^*"] 
    &\End(j^*\sF).\end{tikzcd}
\end{equation}
If $\sF$ is a torsion-free sheaf on $\widetilde{\sN}$, the lower horizontal map in (\ref{equ 3.17'}) is an injection. 
Note that any $\sF$ in the full subcategory $\sP^{B,\leq \nu}_{\k}$ of $\Coh^B(\n)$ corresponds to a vector bundle in $\Coh^G(\widetilde{\sN})$, which is torsion-free. 
It follows that $\ker(z_j)$ vanishes in $Z(\sP^{B,\leq \nu}_{\k})$ for all $\nu$, so $\ker(z_j)=0$. 

Fix a regular nilpotent element $x$ in $\g$ and let $G^x$ be its stabilizer in $G$. 
Taking the fiber at $x$ gives an equivalence from $\Coh^G(\sN^{\reg})$ to $\rep(G^x)$. 
It is known (see e.g. \cite[Thm 6.1]{YZ}) that $G^x$ is commutative and $G^x=Z(G)\times G^x_u$, where $G^x_u$ is the unipotent radical of $G^x$. 
Hence $G^x_u$ is a vector group, and in particular $\g^x_u:=\Lie(G^x_u)$ is abelian. 
We thus identify the categories 
$$\rep(G^x)=\rep(Z(G))\boxtimes \rep(G^x_u)=\Vect^{X^*(Z(G))}_{\k}\boxtimes U\g^x_u\mod^{\text{nil}},$$ 
where $\Vect^{X^*(Z(G))}_{\k}$ is the category of $X^*(Z(G))$-graded vector spaces and $U\g^x_u\mod^{\text{nil}}$ is the category of nilpotent $U\g^x_u$-modules. 
Therefore we have 
$$Z(\Coh^G(\sN^{\reg}))=Z(\rep(G^x))=\big((U\g^x_u)^{\wedge}\big)^{\prod X^*(Z(G))}$$
where $(U\g^x_u)^{\wedge}$ is the completion of $U\g^x_u$ at the augmentation ideal. 
In sum, we have algebra homomorphisms 
\begin{equation}\label{equ 3.16'}
H^\bullet(\Gr)^{\wedge}\xrightarrow{b} Z(\Coh^G(\widetilde{\sN}))\xrightarrow{z_j} 
Z(\Coh^G(\sN^{\reg}))= \big((U\g^x_u)^{\wedge}\big)^{\prod X^*(Z(G))}.
\end{equation} 

Using the geometric Satake equivalence, Ginzburg \cite[Prop 1.7.2]{Gin95} (see also \cite[Cor 6.4]{YZ}) constructed an algebra isomorphism between $H^\bullet(\Gr)$ and $U\g^x_u$ when $G$ is of adjoint type. 
It induces an isomorphism of their completions $H^\bullet(\Gr)^\wedge$ and $(U\g^x_u)^{\wedge}$. 
For general $G$, it gives an isomorphism 
$$H^\bullet(\Gr)^{\wedge} \xs \big((U\g^x_u)^{\wedge}\big)^{\prod X^*(Z(G))}.$$ 
By the compatibility of Theorem \ref{thm ABG} and the geometric Satake equivalence, the map above coincides with the composition of (\ref{equ 3.16'}), showing that the latter is an isomorphism. 
Since we showed that $z_j$ is an injection, $z_j$ and thus $b$ are isomorphisms. 
\end{proof} 

%=================================
\subsubsection{Another description}\label{subsect 3.3.4} 
%=================================
The isomorphism (\ref{equ 4.15}) restricts to an isomorphism 
\begin{equation}\label{equ 4.26} 
\overline{\bb}_{-\rho}: H^\bullet(\Gr^\circ)^{\wedge} \xs Z(\sO^{-\rho}_{\k}). 
\end{equation} 
In this subsection, we find another description for $Z(\sO^{-\rho}_{\k})$ which is independent of the result in \textsection \ref{subsect 4.3.3}, and it will be used in the next section. 

The equivalence in Theorem \ref{thm 4.7} restricts to an equivalence of the blocks 
$$\sO^{-\rho}_{\k}\xs \sO^{0}_{1,\k},$$ 
where $\sO^{0}_{1,\k}=U^\hb_1\Mod^{l\rQ}_\k\cap \sO_{1,\k}$. 
Under the equivalence, the module $Q(-\rho+l\lambda)_\k^{\leq -\rho+l\mu}$ corresponds to the $U^\hb_1$-module 
$$\fQ(\lambda)^{\leq \mu}_\k:=(\k[B]\rtimes U\n)\otimes_{\k[T]\otimes U\n}\big((U\n/\bigoplus_{\nu\nleq \mu-\lambda} (U\n)_{\nu})\otimes \k_{\lambda}\big), $$
where we use the natural $\rQ$-grading on $U^\hb_1$, and $\k_\lambda$ is a trivial $\k[T]\otimes U\n$-module recording the degree shift. 
Note that $\fQ(\lambda)^{\leq \mu}_\k$ is a cyclic $U^\hb_1$-module generated by the element $1_{\lambda}^{\mu}:=1_{-}\otimes 1_+\otimes 1_{\lambda}$, where $1_{-}, 1_+, 1_{\lambda}$ are the identities of $\k[N]$, $U\n$ and $\k_\lambda$, respectively. 
For any nonzero homogenous elements $e\in (U\n)_{\nu}$, $\varphi\in \k[\n]_{-\nu}$ for $\nu\geq 0$, we define the morphisms 
\begin{equation}\label{equ 4.21} 
\begin{aligned}
\iota_e: \fQ(\lambda+\nu)^{\leq \mu}_\k \hookrightarrow \fQ(\lambda)^{\leq \mu}_\k, \quad 1_{\lambda+\nu}^{\mu} \mapsto 1_{-}\otimes e\otimes 1_{\lambda}, \\ 
\iota^-_\varphi: \fQ(\lambda-\nu)^{\leq \mu}_\k \rightarrow \fQ(\lambda)^{\leq \mu}_\k, \quad 1_{\lambda-\nu}^{\mu} \mapsto \varphi\otimes 1_+\otimes 1_{\lambda}. 
\end{aligned} 
\end{equation} 
Denote by $\sP_\k^{\leq \mu}$ the additive full subcategory of $\sO_{1,\k}$ generated by $\fQ(\lambda)^{\leq \mu}_\k$. 
\begin{lem}\label{lem 4.14} 
Morphisms in the category $\sP_\k^{\leq \mu}$ are generated by the $\iota_{e}$'s and $\iota^-_\varphi$'s. 
\end{lem} 
\begin{proof} 
Let $\psi\in \Hom(\fQ(\lambda)^{\leq \mu}_\k,\fQ(\lambda')^{\leq \mu}_\k)$. 
Write $\psi(1_{\lambda}^{\mu})= \sum_{s=1}^n \varphi'_s \otimes e'_s \otimes 1_{\lambda'}$ for some homogenous elements $\varphi'_s\in \k[N]$, $e'_s\in  U\n$ such that $\deg \varphi_s' +\deg e_s' =\lambda-\lambda'$. 
Since $\psi$ is determined by the image of $1_{\lambda}^{\mu}$, it follows that $\psi= \sum_{s=1}^n \iota_{e'_s} \circ \iota^-_{\varphi'_s}$. 
\end{proof} 

For any $\lambda\in \rQ$ and $\mu_2\geq \mu_1\geq \lambda$, there is an algebra homomorphism 
$$\tau^{\leq \mu_1}: \End(\fQ(\lambda)^{\leq \mu_2}_\k)\rightarrow \End(\fQ(\lambda)^{\leq \mu_1}_\k)$$ 
given by the truncation $\epsilon^{\leq \mu_1}:\fQ(\lambda)^{\leq \mu_2}_\k \rightarrow \fQ(\lambda)^{\leq \mu_1}_\k$. 
They form a limit 
$\Lim{\mu\geq\lambda} \End(\fQ(\lambda)^{\leq \mu}_\k)$.  
The natural restriction $Z(\sO^{0}_{1,\k})\rightarrow \End(\fQ(\lambda)^{\leq \mu}_\k)$ yields a homomorphism 
\begin{equation}\label{equ 4.17} 
Z(\sO^{0}_{1,\k})\rightarrow \Lim{\mu\geq\lambda} \End(\fQ(\lambda)^{\leq \mu}_\k). 
\end{equation} 

Let us compute the algebra $\Lim{\mu\geq\lambda} \End(\fQ(\lambda)^{\leq \mu}_\k)$. 
We define a module in $U_1^\hb\Mod_\k^{l\Lambda}$ by 
$$\fQ(\lambda)_\k=(\k[B]\rtimes U\n)\otimes_{\k[T]\otimes U\n} (U\n\otimes \k_{\lambda}) ,$$ 
then any $\fQ(\lambda)^{\leq \mu}_\k$ is a truncation for $\fQ(\lambda)_\k$. 
We take the projective limit $\widehat{\fQ}(\lambda)_\k:=\Lim{\mu\geq \lambda} \fQ(\lambda)^{\leq \mu}_\k$ in the category $U_1^\hb\Mod_\k^{\Lambda}$. 
There are $\k$-linear isomorphisms 
\begin{equation}\label{equ 4.18} 
\begin{aligned}
\Lim{\mu\geq \lambda} \End(\fQ(\lambda)^{\leq \mu}_\k)&=\Lim{\mu\geq \lambda}\Hom(\fQ(\lambda)_\k, \fQ(\lambda)^{\leq \mu}_\k) \\ 
&=\Hom(\fQ(\lambda)_\k, \widehat{\fQ}(\lambda)_\k)\\ 
&=\widehat{\fQ}(\lambda)_{\k,\lambda}=\big(\prod_{\nu\geq 0} \k[N]_{-\nu}\otimes (U\n)_{\nu}\big) \otimes \k_{\lambda}, 
\end{aligned} 
\end{equation} 
where the last equality is by identifying $\fQ(\lambda)^{\leq \mu}_\k=\k[N]\otimes \big(U\n/\bigoplus_{\nu\nleq \mu-\lambda} (U\n)_{\nu}\big)\otimes \k_{\lambda}$ as $\k$-vector spaces. 
Consider the algebra 
\begin{equation*} 
\widehat{\k[N]\rtimes U\n} :=\prod_{\nu_1,\nu_2\geq 0} \k[N]_{-\nu_1}\otimes (U\n)_{\nu_2}, 
\end{equation*} 
whose algebra structure is induced from the one of $\k[N]\rtimes U\n$. 
One can check that (\ref{equ 4.18}) gives an isomorphism of $\k$-algebras 
\begin{equation}\label{equ 4.20} 
\Lim{\mu\geq \lambda} \End(\fQ(\lambda)^{\leq \mu}_\k)^{\op}= (\widehat{\k[N]\rtimes U\n})_0 . 
\end{equation}

\begin{lem}\label{lem 4.13} 
The map (\ref{equ 4.17}) induces an algebra isomorphism 
$$Z(\sO^{0}_{1,\k})\xs Z\big(\Lim{\mu\geq\lambda} \End(\fQ(\lambda)^{\leq \mu}_\k)\big). $$ 
Therefore, the same construction leads to an isomorphism 
$$ Z(\sO^{-\rho}_{\k})\xs Z\big(\Lim{\mu\geq\lambda} \End(Q(-\rho+l\lambda)_\k^{\leq -\rho+l\mu})\big). $$ 
\end{lem} 
\begin{proof} 
Any $z\in Z(\sO^{0}_{1,\k})$ gives a collection of central elements 
$$z_{\lambda}\in \Lim{\lambda\leq\mu} \End(\fQ(\lambda)^{\leq \mu}_\k)=(\widehat{\k[N]\rtimes U\n})_0, \quad \forall \lambda \in \rQ.$$ 
Since $z$ commutes with $\iota_e$ and $\iota^-_\varphi$, we have the following equalities 
\begin{equation}\label{equ 4.22} 
ez_{\lambda+\nu}.1_{\lambda}^{\mu}=z_{\lambda}e.1_{\lambda}^{\mu}, \quad 
\varphi z_{\lambda-\nu}.1_{\lambda}^{\mu}=z_{\lambda}\varphi .1_{\lambda}^{\mu}, \quad 
\forall \lambda,\mu \in \rQ. 
\end{equation} 
Therefore as elements in the algebra $\widehat{\k[N]\rtimes U\n}$, we have 
\begin{equation}\label{equ 4.23} 
ez_{\lambda+\nu}=z_{\lambda}e, \quad \varphi z_{\lambda-\nu}=z_{\lambda}\varphi . 
\end{equation} 
For any $i\in\I$, we choose a non zero element $\varphi_i\in \k[N]_{-\alpha_i}$. 
Note that $\varphi_i$ is central in $\k[N]\rtimes U\n$, so is it in $\widehat{\k[N]\rtimes U\n}$. 
It follows that $\varphi_{i}z_{\lambda-\alpha_i}=z_{\lambda-\alpha_i}\varphi_{i}=\varphi_{i}z_{\lambda}$. 
Since $\varphi_i$ is torsion free in $\widehat{\k[N]\rtimes U\n}$, we deduce that $z_{\lambda-\alpha_i}=z_{\lambda}$ for any $i\in \I$. 
In other words, the function $\lambda\mapsto z_{\lambda}$ is constant. 
Since $z$ is determined by the family $\{z_\lambda\}_{\lambda}$, the restriction map 
\begin{equation}\label{equ ResCent} 
Z(\sO^{0}_{1,\k})\rightarrow Z\big(\Lim{\lambda\leq \mu} \End(\fQ(\lambda)_\k^{\leq \mu})\big) = Z\big((\widehat{\k[N]\rtimes U\n})_0\big) 
\end{equation} 
is injective for any $\lambda\in \rQ$. 

Now we show the surjectivity. 
Let $z'\in Z((\widehat{\k[N]\rtimes U\n})_0)$. 
Since $z'$ commutes with $\varphi_i\otimes e_{\alpha_i}\in (\k[N]\rtimes U\n)_0$, and $\varphi_i$ is central and torsion free, it follows that $z'e_{\alpha_i}=e_{\alpha_i}z'$. 
So $z'$ commutes with $U\n$. 
For any $\varphi\in \k[N]_{-\nu}$, we pick any nonzero element $e\in (U\n)_{\nu}$, then $z'$ commutes with $\varphi\otimes e$. 
Since $e$ is left-torsion free in $\widehat{\k[N]\rtimes U\n}$, $z'$ commutes with $\varphi$. 
Hence (\ref{equ 4.23}) holds for the constant family $\{z'\}_{\lambda}$. 
As in (\ref{equ 4.22}), $z'$ commutes with the morphisms $\iota_{e}$ and $\iota^-_\varphi$. 
By the Lemma \ref{lem 4.14}, $z'$ defines an element in $Z(\sP^{\leq \mu}_\k)$ for each $\mu$. 
Since $M\in \sO^{0}_{1,\k}$ admits a resolution in $\sP^{\leq \mu}_\k$ for some $\mu\in \rQ$, the element $z'\in Z(\sP^{\leq \mu}_\k)$ defines an endomorphism $z'_M\in \End(M)$. 
It gives a well-defined element $z'=(z'_M)_M$ in $Z(\sO^{0}_{1,\k})$. 
Hence (\ref{equ ResCent}) is a surjection. 
\end{proof}

\begin{corollary}\label{cor 3.15} 
There is an isomorphism 
\begin{equation}\label{equ 3.25} 
Z(\sO^{-\rho}_{\k})\xs Z\big((\widehat{\k[N]\rtimes U\n})_0\big). 
\end{equation}
\end{corollary}

\begin{rmk} 
The two descriptions (\ref{equ 4.26}) and (\ref{equ 3.25}) are compatible in the following way. 
Let $x\in \n\simeq N$ be a regular element, and denote by $\n^x$ the centralizer of $x$ in $\n$. 
Recall that Ginzburg \cite[Prop 1.7.2]{Gin95} constructed an algebra isomorphism $H^\bullet(\Gr^\circ)\simeq U\n^x$. 
The evaluation on $x$ gives a linear map $\ev_x:\k[N]\rtimes U\n\rightarrow U\n$. 
We have a commutative diagram 
$$\begin{tikzcd}
Z(\sO^{-\rho}_{\k})\arrow[r,"\simeq","(\ref{equ 3.25})"'] & Z\big((\widehat{\k[N]\rtimes U\n})_0\big)\arrow[d,"\ev_x"] \\ 
H^\bullet(\Gr^\circ)^\wedge \arrow[r,"\simeq"] \arrow[u,"(\ref{equ 4.26})","\simeq"'] 
& (U\n^x)^\wedge,
\end{tikzcd}$$
where $(U\n^x)^\wedge$ is the completion of $U\n^x$ at the augmentation ideal. 
\end{rmk}

%%%%%%%%%%%%%%%%%%%%%%%%%
\section{Center of principal block}\label{sect 4} 
%%%%%%%%%%%%%%%%%%%%%%%%%
In this section, we study the principal block $\sO^{0}_\k$, and show that the algebra homomorphism 
$$\overline{\bb}_0:H^\bullet(\Fl^\circ)^\wedge\rightarrow Z(\sO^{0}_{\k})$$ 
is an isomorphism. 

%---------------------------------------------------------------
\subsection{Translation functors}\label{subsect 4.0} 
%---------------------------------------------------------------
Let $R$ be a commutative Noetherian $S$-algebra. 
For $\omega_1,\omega_2 \in \Xi_\sc$, there is a unique dominant weight $\nu$ in $W(\omega_2-\omega_1)$. 
Denote by $V(\nu)_q$ the Weyl module for $U_q$ of highest weight $\nu$, and $V(\nu)_\k$ its specialization at $q_e=\zeta_e$. 
Recall that the \textit{translation functors} are given by 
$$\sfT_{\omega_1}^{\omega_2} : \sO^{ \omega_1}_{R} \rightarrow \sO^{ \omega_2}_{R}, \quad M \mapsto \pr_{\omega_2} (M \otimes V(\nu)_\k),$$ 
$$\sfT^{\omega_1}_{\omega_2} : \sO^{ \omega_2}_{R} \rightarrow \sO^{ \omega_1}_{R}, \quad M \mapsto \pr_{\omega_1} (M \otimes V(\nu)^*_\k),$$ 
where $\pr_{\omega_i}$ is the natural projection to the block $\sO^{ \omega_i}_{R}$. 

\begin{lem}[{\cite[II \textsection 7.8]{Jan03} and \cite[Prop 3.9]{Situ1}}] \label{lem 5.0}\ 
\begin{enumerate} 
\item $\sfT_{\omega_1}^{\omega_2}$ and $\sfT^{\omega_1}_{\omega_2}$ are exact and biadjoint to each other. 
\item For any $x\in W_{l,\af}$, the module $\sfT_{\omega_1}^{\omega_2}M(x\bullet \omega_1)_R$ admits Verma factors $M(xy\bullet \omega_2)_R$, where $y$ runs through a system of representatives for 
$$W_{l,\omega_1}/W_{l,\omega_1}\cap W_{l,\omega_2}.$$ 
\item Suppose that $\omega_2$ is contained in the closure of the $\omega_1$-facet, i.e. $W_{l,\omega_1}\subseteq W_{l,\omega_2}$. Then there is a natural isomorphism 
$$\Upsilon_{\omega_2}^{\omega_1}:\ \id^{\oplus |W_{l,\omega_2}/W_{l,\omega_1}|} \xs \sfT_{\omega_1}^{\omega_2} \sfT^{\omega_1}_{\omega_2}$$ 
of functors on $\sO^{\omega_2}_{R}$. 
\end{enumerate} 
\end{lem} 
\begin{rmk}\label{rmk 4.2}
Although there might be other choices, we will always use the biadjunction of $(\sfT_{\omega_1}^{\omega_2}, \sfT^{\omega_1}_{\omega_2})$ given by the the isomorphism $V(\nu)_q\xs V(\nu)^{**}_q$ via $K_{2\rho}$-action. 
\end{rmk} 

%---------------------------------------------------------------
\subsection{New truncation}\label{subsect 4.2} 
%---------------------------------------------------------------
Recall the order $\uparrow$ on $\Lambda$ defined in \textsection\ref{subsect 2.3.2}. 
In this subsection, we construction a truncation of $\sO_S$ by the order $\uparrow$, which refines the truncation discussed in \textsection\ref{subsect 2.3.3new}. 
The advantage is that this new truncation is more compatible with the translation functors $\sfT_0^{-\rho}$ and $\sfT^0_{-\rho}$, see Lemma~\ref{lem 5.1}. 

Till the end of this subsection, we let $R=S$ or $\C$. 
\begin{lem}\label{lem new5.8.0}
Let $\mu,\lambda\in \Lambda$. 
We have 
\begin{equation}\label{equ new5.2.0}
\Hom_{\sO_R}(M(\mu)_R, M(\lambda)_R)\neq 0 \quad \text{only if}\quad \mu\uparrow \lambda, 
\end{equation}
and 
\begin{equation}\label{equ new5.3.0}
\Ext^1_{\sO_S}(M(\mu)_R, M(\lambda)_R)\neq 0 \quad \text{only if}\quad \mu\uparrow \lambda \ \text{and}\ \mu\neq \lambda . 
\end{equation}
\end{lem}
\begin{proof} 
We firstly show \eqref{equ new5.2.0}. 
Suppose $R=S$ and let $\K$ be the fraction field of $S$. 
Since $M(\lambda)_S$ is free over $S$, we have a natural inclusion $\Hom_{\sO_S}(M(\lambda)_S, M(\mu)_S)\subset \Hom_{\sO_\K}(M(\lambda)_\K, M(\mu)_\K)$. 
By \cite[Lem 3.5]{Situ1} the category $\sO_\K$ is semi-simple, whose simple objects are Verma modules. 
Hence $\Hom_{\sO_\K}(M(\mu)_\K, M(\lambda)_\K)$ $=0$ if $\mu\neq \lambda$. 
Now assume $R=\C$. 
If $\Hom_{\sO_\C}(M(\mu)_\C, M(\lambda)_\C)$ $\neq 0$ then $L(\mu)_\C$ appears as a factor in $M(\lambda)_\C$. 
By the the linkage principle \eqref{equ LP} we have $\mu\uparrow \lambda$. 
It proves (\ref{equ new5.2.0}). 

For (\ref{equ new5.3.0}), recall the standard fact that $\Ext^1_{\sO_R}(M(\mu)_R, M(\lambda)_R)\neq 0$ only if $\mu<\lambda$, see e.g. \cite[Prop 3.1]{Hum08}. 
Thus we may assume $\mu <\lambda$. 
Then any extension of $M(\mu)_R$ and $M(\lambda)_R$ are contained in $\sO_R^{\leq \lambda}$, hence we only need to compute $\Ext^1$ in the category $\sO_R^{\leq \lambda}$. 
By the linkage principle \eqref{equ LP} and BGG reciprocity (Lemma~\ref{lem 3.2}(3)), $M(\mu)_R$ admits a resolution by projective modules in $\sO_R^{\leq \lambda}$ that are composed by $M(\nu)_R$ with $\mu\uparrow \nu$. 
Now (\ref{equ new5.3.0}) follows from (\ref{equ new5.2.0}). 
\end{proof}

Let $\nu\in \Lambda$. 
We set 
\[
\sO^{\uparrow \nu}_R
\]
as the the full subcategory of modules $M$ in $\sO_R$ that admit a surjection $Q\twoheadrightarrow M$ from a module $Q$ admitting a Verma flag with factors $M(\lambda)_R$ with $\lambda\uparrow \nu$. 
Since $\Lambda$ is the union of the poset ideals of the form $\{\lambda \in \Lambda|\lambda \uparrow \nu \}$, any module in $\sO_R$ is a direct sum of submodules in $\sO^{\uparrow \nu}_R$ for some $\nu\in \Lambda$. 
If $\nu\in W_{l,\af}\bullet \omega$ for $\omega\in \Xi_\sc$, then $\sO^{\uparrow \nu}_R$ is contained in the block $\sO^\omega_R$. 

\begin{lem}\label{lem new5.9.0}
There is a truncation functor 
$$\tau^{\uparrow \nu}:\ U^\hb_\zeta\mod^{\Lambda}_R \rightarrow \sO^{\uparrow \nu}_R$$ 
by taking the maximal quotient in $\sO^{\uparrow \nu}_R$, which is left adjoint to the natural inclusion. 
\end{lem}
\begin{proof}
Since $\sO^{\uparrow \nu}_R$ is contained in $\sO^{\leq \nu}_R$, any morphism from $M\in U^\hb_\zeta\mod^{\Lambda}_R$ to a module in $\sO^{\uparrow \nu}_R$ factors through $\tau^{\leq \nu}(M)$. 
Hence it is enough to define the functor 
\begin{equation}\label{equ new5.5.0}
\tau^{\uparrow \nu}:\ \sO^{\leq \nu}_R \rightarrow \sO^{\uparrow \nu}_R. 
\end{equation}
Let $Q\in \sO^{\leq \nu}_R$ be a projective object. 
It admits a Verma flag, and by (\ref{equ new5.3.0}) we can define the quotient 
$$\tau^{\uparrow \nu}(Q)$$ 
of $Q$ by the submodule composed by the Verma factors not containing in $\{M(\mu)_R\}_{\mu \uparrow \nu}$. 
Let $Q'\in \sO_R$ admitting a Verma flag with factors in $\{M(\mu)_R\}_{\mu \uparrow l\nu}$, and let $Q'\twoheadrightarrow M'$ be a surjection. 
Since $Q$ is projective in $\sO^{\leq \nu}_R$, any morphism from $Q$ to $M'$ can be lifted to $Q'$. 
By (\ref{equ new5.2.0}) any morphism from  $Q$ to $Q'$ factors through $\tau^{\uparrow \nu}(Q)$. 
In sum, any morphism from $Q$ to $M'$ factors through $\tau^{\uparrow \nu}(Q)$, which thus is the maximal quotient of $Q$ in $\tau^{\uparrow \nu}(Q)$. 

In general, let $M\in \sO^{\leq \nu}_R$, and we choose a resolution $Q_2\rightarrow Q_1\rightarrow M\rightarrow 0$ with projective objects $Q_i$ ($i=1,2$) in $\sO^{\leq \nu}_R$. 
Then we set 
$$\tau^{\uparrow \nu}(M):=\mathrm{coker}\big(\tau^{\uparrow \nu}(Q_2)\rightarrow \tau^{\uparrow \nu}(Q_1)\big).$$ 
Then $\tau^{\uparrow \nu}(M)$ is contained in $\sO^{\uparrow \nu}_R$. 
For any $M'\in \sO^{\uparrow \nu}_R$, we have a commutative diagram with exact rows 
$$\begin{tikzcd}
0\arrow[r] & \Hom(\tau^{\uparrow \nu}(M), M') \arrow[d]\arrow[r] 
& \Hom(\tau^{\uparrow \nu}(Q_1),M') \arrow[d,equal] \arrow[r] 
& \Hom(\tau^{\uparrow \nu}(Q_2),M') \arrow[d,equal]\\ 
0\arrow[r] & \Hom(M,M') \arrow[r] 
& \Hom(Q_1,M') \arrow[r] 
& \Hom(Q_2,M'). 
\end{tikzcd}$$ 
Hence the left vertical map is an isomorphism, which shows that $\tau^{\uparrow \nu}(M)$ is the maximal quotient of $M$ in $\sO^{\uparrow \nu}_R$. 
It gives the desired functor. 
\end{proof}

For any $\lambda,\nu\in \Lambda$, we abbreviate $Q(\lambda)_R^{\uparrow \nu}=\tau^{\uparrow \nu}(Q(\lambda)_R)$. 

\begin{lem}\label{lem new5.17}
\begin{enumerate}
\item The category $\sO^{\uparrow \nu}_R$ is a Serre subcategory in $\sO_R$. 
\item For ${\lambda\uparrow \nu}$, the module $Q(\lambda)_R^{\uparrow \nu}$ is the projective cover of $E(\lambda)_\C$ in $\sO^{\uparrow \nu}_R$. 
Each projective object in $\sO^{\uparrow \nu}_R$ admits a Verma flag with factors of the form $M(\mu)_R$, $\mu \uparrow \nu$. 
Moreover we have 
\begin{equation}\label{equ BGG2}
(Q(\lambda)_R^{\uparrow \nu}:M(\mu)_R)=[M(\mu)_\C:E(\lambda)_\C], \quad \forall \lambda,\mu \uparrow \nu. 
\end{equation}
\end{enumerate}
\end{lem}
\begin{proof}
(1) By definition $\sO^{\uparrow \nu}_R$ is closed under taking quotient modules. 
Let $M$ be a submodule of $M'\in \sO^{\uparrow \nu}_R$. 
Then the inclusion $M\hookrightarrow M'$ factors through $\tau^{\uparrow \nu}(M)$, hence $M=\tau^{\uparrow \nu}(M)$. 
So $\sO^{\uparrow \nu}_R$ is also closed under taking submodules. 

Now we show that $\sO^{\uparrow \nu}_R$ is closed under extension. 
Let $0\rightarrow M_1\rightarrow M\rightarrow M_2\rightarrow 0$ be a short exact sequence in $\sO_R$ with $M_1,M_2\in \sO^{\uparrow \nu}_R$. 
Then $M\in \sO^{\leq \nu}_R$, and we can choose a surjection $Q\twoheadrightarrow M$ from a projective module $Q\in \sO^{\leq \nu}_R$. 
We have short exact sequence $0\rightarrow Q'\rightarrow Q\rightarrow \tau^{\uparrow \nu}(Q)\rightarrow 0$, where $Q'$ is the submodule of $Q$ composed by the Verma factors not containing in $\{M(\mu)_R\}_{\mu \uparrow \nu}$. 
Then $\tau^{\uparrow \nu}(Q')=0$, so $\Hom(Q',M_i)=0$ ($i=1,2$). 
It follows that $\Hom(Q',M)=0$. 
Hence the surjection $Q\twoheadrightarrow M$ factors through $\tau^{\uparrow \nu}(Q)\twoheadrightarrow M$, which implies that $M\in \sO^{\uparrow \nu}_R$. 

(2) By Lemma~~\ref{lem 3.2}(2) and discussions in Lemma~\ref{lem new5.9.0}, we have $Q(\lambda)_R^{\uparrow \nu}=\tau^{\uparrow \nu}\circ \tau^{\leq \nu}(Q(\lambda)_R)= \tau^{\uparrow \nu}(Q(\lambda)^{\leq \nu}_R)$, which is the quotient of $Q(\lambda)^{\leq \nu}_R$ by the submodule composed by Verma factors not containing in $\{M(\mu)_R\}_{\mu\uparrow \nu}$. 
Since $\tau^{\uparrow \nu}: \sO^{\leq \nu}_R \rightarrow \sO^{\uparrow \nu}_R$ is left adjoint to the (exact) inclusion functor, $Q(\lambda)_R^{\uparrow \nu}$ is projective in $\sO^{\uparrow \nu}_R$, and we have 
$$\Hom(Q(\lambda)_R^{\uparrow \nu},E(\mu)_\C)\simeq \Hom(Q(\lambda)_R^{\leq \nu},E(\mu)_\C)=\delta_{\lambda,\mu}\C$$
for any $\mu\uparrow \nu$. 
It shows that $Q(\lambda)_R^{\uparrow \nu}$ is the projective cover of $E(\lambda)_\C$ in $\sO^{\uparrow \nu}_R$. 
This implies the first two assertions. 
And \eqref{equ BGG2} follows from the linkage principle \eqref{equ LP} and BGG reciprocity (Lemma~\ref{lem 3.2}(3)). 
\end{proof}

\begin{lem}\label{lem new5.10}
Let $\nu,\mu \in \Lambda$ and $w\in W$. 
\begin{enumerate}
\item We have $(w\bullet 0 +l\mu) \uparrow l\nu$ if and only if $\mu\leq \nu$. 
Therefore, $\{\lambda \in \Lambda|\lambda \uparrow l\nu \}= \{w\bullet 0 +l\mu\}_{\mu\leq \nu, w\in W}$. 
\item We have $\{\lambda \in \Lambda|\lambda \uparrow (-\rho+l\nu) \}= \{-\rho +l\mu\}_{\mu\leq \nu}$. 
\end{enumerate}
\end{lem}
\begin{proof}
(1) Consider the identifications $W_\af \simeq W_{\af,l} \simeq W_{\af,l}\bullet 0$. 
The order $\uparrow$ on $W_{\af,l}\bullet 0$ defines an order on $W_\af$, which is clearly independent on $l$. 
Hence we may assume that $l>\langle 2\rho, \check{\varpi}_{i}\rangle$ for any fundamental coweight $\check{\varpi}_{i}$ associated to simple root $\alpha_i$. 
If $(w\bullet 0 +l\mu) \uparrow l\nu$, then we have $(w\bullet 0 +l\mu) \leq l\nu$, hence $-2\rho+l\mu\leq l\nu$. 
Our assumption on $l$ forces that $\mu\leq \nu$. 
On the other hand, we have $w\bullet 0 \uparrow 0$ and $0\uparrow l\eta$ for any $\eta\geq 0$, which implies the ``if" part. 

(2) It is enough to notice that $-\rho \uparrow -\rho+l\eta$ for any $\eta\geq 0$. 
\end{proof}

From now on, we abbreviate $\bm{\nu}:=-\rho+l\nu$ for any $\nu\in \rQ$.  
By Lemma~\ref{lem new5.10}(2), we have $\sO_R^{\uparrow \bm{\nu}}=\sO_R^{-\rho,\leq \bm{\nu}}$ and $Q(\bm{\lambda})_R^{\uparrow \bm{\nu}}=Q(\bm{\lambda})_R^{\leq \bm{\nu}}$, for any $\lambda\leq \nu$. 

\begin{lem}\label{lem 5.1} 
Let $\nu\in \rQ$. 
\begin{enumerate} 
\item For any $w\in W$, we have $\sfT_0^{-\rho} M(w\bullet 0+l\nu)_R= M(\bm{\nu})_R$. 
The module $\sfT^0_{-\rho} M(\bm{\nu})_R$ admits Verma factors $M(w\bullet 0+l\nu)_R$, where $w\in W$ and each appears once. 
\item The translation functors $\sfT_0^{-\rho}$ and $\sfT^0_{-\rho}$ restrict on the truncated categories 
$$\sfT^0_{-\rho}:\sO^{-\rho, \leq \bm{\nu}}_{R}\rightarrow \sO^{\uparrow l\nu}_{R},
\quad \sfT_0^{-\rho}:\sO^{\uparrow l\nu}_{R} \rightarrow \sO^{-\rho, \leq \bm{\nu}}_{R}.$$ 
\item There are natural isomorphisms 
$$\sfT_0^{-\rho}\circ \tau^{\uparrow l\nu}=\tau^{\leq \bm{\nu}}\circ \sfT_0^{-\rho}, \quad \tau^{\uparrow l\nu}\circ \sfT^0_{-\rho}=\sfT^0_{-\rho}\circ \tau^{\leq \bm{\nu}}$$ 
of functors on $\sO^0_R$ and $\sO^{-\rho}_R$, respectively. 
\end{enumerate} 
\end{lem} 
\begin{proof} 
(1) It is a special case of Lemma \ref{lem 5.0}(2). 

(2) We prove the assertion for $\sfT^0_{-\rho}$, and the proof for $\sfT_0^{-\rho}$ is similar. 
Since $\sfT^0_{-\rho}$ is exact, and any object in $\sO^{\leq \bm{\nu}}_{R}$ is a quotient of a module composed by Verma factors in $\{M(\bm{\lambda})_R\}_{\lambda\leq \nu}$, it is enough to show that $\sfT_0^{-\rho}M(\bm{\lambda})_R$ lies in $\sO^{\uparrow l\nu}_{R}$. 
By (1), $\sfT_0^{-\rho}M(\bm{\lambda})_R$ admits Verma factors $M(w\bullet 0+l\lambda)_R$ for $w\in W$, which therefore lies in $\sO^{\uparrow l\nu}_{R}$ by Lemma~\ref{lem new5.10}(1). 

(3) We only prove the first isomorphism. 
Recall that $\tau^{\uparrow l\nu}$ and $\tau^{\leq \bm{\nu}}$ are left adjoint to the natural inclusions $\sO^{\uparrow l\nu}_R\hookrightarrow \sO^0_R$ and $\sO^{-\rho, \leq \bm{\nu}}_R\hookrightarrow \sO^{-\rho}_R$, respectively. 
By (2), $\sfT_0^{-\rho}\circ \tau^{\uparrow l\nu}$ and $\tau^{\leq \bm{\nu}}\circ \sfT_0^{-\rho}$ are both left adjoint to the functor $\sfT_{-\rho}^0:\sO^{-\rho,\leq \bm{\nu}}_{R}\rightarrow \sO^{0}_{R}$, hence they are natural isomorphic to each other. 
\end{proof} 

%---------------------------------------------------------------
\subsection{Center of $\sO_\C^0$}\label{subsect 4.3} 
%---------------------------------------------------------------
Now we study the center of principal block $Z(\sO_\C^0)$. 

\subsubsection{ } 
In this subsection, we let the deformation ring $R$ be either $S$ or $\k$. 

Consider the fibration $\check{G}/\check{B}=\check{G}[\![t]\!]/\sI\rightarrow \Fl^\circ \rightarrow \Gr^\circ$, whose restriction on the fiber of $\check{G}_\sO/\check{G}_\sO$ induces an $S'$-algebra homomorphism $H_{\check{T}}^\bullet(\Fl^\circ) \rightarrow H_{\check{T}}^\bullet(\check{G}/\check{B})$ by pullback. 
It is known that this map admits a retraction of $S'$-algebras 
$$H_{\check{T}}^\bullet(\check{G}/\check{B}) \rightarrow H_{\check{T}}^\bullet(\Fl^\circ)$$
such that its tensor product with the natural map $H_{\check{T}}^\bullet(\Gr^\circ)\rightarrow H_{\check{T}}^\bullet(\Fl^\circ)$ yields an $S'$-algebra isomorphism 
\begin{equation}\label{equ 5.1} 
H_{\check{T}}^\bullet(\check{G}/\check{B})\otimes_{S'} H_{\check{T}}^\bullet(\Gr^\circ)\xs H_{\check{T}}^\bullet(\Fl^\circ). 
\end{equation}

\begin{lem}\label{lem 5.2} 
\begin{enumerate}
\item There is an isomorphism $\sfT^0_{-\rho} M(-\rho)_R= Q(w_0\bullet 0)^{\uparrow 0}_R$. 
\item The composition 
\begin{equation}\label{equ 5.2}
H_{\check{T}}^\bullet(\check{G}/\check{B}) \rightarrow H_{\check{T}}^\bullet(\Fl^\circ)
\xrightarrow{\bb_0}Z(\sO^{0}_{R})\rightarrow \End(Q(w_0\bullet 0)_R^{\uparrow 0}) 
\end{equation}
induces an isomorphism of $R$-algebras $H_{\check{T}}^\bullet(\check{G}/\check{B})\otimes_{S'}R\xs \End(Q(w_0\bullet 0)_R^{\uparrow 0})$. 
\end{enumerate} 
\end{lem} 
\begin{proof} 
(1) By Lemma \ref{lem 5.1}(2), $(\sfT^0_{-\rho}, \sfT^{-\rho}_{0})$ forms a biadjoint pair on the truncated categories $\sO^{-\rho, \leq \bm{\nu}}_{R}$ and $\sO^{\uparrow l\nu}_{R}$, hence they send projective objects to projective objects. 
In particular, $\sfT^0_{-\rho}M(-\rho)_R$ is projective in $\sO^{\uparrow 0}_{R}$. 
By Lemma \ref{lem 5.1}(1), $\sfT^0_{-\rho}M(-\rho)_R$ admits Verma factors $M(w\bullet 0)_R$ with $w\in W$, each of which appears once. 
Since by Lemma~\ref{lem new5.17} $Q(w_0\bullet 0)_R^{\uparrow 0}$ is the projective cover of $E(w_0\bullet 0)_\C$ (and thus of $M(w_0\bullet 0)_R$) in $\sO^{\uparrow 0}_{R}$, the module $\sfT^0_{-\rho}M(-\rho)_R$ must contain $Q(w_0\bullet 0)_R^{\uparrow 0}$ as a direct summand. 
By the linkage principle \eqref{equ LP} and BGG reciprocity \eqref{equ BGG2}, we have $(Q(w_0\bullet 0)_R^{\uparrow 0}:M(w\bullet 0)_R)\geq 1$ for each $w\in W$. 
It forces that $\sfT^0_{-\rho}M(-\rho)_R=Q(w_0\bullet 0)_R^{\uparrow 0}$. 

Part (2) can be proved as in the case of the principal block of the category $\sO$ for $U\g$, see e.g. \cite[Thm 3.6]{Fie03}. 
\end{proof} 

For any $\mu\geq 0$, consider the composition of algebra homomorphisms 
$$H_{\check{T}}^\bullet(\check{G}/\check{B}) 
\rightarrow H_{\check{T}}^\bullet(\Fl^\circ)
\xrightarrow{\bb_0} Z(\sO^{0}_{R})
\rightarrow \End(\sfT_{-\rho}^0 Q(-\rho)_R^{\leq  \bm{\mu}}),$$ 
whose image is central in $\End(\sfT_{-\rho}^0 Q(-\rho)_R^{\leq \bm{\mu}})$. 
Its tensor product with the map 
$$\sfT_{-\rho}^0 :\End(Q(-\rho)_R^{\leq \bm{\mu}})\rightarrow \End(\sfT_{-\rho}^0 Q(-\rho)_R^{\leq \bm{\mu}})$$ 
yields an algebra homomorphism 
\begin{equation}\label{equ 5.3} 
H_{\check{T}}^\bullet(\check{G}/\check{B})\otimes_{S'} \End(Q(-\rho)_R^{\leq \bm{\mu}})\rightarrow \End(\sfT_{-\rho}^0 Q(-\rho)_R^{\leq \bm{\mu}}). 
\end{equation}

\begin{lem}\label{lem 5.3} 
There is an isomorphism $\sfT^0_{-\rho}Q(-\rho)_R^{\leq \bm{\mu}}=Q(w_0\bullet 0)_R^{\uparrow l\mu}$. 
Moreover, the map (\ref{equ 5.3}) yields an isomorphism 
\begin{equation}\label{equ 5.4} 
H_{\check{T}}^\bullet(\check{G}/\check{B})\otimes_{S'} \End(Q(-\rho)_R^{\leq \bm{\mu}})\xs \End(Q(w_0\bullet 0)_R^{\uparrow l\mu}). 
\end{equation}
\end{lem}
\begin{proof} 
We abbreviate $Q_R=Q(-\rho)_R^{\leq \bm{\mu}}$ and $M_R=M(-\rho)_R$. 
As in the proof of Lemma \ref{lem 5.2}(1), $\sfT^0_{-\rho}Q_R$ is projective in $\sO^{\uparrow l\mu}_{R}$ and contains $Q(w_0\bullet 0)_R^{\uparrow l\mu}$ as a direct summand. 

We firstly show that (\ref{equ 5.3}) is an isomorphism. 
By Lemma \ref{lem 5.1}(3), there is a commutative diagram of $R$-algebras 
\begin{equation}\label{equ 4-5} 
\begin{tikzcd}
    \End(Q_R)\arrow[r,"\tau^{\leq -\rho}"]\arrow[d,"\sfT_{-\rho}^0"] & \End(M_R) \arrow[d,"\sfT_{-\rho}^0"]\\ 
    \End(\sfT_{-\rho}^0 Q_R)\arrow[r,"\tau^{\uparrow 0}"] & \End(\sfT_{-\rho}^0 M_R).
\end{tikzcd}
\end{equation}
Let $R=\k$. 
Since $Q_\k$ is the projective cover of $E(-\rho)_\k$ in $\sO^{-\rho, \leq \bm{\mu}}_{\k}$, the algebra $\End(Q_\k)$ is a local ring, whose Jacobson radical $\rad (\End(Q_\k))$ coincides with the kernel of the homomorphism  
$$\tau^{\leq -\rho}: \End(Q_\k)\rightarrow \End(M_\k)=\End(E(-\rho)_\k)=\k.$$ 
It shows that the map $\tau^{\uparrow 0}$ in (\ref{equ 4-5}) factors through a homomorphism of right $\End(Q_\k)$-modules 
\begin{equation}\label{equ 4-6} 
\frac{\End(\sfT_{-\rho}^0 Q_\k)}{\End(\sfT_{-\rho}^0 Q_\k)\cdot \rad (\End(Q_\k))} \rightarrow \End(\sfT_{-\rho}^0 M_\k). 
\end{equation} 

\begin{claim}
The map (\ref{equ 4-6}) is an isomorphism. 
\end{claim} 
\begin{proof} 
We prove the isomorphism by constructing $\k$-basis on both sides. 
By adjunction there is a factorial isomorphism for any $M_1\in \sO^{-\rho}_{R}$ and $M_2\in \sO^{0}_{R}$, 
$$\adj: \Hom(\sfT^0_{-\rho}M_1,M_2)\xs \Hom(M_1, \sfT_0^{-\rho}M_2).$$ 
There is a natural isomorphism $\Upsilon:\id^{\oplus|W|}\xs \sfT^{-\rho}_0 \sfT_{-\rho}^0$ by Lemma \ref{lem 5.0}(3). 
For $w\in W$, we let $\iota_w\in \End(\sfT_{-\rho}^0Q_R)$ be the element whose image under the composition 
\begin{equation}\label{equ 5.5} 
\End(\sfT^0_{-\rho}Q_R)\xrightarrow[\sim]{\adj} 
\Hom(Q_R, \sfT_0^{-\rho} \sfT^0_{-\rho}Q_R)\xrightarrow[\sim]{\Upsilon_{Q_R,*}^{-1}} \Hom(Q_R,Q_R^{\oplus |W|}) 
\end{equation} 
represents the embedding of the $w$-th direct factor. 
By adjunction, we have $\adj(\iota_w\circ \sfT_{-\rho}^0f)=\adj(\iota_w)\circ f$ for any $f\in \End(Q_R)$, hence the family $\{\iota_w\}_{w\in W}$ forms a free basis of $\End(\sfT^0_{-\rho}Q_R)$ as a right $\End(Q_R)$-module. 

By Lemma \ref{lem 5.1}(3), we have $\tau^{\leq -\rho} \sfT^{-\rho}_0 \sfT_{-\rho}^0= \sfT^{-\rho}_0 \sfT_{-\rho}^0\tau^{\leq -\rho}$ as functors on $\sO^{-\rho}_R$. 
Hence for any $f\in \End(\sfT_{-\rho}^0Q_R)$, there is a commutative diagram 
$$\begin{tikzcd}[column sep=huge]
Q_R\arrow[r,"\adj(f)"] \arrow[d,two heads,"\epsilon^{\leq -\rho}"'] 
& \sfT^{-\rho}_0 \sfT_{-\rho}^0 Q_R \arrow[d,two heads,"\sfT^{-\rho}_0 \sfT_{-\rho}^0(\epsilon^{\leq -\rho})=\epsilon^{\leq -\rho}"] \\ 
M_R \arrow[r,"\adj(\tau^{\leq 0}(f))"] & \sfT^{-\rho}_0 \sfT_{-\rho}^0 M_R.
\end{tikzcd}$$
It follows that $\adj(\tau^{\uparrow 0} (f))=\tau^{\leq -\rho}(\adj(f))$. 
Note that 
$$\Upsilon^{-1}_{M_R}\circ \adj(\tau^{\leq -\rho}(\iota_w))=\Upsilon^{-1}_{M_R}\circ \tau^{\leq -\rho}(\adj(\iota_w))=\tau^{\leq -\rho}(\Upsilon^{-1}_{Q_R}\circ \adj(\iota_w))$$ 
is the embedding of the $w$-th factor $M_R\rightarrow M_R^{\oplus |W|}$. 
So the family $\{\tau^{\leq -\rho}(\iota_w)\}_{w\in W}$ forms an $R$-basis of $\End(\sfT_0^{-\rho}M_R)$, using (\ref{equ 5.5}) for $M_R$. 
Finally, we let $R=\k$, then the homomorphism $\tau^{\leq 0}:\End(\sfT_{-\rho}^0Q_\k)\rightarrow \End(\sfT_{-\rho}^0M_\k)$ maps the basis $\{\iota_w\}_{w\in W}$ to $\{\tau^{\leq -\rho}(\iota_w)\}_{w\in W}$, which shows that (\ref{equ 4-6}) is an isomorphism. 
\end{proof}

We see that the composition 
\begin{equation}\label{equ 4-8} 
H^\bullet(\check{G}/\check{B})\otimes \End(Q_\k)\rightarrow \End(\sfT_{-\rho}^0 Q_\k)\xrightarrow{\tau^{\uparrow 0}} \End(\sfT_{-\rho}^0 M_\k) 
\end{equation}
kills $H^\bullet(\check{G}/\check{B})\otimes \rad(\End(Q_\k))$, and modulo $\rad(\End(Q_\k))$ it becomes an isomorphism 
$$H^\bullet(\check{G}/\check{B})\xs \End(\sfT_{-\rho}^0 M_\k)$$ 
by Lemma \ref{lem 5.2}. 
Combining with the claim above, the first map in (\ref{equ 4-8}) is an isomorphism modulo $\rad(\End(Q_\k))$. 
Applying Nakayama's Lemma to the local algebra $\End(Q_\k)$, we deduce that the map 
$$H^\bullet(\check{G}/\check{B})\otimes \End(Q_\k)\rightarrow \End(\sfT_{-\rho}^0 Q_\k)$$ 
is surjective. 
It must be an isomorphism, as both sides have the same dimension by (\ref{equ 5.5}). 
Applying Nakayama's Lemma again to the local ring $S$, the map 
$$H_{\check{T}}^\bullet(\check{G}/\check{B})\otimes_{S'} \End(Q_S)\rightarrow \End(\sfT_{-\rho}^0 Q_S)$$ 
is a surjection of free $S$-modules of finite ranks, which then must be an isomorphism by equal ranks on both sides using (\ref{equ 5.5}) again. 

By the isomorphism (\ref{equ 5.3}) and the fact that $H^\bullet(\check{G}/\check{B})$ and $\End(Q_\k)$ are local $\k$-algebras, it follows that $\End(\sfT_{-\rho}^0 Q_\k)$ is local. 
Therefore $\sfT_{-\rho}^0 Q_\k$ is indecomposable, then so is $\sfT_{-\rho}^0 Q_S$ because $S$ is local. 
It implies that $\sfT_{-\rho}^0 Q_R=Q(w_0\bullet 0)_R^{\uparrow l\mu}$ for $R=S$ or $\k$. 
\end{proof} 

Taking limit on both sides of (\ref{equ 5.4}), we get an algebra isomorphism 
$$H_{\check{T}}^\bullet(\check{G}/\check{B})\otimes_{S'} \Lim{\mu\geq 0} \End(Q(-\rho)_R^{\leq \bm{\mu}})\xs \Lim{\mu\geq 0} \End(Q(w_0\bullet 0)_R^{\uparrow l\mu}) ,$$ 
which therefore induces an isomorphism 
\begin{equation}\label{equ 5.6} 
H_{\check{T}}^\bullet(\check{G}/\check{B})\otimes_{S'} Z\big(\Lim{\mu\geq 0} \End(Q(-\rho)_R^{\leq \bm{\mu}})\big)\xs Z\big(\Lim{\mu\geq 0} \End(Q(w_0\bullet 0)_R^{\uparrow l\mu})\big). 
\end{equation} 

\subsubsection{ } 
Let $R$ be a commutative Noetherian $S$-algebra. 

\begin{lem}\label{lem 5.5} 
For any $M\in \sO^{-\rho}_R$, there is a commutative diagram 
\begin{equation}\label{equ 5.7} 
    \begin{tikzcd}
    H_{\check{T}}^\bullet(\Gr^\circ) \arrow[rr,"(-\otimes_S R)\circ\bb_0"]\arrow[d] & &Z(\sO^{-\rho}_R)\arrow[r] & \End(M) 
        \arrow[d,"\sfT_{-\rho}^0"] \\ 
   H_{\check{T}}^\bullet(\Fl^\circ) \arrow[rr,"(-\otimes_S R)\circ\bb_0"]& & Z(\sO^{0}_R)\arrow[r] & \End(\sfT_{-\rho}^0M).
    \end{tikzcd}
    \end{equation}
\end{lem} 
\begin{proof} 
Choose a surjection $f:Q\twoheadrightarrow M$, where $Q$ is a projective module in a truncation of $\sO^{-\rho}_R$. 
Denote by $\End(Q;M)$ the subring of endomorphisms of $Q$ preserving $\ker f$. 
Then we have the following diagram 
$$\begin{tikzcd}
H_{\check{T}}^\bullet(\Gr^\circ) \arrow[r]\arrow[d] & \End(Q;M)\arrow[d,"\sfT_{-\rho}^0"]\arrow[r,two heads] & \End(M) \arrow[d,"\sfT_{-\rho}^0"] \\ 
H_{\check{T}}^\bullet(\Fl^\circ) \arrow[r]& \End(\sfT_{-\rho}^0Q; \sfT_{-\rho}^0M)\arrow[r,two heads] & \End(\sfT_{-\rho}^0M),
\end{tikzcd}$$
where the right square is naturally commutes, and the right horizontal maps are surjective because $Q$ and $\sfT_{-\rho}^0Q$ are projective in some truncated categories, thanks to Lemma \ref{lem 5.1}(2). 
Hence it is enough to prove the assertion when $M$ is projective in a truncated category. 
By (\ref{equ 2.10}), it reduces to the case when $R=S$. 
Let $\K$ be the fraction field of $S$. 
Since $M$ is torsion free over $S$, we have an inclusion $\End(M)\hookrightarrow \End(M\otimes_S \K)$. 
We only need to prove for the case $R=\K$. 

By \cite[Lem 3.5]{Situ1} the category $\sO_\K$ is semi-simple, whose simple objects are Verma modules. 
Now $M\otimes_S \K$ decomposes into a direct sum of Verma modules in $\sO^{-\rho}_\K$, we may assume that $M=M(\bm{\lambda})_\K$ for some $\lambda\in \rQ$. 
By Lemma \ref{lem 5.1}(1) $\sfT_{-\rho}^0M(\bm{\lambda})_\K=\bigoplus_{w\in W}M(w\bullet 0+l\lambda)_\K$. 
By Theorem \ref{thm 3.11}, the actions of cohomology rings on Verma modules coincides with the restrictions on certain $\check{T}$-fixed points.  
Now the conclusion follows from the commutative diagram 
$$\begin{tikzcd}
H_{\check{T}}^\bullet(\Gr^\circ) \arrow[r,hook] \arrow[d] 
& \Fun(\Lambda,S') \arrow[d] \\ 
H_{\check{T}}^\bullet(\Fl^\circ) \arrow[r,hook] & \Fun(W_\af,S'),
\end{tikzcd}$$
where the horizontal maps are by restrictions on the $\check{T}$-fixed points, and the right vertical maps is by identifying $\Fun(\Lambda,S')=\Fun(W_\af,S')^W$ via the right action of $W$ on $W_\af$. 
\end{proof} 

The restrictions $Z(\sO^{0}_{\k})\rightarrow \End(Q(w_0\bullet 0)_\k^{\uparrow l\mu})$ for each $\mu$ yield an algebra homomorphism $Z(\sO^{0}_{\k})\rightarrow \Lim{\mu\geq 0}\End(Q(w_0\bullet 0)_\k^{\uparrow l\mu})$, whose image is a central subalgebra. 

\begin{prop}\label{prop 5.6} 
The composition 
$$H^\bullet(\Fl^\circ)^{\wedge} \xrightarrow{\overline{\bb}_0} Z(\sO^{0}_{\k}) \rightarrow Z\big(\Lim{\mu\geq 0}\End(Q(w_0\bullet 0)_\k^{\uparrow l\mu})\big)$$ 
is an isomorphism. 
\end{prop} 
\begin{proof} 
By Lemma \ref{lem 5.5}, there is a commutative diagram 
$$\begin{tikzcd}
H^\bullet(\check{G}/\check{B})\otimes H^\bullet(\Gr^\circ)^{\wedge} \arrow[d,"(\ref{equ 5.1})","\simeq"'] \arrow[r] 
& H^\bullet(\check{G}/\check{B}) \otimes Z\big(\Lim{\mu\geq 0}\End(Q(-\rho)_\k^{\leq \bm{\mu}}\big) \arrow[d,"\simeq"',"(\ref{equ 5.6})"] \\ 
H^\bullet(\Fl^\circ)^{\wedge} \arrow[r] 
& Z\big(\Lim{\mu\geq 0}\End(Q(w_0\bullet 0)_\k^{\uparrow l\mu})\big).
\end{tikzcd}$$
By Lemma \ref{lem 4.13}, the upper horizontal map is an isomorphism, so is the lower one. 
\end{proof} 

Let $\eta:\id \rightarrow \sfT_{-\rho}^{0} \sfT^{-\rho}_{0}$ be the unit for the adjoint pair $(\sfT^{-\rho}_{0}, \sfT_{-\rho}^{0})$. 
\begin{lem}\label{lem 5.7} 
Let $R$ be a deformation ring of $U^\hb_\zeta$. 
\begin{enumerate} 
\item For any module $Q$ in $\sO^{0}_R$ admitting Verma flags, the unit $\eta_Q: Q\rightarrow \sfT_{-\rho}^{0} \sfT^{-\rho}_{0}Q$ is an injection; 
\item The functor $\sfT_0^{-\rho}$ is faithful for modules in $\sO^{0}_R$ admitting Verma flags. 
\end{enumerate} 
\end{lem} 
\begin{proof} 
(1) 
%By 5-lemma it is enough to show the case when $Q$ is a Verma module. 
Suppose that $K=\ker \eta_Q$ is nonzero. 
By adjunction the map $\sfT_{0}^{-\rho}K\rightarrow \sfT^{-\rho}_{0}Q$ is by zero. 
Since $\sfT_{0}^{-\rho}$ is exact, $\sfT_{0}^{-\rho}K=0$. 
Choose a highest weight vector $k\in K_\lambda$. 
Since $Q$ admits Verma flags, it is free as a module of $\fU^-_\zeta\otimes R$. 
In particular, $K$ is torsion-free over $\fU^-_\zeta$. 
It shows that the surjection $M(\lambda)_R\rightarrow (U^\hb_\zeta\otimes R).k$ induces an isomorphism from $M(\lambda)_{R/\mathrm{Ann}(k)}$ to the image, where $\mathrm{Ann}(k)$ is the annihilator of $k$ in $R$. 
The submodule $\sfT_{0}^{-\rho} M(\lambda)_{R/\mathrm{Ann}(k)}$ of $\sfT_{0}^{-\rho}K$ is nonzero by Lemma \ref{lem 5.0}(2), which leads to a contradiction. 

(2) Let $M_i$, $i=1,2$ be modules in $\sO^{0}_R$ admitting Verma flags. 
For any $f\in \Hom(M_1,M_2)$ such that $\sfT_0^{-\rho}(f)=0$, we have $\eta_{M_2} \circ f=\sfT_{-\rho}^{0} \sfT^{-\rho}_{0}(f)\circ \eta_{M_1}=0$. 
By (1) $\eta_{M_2}$ is an injection, it follows that $f=0$. 
\end{proof} 

\begin{thm}\label{thm 5.10} 
There is an isomorphism 
$$\overline{\bb}_0: H^\bullet(\Fl^\circ)^{\wedge}\xs Z(\sO^{0}_{\k}).$$ 
\end{thm}
\begin{proof} 
Using Proposition \ref{prop 5.6}, it is enough to show that the restriction 
$$ Z(\sO^{0}_{ \k}) \rightarrow Z\big( \Lim{\mu\geq 0}\End(Q(w_0\bullet 0)_\k^{\uparrow l\mu}) \big) $$ 
is an injection. 
Suppose there is an element $z\in Z(\sO^{0}_{ \k})$ acting by zero on $Q(w_0\bullet 0)_\k^{\uparrow l\mu}$ for each $\mu\geq 0$, we have to show that $z$ acts by zero on $Q(w\bullet 0+ l\nu)^{\uparrow l\mu}_\k$ for any $w\in W$ and any $\mu\geq\nu$ in $\rQ$. 
The proof is by combining the following two steps and using the $l\rQ$-symmetry on the category $\sO^{0}_{\k}$. 
	
\textit{Step 1. Show that $z$ acts by zero on $Q(w_0\bullet 0+ l\nu)^{\uparrow l\mu}_\k$, for any any $\mu\geq\nu$ in $\rQ$.} 
By the $l\rQ$-symmetry, it is enough to consider the cases when $\nu=\pm \alpha_i$. 
Suppose $\nu=\alpha_i$. 
Recall the injection $\iota_{\alpha_i}$ defined in (\ref{equ 4.21}), which gives an injection $\iota_{\alpha_i}: Q(\bm{\alpha_i})^{\leq \bm{\mu}}_\k \hookrightarrow Q(\bm{0})^{\leq \bm{\mu}}_\k$ by the equivalence in Theorem \ref{thm 4.7}. 
It yields an inclusion 
$$\sfT^0_{-\rho}\iota_{\alpha}: Q(w_0\bullet 0+l{\alpha_i})^{\uparrow l\mu}_\k \hookrightarrow Q(w_0\bullet 0)^{\uparrow l\mu}_\k.$$ 
Hence $z$ acts by zero on $Q(w_0\bullet 0+l{\alpha_i})^{\uparrow l\mu}_\k$. 
The case $\nu=-\alpha_i$ is proved similarly, using the injection 
$$\fQ(-\alpha)^{\leq \mu}_\k \hookrightarrow \fQ(0)^{\leq \mu+\alpha}_\k, \quad 1\otimes 1\otimes 1 \mapsto \varphi_i \otimes 1\otimes 1.$$ 

\textit{Step 2. Show that $z$ acts by zero on $Q(w\bullet 0)^{\uparrow l\mu}_\k$, for any $w\in W$.} 
By Lemma \ref{lem 5.7}(1), it is enough to show that $z$ acts by zero on $\sfT_{-\rho}^{0} \sfT^{-\rho}_{0}Q(w\bullet 0)^{\uparrow l\mu}_\k$. 
Since $\sfT^{-\rho}_{0} Q(w\bullet 0)^{\uparrow l\mu}_\k$ is projective in the category $\sO^{-\rho, \leq \bm{\mu}}_{\k}$, it is a direct sum of $Q(\bm{\nu})^{\leq \bm{\mu}}_\k$ for some $\nu\leq \mu$. 
By Lemma \ref{lem 5.3}, $\sfT_{-\rho}^{0} \sfT^{-\rho}_{0}Q(w\bullet 0)^{\uparrow l\mu}_\k$ is a direct sum of $Q(w_0\bullet0+l\nu)^{\uparrow l\mu}_\k$ for some $\nu\leq \mu$, on which $z$ acts by zero, thanks to \textit{Step 1}. 
\end{proof}

%%%%%%%%%%%%%%%%%%%%%%%%%
\section{Center of singular blocks}\label{sect 5} 
%%%%%%%%%%%%%%%%%%%%%%%%%

In this section, we use the isomorphism $\overline{\bb}_0$ for the principal block to show that the map 
$$\overline{\bb}_\omega: H^\bullet(\Fl^{\omega, \circ})^\wedge\rightarrow \sO^{\omega}_{\k}$$ is an isomorphism, for any $\omega \in \Xi_\sc$. 
Our main technique is to relate the centers of $\sO^{0}_{\k}$ and $\sO^{\omega}_{\k}$ by taking the trace of the translation functors. 
The general construction of taking trace of a functor is discussed in Appendix \ref{app D}. 

%---------------------------------------------------------------
\subsection{Trace of translation functors}\label{subsect 5.1} 
%---------------------------------------------------------------
Till the end of the section, we let $R=S$ or $\k$. 
Let $\omega_1,\omega_2\in \Xi_\sc$. 
By Lemma \ref{lem 5.0}(1), there is a triple of adjoint functors $(\sfT^{\omega_2}_{\omega_1}, \sfT^{\omega_1}_{\omega_2}, \sfT^{\omega_2}_{\omega_1})$ between $\sO^{\omega_1}_{R}$ and $\sO^{\omega_2}_{R}$. 
By \textsection \ref{subsect D.1}, it yields an $R$-linear map 
$$\tr_{\sfT^{\omega_2}_{\omega_1}}:\ Z(\sO^{\omega_2}_{R}) \rightarrow Z(\sO^{\omega_1}_{R}).$$ 
When $R=S$, by the isomorphism $\bb$ in Theorem \ref{thm 3.11}, we have an $S$-linear map 
$$\tr^{\omega_2}_{\omega_1,S}:\ H_{\check{T}}^\bullet(\Fl^{\omega_2,\circ})^\wedge_{S}\rightarrow H_{\check{T}}^\bullet(\Fl^{\omega_1,\circ})^\wedge_{S} .$$ 

\begin{lem} 
The map $\tr^{\omega_2}_{\omega_1,S}$ specializes to a well-defined $\k$-linear map 
$$\tr^{\omega_2}_{\omega_1}: H^\bullet(\Fl^{\omega_2,\circ})^\wedge \rightarrow H^\bullet(\Fl^{\omega_1,\circ})^\wedge$$ 
such that there is a commutative diagram 
$$\begin{tikzcd}
H^\bullet(\Fl^{\omega_2,\circ})^\wedge \arrow[d,"\overline{\bb}_{\omega_2}"']\arrow[r,"\tr^{\omega_2}_{\omega_1}"] & H^\bullet(\Fl^{\omega_1,\circ})^\wedge \arrow[d,"\overline{\bb}_{\omega_1}"] \\ 
Z(\sO^{\omega_2}_{\k})\arrow[r,"\tr_{\sfT^{\omega_2}_{\omega_1}}"] & Z(\sO^{\omega_1}_{\k}).
\end{tikzcd}$$
\end{lem} 
\begin{proof} 
Since $H_{\check{T}}^\bullet(\Fl^{\omega_i,\circ})^\wedge_S$ is the space of formal sums of the Schubert classes $[\Fl^{\omega_i,x}]_{\check{T}}$, $x\in W_{l,\af}^{\omega_i}$, we have to show that $\tr^{\omega_2}_{\omega_1,S}$ is compatible with infinite sum, namely, 
\begin{equation}\label{equ 4.a} 
\tr^{\omega_2}_{\omega_1,S}(\sum_{x\in W_{l,\af}^{\omega_2}} r_x\cdot [\Fl^{\omega_2,x}]_{\check{T}})=\sum_{x\in W_{l,\af}^{\omega_2}}r_x\cdot \tr^{\omega_2}_{\omega_1,S}([\Fl^{\omega_2,x}]_{\check{T}}), \quad \forall r_x\in S.
\end{equation}
Denote by $l_1$ the length of the longest element in $W_{l,\omega_1}$. 
For any $x\in W^{\omega_2}_{l,\af}$ and $y\in W^{\omega_1}_{l,\af}$, Lemma \ref{lem 5.0}(2) shows that 
$$\big(\sfT^{\omega_2}_{\omega_1}M(y\bullet \omega_1)_S:M(x\bullet \omega_2)_S\big)\neq 0,
\quad \text{only if}\quad \ell(y)\geq \ell(x)-l_1.$$ 
By definition, for any $z\in Z(\sO^{\omega_2}_{S})$, the element $\tr_{T^{\omega_2}_{\omega_1}}(z)\in Z(\sO^{\omega_1}_{S})$ acts on a module $M\in \sO^{\omega_1}_{S}$ by the composition 
$$M\rightarrow \sfT_{\omega_2}^{\omega_1} \sfT^{\omega_2}_{\omega_1}M \xrightarrow{\sfT_{\omega_2}^{\omega_1}z \sfT^{\omega_2}_{\omega_1}} \sfT_{\omega_2}^{\omega_1} \sfT^{\omega_2}_{\omega_1}M\rightarrow M.$$ 
Recall that the pullback of $[\Fl^{\omega_2,x}]_{\check{T}}$ to the point $\delta_{x'}$ is nonzero only if $x'\geq x$ in the Bruhat order, and it is the scalar how $[\Fl^{\omega_2,x}]_{\check{T}}$ acts on the Verma module $M(x'\bullet \omega_2)_S$. 
It follows that $\tr_{T^{\omega_2}_{\omega_1}}([\Fl^{\omega_2,x}]_{\check{T}})$ acts by zero on the Verma module $M(y\bullet \omega_1)_S$ unless $\ell(y)\geq \ell(x)-l_1$. 
Hence $\tr_{T^{\omega_2}_{\omega_1}}([\Fl^{\omega_2,x}]_{\check{T}})$ is contained in $\prod\limits_{\ell(y)\geq \ell(x)-l_1} S\cdot [\Fl^{\omega_1,y}]_{\check{T}}$. 
Therefore, the RHS of (\ref{equ 4.a}) is well-defined and the equation holds. 
So $\tr^{\omega_2}_{\omega_1,S}$ specializes to a $\k$-linear map $\tr^{\omega_2}_{\omega_1}: H^\bullet(\Fl^{\omega_2,\circ})^\wedge \rightarrow H^\bullet(\Fl^{\omega_1,\circ})^\wedge$. 
The desired commutative diagram is induced by the specialization of the following one 
$$\begin{tikzcd}
H_{\check{T}}^\bullet(\Fl^{\omega_2,\circ})^\wedge \arrow[d,"\bb_{\omega_2}"']\arrow[r,"\tr^{\omega_2}_{\omega_1,S}"] & H_{\check{T}}^\bullet(\Fl^{\omega_1,\circ})^\wedge \arrow[d,"\bb_{\omega_1}"] \\ 
Z(\sO^{\omega_2}_{S})\arrow[r,"\tr_{\sfT^{\omega_2}_{\omega_1}}"] & Z(\sO^{\omega_1}_{S}).
\end{tikzcd}$$
\end{proof}

Let $\omega \in \Xi_\sc$. 
We abbreviate $\sT=\sfT^\omega_0$ and $\sT'=\sfT_\omega^0$, and denote the units and counits by 
$$\varepsilon': \id \rightarrow \sT\sT',\quad \varepsilon: \id \rightarrow \sT'\sT, \quad \eta': \sT\sT' \rightarrow \id,\quad \eta: \sT'\sT\rightarrow \id.$$ 

\begin{lem}\label{lem 1-2.4} 
The composition $\tr_{\sT'}\circ \tr_{\sT}$ is $Z(\sO^{\omega}_{R})$-linear, i.e.  
$$\tr_{\sT'}\circ \tr_{\sT}(z)=
z\cdot \tr_{\sT'}\circ \tr_{\sT}(1), \quad \forall z\in Z(\sO^{\omega}_{R}).$$ 
\end{lem}
\begin{proof} 
For any $z\in Z(\sO^{\omega}_{R})$, by definition $\tr_{\sT'}\circ \tr_{\sT}(z)$ is the natural transformation 
\begin{equation}\label{equ 1-2.2} 
\id \xrightarrow{(\sT\varepsilon \sT')\circ \varepsilon'} \sT\sT'\sT\sT' \xrightarrow{\sT\sT'z{\sT\sT'}} \sT\sT'\sT\sT' \xrightarrow{\eta'\circ(\sT\eta \sT')} \id.
\end{equation}
Recall the natural isomorphism $\Upsilon: \id^{\oplus |W_{l,\omega}|}\xs \sT\sT'$ in Lemma \ref{lem 5.0}(3). 
Consider the following diagram 
$$\begin{tikzcd}[column sep=large]
\id \arrow[r,"{(\sT\varepsilon \sT')\circ \varepsilon'}"] \arrow[d,"{z}"'] 
&\sT\sT'\sT\sT'\arrow[r,"{\sT\sT'\Upsilon^{-1}}","{\sim}"'] \arrow[d,"\sT\sT'z{\sT\sT'}"] 
& (\sT\sT')^{\oplus |W_{l,\omega}|}\arrow[d,"\sT\sT'z^{\oplus |W_{l,\omega}|}"] \\ 
\id \arrow[r,"{(\sT\varepsilon \sT')\circ \varepsilon'}"] \arrow[rd,"\tr_{\sT'}\circ \tr_{\sT}(1)"']
& \sT\sT'\sT\sT'\arrow[r,"{\sT\sT'\Upsilon^{-1}}","{\sim}"'] \arrow[d,"\eta'\circ(\sT\eta \sT')"]
& (\sT\sT')^{\oplus |W_{l,\omega}|} \\ 
& \id & 
\end{tikzcd}$$
where the upper rectangle commutes since the horizontal compositions are natural transformations; the upper right square commutes by the property of center; the lower triangle commutes by (\ref{equ 1-2.2}) applied to $z=1$. 
Hence the upper left square commutes, and it proves the assertion. 
\end{proof} 

%---------------------------------------------------------------
\subsection{ }\label{subsect 5.3.4} 
%---------------------------------------------------------------
Recall in \textsection \ref{subsect 2.3.2}, we show that the algebra homomorphism 
$$\bA[T/W]\xrightarrow{\hc^{-1}} Z(\fU_q)\rightarrow Z(\sO_{\k})$$ 
factors through the quotient $\bA[T/W]\rightarrow \k[\Omega]$, and leads to compatible decompositions 
$$\k[\Omega]=\prod\limits_{[\omega]\in \Xi}\k[\Omega_{[\omega]}], \quad \sO_{\k}=\bigoplus_{[\omega]\in \Xi} \sO^{[\omega]}_{\k}.$$ 
We denote by $\m_{[\omega]}$ the maximal ideal of $\bA[T/W]$ corresponding to the point $\Omega^\red_{[\omega]}$. 
For any integer $n_{[\omega]}\geq 1$, there exists an element $p_{[\omega]}\in \bA[T/W]$ such that 
	\begin{itemize}
	\item $p_{[\omega]}=1_{[\omega]}$ in $\k[\Omega]$ (the idempotent for $\k[\Omega_{[\omega]}]$); 
	\item $p_{[\omega]} \equiv 
    \begin{cases}
	1 & \text{mod }\m_{[\omega]}^{n_{[\omega]}}, \\ 
	0 & \text{mod }\m_{[\omega']}^{n_{[\omega]}},\quad \text{if }{[\omega']}\neq {[\omega]}. 
    \end{cases}$ 
	\end{itemize} 
	
\ 

Fix $\omega\in \Xi_\sc$. 
The natural projection $W_{l,\af}\twoheadrightarrow W$ induces an isomorphism 
$$W_{l,\omega}\xs W_{\zeta^\omega}:=\{x\in W|\ x\bullet \zeta^\omega=\zeta^\omega\}.$$ 
Set $V_q$ be the Weyl module of $U_q$ with extreme weight $-\omega$, and let $V=V_q\otimes_\bA \k$ be the specialization at $q_e=\zeta_e$. 
Fix a dominant weight $\omega'=\omega+2kl\rho$ for some $k\geq 0$. 
Recall that for any $\mu\in \Lambda$, we denote by $[\mu]$ its image in $\Xi=\Lambda/(W_{l,\ex},\bullet)$. 
We abbreviate $\tr_V=\tr_{V_q}$, and see its definition in \textsection \ref{subsect B.1.2}. 

\begin{lem}\label{lem 5.14} \ 
\begin{enumerate}
\item The element $\tr_{\sT'}\circ \tr_{\sT}(1)$ acts on $V(\omega')_\k$ by the scalar $\tr_V\big(p_{[0]}\cdot \tr_{V^*}(p_{[\omega]})\big)(\zeta^{2(\omega+\rho)})$. 
\item We have    $\tr_V\big(p_{[0]}\cdot \tr_{V^*}(p_{[\omega]})\big)(\zeta^{2(\omega+\rho)})=|W_{l,\omega}|$. 
\end{enumerate}
\end{lem} 
\begin{proof}
(1) 
Recall that any $f\in \bA[T/W]$ acts on a Weyl module $V(\lambda)_q$ by the scalar $f(q^{2(\lambda+\rho)})$. 
Note that $V(\omega')_\k$ and $V$ both admit liftings $V(\omega')_q$ and $V_q$ as $U_q$-modules. 
By Proposition \ref{prop D.2}, there is a commutative diagram 
$$\begin{tikzcd} 
V(\omega')_\k\arrow[r,"\varepsilon_V"]\arrow[d,"\tr_V\big(p_{[0]}\cdot \tr_{V^*}(p_{[\omega]})\big)(\zeta^{2(\omega+\rho)})"] 
& V(\omega')_\k\otimes V\otimes V^{*} \arrow[r,"p_{[0]}\otimes V^{*}"] 
& V(\omega')_\k\otimes V\otimes V^{*} \arrow[r,"\varepsilon_{V^*}"] \arrow[ld,"\tr_{V^*}(p_{[\omega]})\otimes V^{*}"] 
& V(\omega')_\k\otimes V\otimes V^* \otimes V^{**} \otimes V^{*} \arrow[d,"p_{[\omega]}\otimes V^{**}\otimes V^{*}"] \\ 
V(\omega')_\k& V(\omega')_\k\otimes V\otimes V^{*} \arrow[l,"\eta_V"]& & V(\omega')_\k\otimes V\otimes V^* \otimes V^{**} \otimes V^{*} \arrow[ll,"\eta_{V^*}"],
\end{tikzcd}$$
where $\varepsilon_V$, $\varepsilon_{V^*}$ and $\eta_V$, $\eta_{V^*}$ are unit and counit maps for $V$, $V^*$. 
Since $p_{[0]}$ is a lifting of the idempotent $1_{[0]}$ in $\k[\Omega]$, it acts on $\sO_{\k}$ as a projector to $\sO^{[0]}_{\k}$. 
By the equality $l\Lambda \cap \rQ=l\rQ$, the direct factor of $V(\omega')_\k\otimes V$ in $\sO^{[0]}_{\k}$ actually lies in $\sO^{0}_{\k}$, hence 
$$p_{[0]}(V(\omega')_\k\otimes V)=\pr_0(V(\omega')_\k\otimes V).$$ 
Similarly, we have 
$$p_{[\omega]}(V(\omega')_\k\otimes V\otimes V^*)=\pr_\omega(V(\omega')_\k\otimes V\otimes V^*).$$ 
Therefore, the morphism $V(\omega')_\k\rightarrow V(\omega')_\k$ provided by the composition along the longest path of the diagram above coincides with 
\begin{equation}\label{equ 1-2.0} 
\begin{aligned}
V(\omega')_\k\rightarrow \pr_0(V(\omega')_\k\otimes V)\otimes V^* \rightarrow 
&\pr_\omega\big(\pr_0(V(\omega')_\k\otimes V)\otimes V^* \big)\otimes V^{**} \otimes V^{*} \\ 
&\rightarrow \pr_0(V(\omega')_\k\otimes V)\otimes V^* \rightarrow V(\omega')_\k, 
\end{aligned} 
\end{equation}
where the arrows are given by suitable unit or counit maps restricted to the corresponding direct summands. 
The morphism (\ref{equ 1-2.0}) can be further factorized into the composition 
$$V(\omega')_\k\rightarrow \sT\sT'V(\omega')_\k \rightarrow \sT\sT'\sT\sT'V(\omega)_\k \rightarrow \sT\sT'V(\omega')_\k \rightarrow V(\omega')_\k,$$
which is the action of $\tr_{\sT'}\circ \tr_{\sT}(1)$ on $V(\omega')_\k$. 

(2) By (1), it is enough to check the equality for $n_{[\omega]}$ and $n_{[0]}$ large enough. 
By the formula (\ref{equ D.3}), we have 
\begin{equation}\label{equ 1-2.3} 
\tr_V\big(p_{[0]}\cdot \tr_{V^*}(p_{[\omega]})\big)
= \frac{\sum\limits_{\nu\in \bP(V_q)} {\tau^{2}_{\nu}}(p_{[0]}\cdot \tr_{V^*}(p_{[\omega]})\cdot \bL)}{\bL}, 
\end{equation} 
where $\bP(V_q)$ is the set of weights in $V_q$ (with multiplicities), and $\bL=K_{\rho}\prod_{\alpha\in \Phi^+}(1-K^{-1}_\alpha)$. 
We factorize $\bL=\bL_\omega \cdot \bL'_\omega$, where 
\begin{equation}\label{equ 4.e15} 
\bL_\omega:=K_{\rho}\prod_{s_\alpha\in W_{\zeta^\omega}}(1-K^{-1}_\alpha),\quad \bL'_\omega:=\prod_{s_\alpha\notin W_{\zeta^\omega}}(1-K^{-1}_\alpha). 
\end{equation}
Note that $\Lambda'_\omega(\zeta^{2(\omega+\rho)})\neq 0$ and $\Lambda'_{\omega}$ is $W_{\zeta^\omega}$-invariant. 
As an element in $\bA[T/W_{\zeta^\omega}][\frac{1}{\bL'_\omega}]$, the RHS of (\ref{equ 1-2.3}) can be further decomposed as 
\begin{equation}\label{equ 1-2.4} 
\frac{1}{\bL'_\omega}
\sum_{\nu\in \bP(V_q)} \frac{\dim V_\nu}{|\Stab_{W_{\zeta^\omega}}(\nu)|}  \cdot  
\frac 
{\sum\limits_{x\in W_{\zeta^\omega}} 
\tau^{2}_{x\nu}(p_{[0]}\cdot \tr_{V^*}(p_{[\omega]}) \cdot \bL) }{\bL_\omega}. 
\end{equation} 

\ 

\begin{claim}\label{claim 1-2.6}
Let $\nu\in \bP(V_q)$. 
The weights $\{\omega+x\nu\}_{x\in W_{\zeta^\omega}}$ are conjugate to each other under the $\bullet$-action of $W_{l,\af}$. 
There exists an integer $n=n_{\omega,\nu}\geq 0$, such that for any $p\in \bA[T/W]$ satisfying $p\in \m_{[\omega+\nu]}^{n}$ and any $f\in \bA[T/W_{\zeta^\omega}]$, we have 
$$\frac
{\sum\limits_{x\in W_{\zeta^\omega}} 
\tau^{2}_{x\nu}\big(p\cdot f \cdot \bL_\omega \big)}{\bL_\omega}(\zeta^{2(\omega+\rho)})=0.$$ 
\end{claim}
\begin{proof}[Proof of Claim \ref{claim 1-2.6}] 
For any $x\in W_{\zeta^\omega}$, denote by $x'$ its preimage under the isomorphism $W_{l,\omega}\xs W_{\zeta^\omega}$. 
Then the first assertion follows from the equality $x'\bullet (\omega' +\nu)= \omega' +x\nu$. 
 
Now we show the second assertion. 
For any $\mu\in \Lambda$, denote by $\m_{\zeta^\mu}$ the maximal ideal of $\fU^0_q=\bA[T]$ corresponding to the point $(\zeta_e,\zeta^\mu)$. 
Note that $\m_{[\mu]}=\m_{\zeta^{2(\mu+\rho)}}\cap \bA[T/W]$. 
Let $n$ be a positive integer such that $\bL_\omega\in \m_{\zeta^{2(\omega+\rho)}}^{n-1}\backslash \m_{\zeta^{2(\omega+\rho)}}^{n}$. 
By the first assertion, we have $\m_{[\omega+\nu]}=\m_{[\omega+x\nu]}$. 
So $p\cdot f \cdot \bL_\omega\in \m_{\zeta^{2(\omega+x\nu+\rho)}}^{n}$, and we have 
$$\tau_{x\nu}\big(p\cdot f \cdot \bL_\omega \big)\in \m_{\zeta^{2(\omega+\rho)}}^{n}, \quad \forall\ x\in W_{\zeta^\omega}.$$ 
It follows that 
$$\frac
{\sum\limits_{x\in W_{\zeta^\omega}} 
\tau_{x\nu}\big(p\cdot f \cdot \bL_\omega \big)}{\bL_\omega}\in \m_{\zeta^{2(\omega+\rho)}}. \qedhere$$ 
\end{proof} 

By \cite[Lem 7.7]{Jan03}, for any $\nu\in \bP(V_q)$, we have $[\omega+\nu]=[0]$ if and only if $\nu\in W_{\zeta^\omega}\cdot (-\omega)$. 
By the claim above, provided that $n_{[0]}\geq \max\{n_{\omega,\nu}|\ \nu\in \bP(V_q)\}$, we have 
$$\frac
{\sum\limits_{x\in W_{\zeta^\omega}} 
\tau^{2}_{x\nu}(p_{[0]}\cdot \tr_{V^*}(p_{[\omega]}) \cdot \bL)}{\bL_\omega}(\zeta^{2(\omega+\rho)}) = 0, \quad \text{if $\nu\notin W_{\zeta^\omega}\cdot (-\omega)$},$$ 
and 
$$\frac
{\sum\limits_{x\in W_{\zeta^\omega}} 
\tau^{2}_{-x\omega}(p_{[0]}\cdot \tr_{V^*}(p_{[\omega]}) \cdot \bL)}{\bL_\omega}(\zeta^{2(\omega+\rho)}) 
=\frac 
{\sum\limits_{x\in W_{\zeta^\omega}} 
\tau^{2}_{-x\omega}(\tr_{V^*}(p_{[\omega]}) \cdot \bL)}{\bL_\omega}(\zeta^{2(\omega+\rho)}).$$ 
Therefore we have 
\begin{equation}\label{equ 1-2.5} 
\tr_V\big(p_{[0]}\cdot \tr_{V^*}(p_{[\omega]})\big)(\zeta^{2(\omega+\rho)})= 
\frac{1}{\bL'_\omega(\zeta^{2(\omega+\rho)})}\cdot 
\frac 
{\sum\limits_{x\in W_{\zeta^\omega}} 
\tau^{2}_{-x\omega}(\tr_{V^*}(p_{[\omega]})\cdot \bL)}{\bL_\omega}(\zeta^{2(\omega+\rho)}). 
\end{equation} 
Set the $\Lambda$-graded $\bA$-module $V_0:=\bigoplus\limits_{x\in W_{\zeta^\omega}}\bA_{-x\omega}$. 
By the formula (\ref{equ D.3}) again, we have 
\begin{align}\label{equ 1-2.6} 
\notag \frac{\sum\limits_{x\in W_{\zeta^\omega}} \tau^{2}_{-x\omega}(\tr_{V^*}(p_{[\omega]})\cdot \bL)}{\bL_\omega} 
&=\frac{\sum\limits_{x\in W_{\zeta^\omega}} \tau^{2}_{-x\omega} 
	\sum\limits_{\nu'\in \bP(V_q^*)} \tau^{2}_{\nu'}(p_{[\omega]}\cdot \bL)}{\bL_\omega}\\ 
\notag &= \frac{\sum\limits_{\nu'\in \bP(V_q^*\otimes V_0)} \tau^{2}_{\nu'}(p_{[\omega]}\cdot \bL)}{\bL_\omega} \\ 
&=|W_{\zeta^\omega}|\cdot p_{[\omega]}\cdot \bL_\omega' + 
\frac{\sum\limits_{\nu'\in \bP(V_q^*\otimes V_0),\nu'\neq 0} \tau^{2}_{\nu'}(p_{[\omega]}\cdot \bL)}{\bL_\omega}. 
\end{align} 
By \cite[Lem 7.7]{Jan03} again, for any $\nu'\in \bP(V_q^*\otimes V_0)$, we have $[\omega+\nu']=[\omega]$ if and only if $\nu'=0$. 
Similarly, we can show that for $n_{[\omega]}$ large enough, the second term of (\ref{equ 1-2.6}) vanishes on $\zeta^{2(\omega+\rho)}$. 
Therefore, we deduce that 
$$\tr_V\big(p_{[0]}\cdot \tr_{V^*}(p_{[\omega]})\big)(\zeta^{2(\omega+\rho)})
=\frac{(|W_{\zeta^\omega}|\cdot p_{[\omega]}\cdot \bL_\omega')(\zeta^{2(\omega+\rho)})}{\bL'_\omega(\zeta^{2(\omega+\rho)})}=|W_{\zeta^\omega}|.\qedhere $$ 
\end{proof} 

\ 

\begin{corollary}\label{cor 1-2.7} 
The element $\tr_{\omega}^{0}\circ \tr^{\omega}_{0}(1)$ is invertible in $H^\bullet(\Fl^{\omega,\circ})^\wedge$, so $\tr_{\sT'}\circ \tr_{\sT}(1)$ is also invertible in $Z(\sO^{\omega}_{\k})$. 
\end{corollary}
\begin{proof} 
Since $H^\bullet(\Fl^{\omega,\circ})^\wedge$ is pro-unipotent, it is enough to show that the degree zero term of the element $\tr_{\omega}^{0}\circ \tr^{\omega}_{0}(1)$ does not vanish. 
Note that this term is exactly the scalar how $\tr_{\sT'}\circ \tr_{\sT}(1)$ acts on any Verma module in $\sO^{\omega}_{\k}$. 
We now show the action of $\tr_{\sT'}\circ \tr_{\sT}(1)$ on $M(\omega')_{\k}$ is nonzero. 
Indeed, choose a nonzero morphism $M(\omega')_{\k}\rightarrow V(\omega')_{\k}$ and consider the commutative diagram 
$$\begin{tikzcd}[column sep=large]
M(\omega')_{\k}\arrow[r,"\tr_{\sT'}\circ \tr_{\sT}(1)"] \arrow[d] 
& M(\omega')_{\k}\arrow[d]\\ 
V(\omega')_{\k}\arrow[r,"\tr_{\sT'}\circ \tr_{\sT}(1)"] & V(\omega')_{\k}.
\end{tikzcd}$$ 
By Lemma \ref{lem 5.14} above, the lower horizontal arrow is by scalar $|W_{l,\omega}|$. 
Since $\End(M(\omega')_{\k})=\k$, the upper one acts by the same scalar. 
\end{proof} 

%---------------------------------------------------------------
\subsection{Center of $\sO^{\omega}_{\k}$}\label{subsect 5.3} 
%---------------------------------------------------------------
\begin{thm}\label{thm 5.17} 
There is an isomorphism 
$$\overline{\bb}_\omega: H^\bullet(\Fl^{\omega,\circ})^\wedge \xs Z(\sO^{\omega}_{\k}).$$
\end{thm} 
\begin{proof}
By Lemma \ref{lem 1-2.4} and the commutative diagram 
$$\begin{tikzcd}
H_{\check{T}}^\bullet(\Fl^{\omega,\circ})^\wedge_{S} \arrow[d,"\simeq"'] \arrow[r,"\tr_0^\omega"] 
& H_{\check{T}}^\bullet(\Fl^{\circ})^\wedge_{S} \arrow[d,"{\simeq}"']\arrow[r,"\tr^0_\omega"] & H_{\check{T}}^\bullet(\Fl^{\omega,\circ})^\wedge_{S} \arrow[d,"{\simeq}"'] \\ 
Z(\sO^{\omega}_{S}) \arrow[r,"\tr_{\sT}"]& Z(\sO^{0}_{S}) \arrow[r,"\tr_{\sT'}"] & Z(\sO^{\omega}_{S}), 
\end{tikzcd}$$ 
we have 
$$\tr^0_\omega\circ \tr_0^\omega(z)=z\cdot \tr^0_\omega\circ \tr_0^\omega(1),\quad \forall z\in H_{\check{T}}^\bullet(\Fl^{\omega,\circ})^\wedge_{S}.$$ 
Hence the same formula holds for $H^\bullet(\Fl^{\omega,\circ})^\wedge$. 
By Corollary \ref{cor 1-2.7}, the $\k$-linear maps
$$\tr^0_\omega\circ \tr_0^\omega: H^\bullet(\Fl^{\omega,\circ})^\wedge \rightarrow H^\bullet(\Fl^{\omega,\circ})^\wedge, \quad 
\tr_{\sT'}\circ \tr_{\sT}:Z(\sO^{\omega}_{\k}) \rightarrow Z(\sO^{\omega}_{\k})$$ 
are isomorphisms. 
So $\tr_0^\omega$ is an injection, and $\tr_{\sT'}$ is a surjection. 
Consider the commutative diagram 
$$\begin{tikzcd}
H^\bullet(\Fl^{\omega,\circ})^\wedge \arrow[d,"\overline{\bb}_\omega"'] \arrow[r,"\tr_0^\omega"] 
& H^\bullet(\Fl^{\circ})^\wedge \arrow[d,"\overline{\bb}_0","{\simeq}"']\arrow[r,"\tr^0_\omega"] & H^\bullet(\Fl^{\omega,\circ})^\wedge \arrow[d,"\overline{\bb}_\omega"] \\ 
Z(\sO^{\omega}_{\k})\arrow[r,"\tr_{\sT}"]& Z(\sO^{0}_{\k})\arrow[r,"\tr_{\sT'}"] & Z(\sO^{\omega}_{\k}),
\end{tikzcd}$$ 
where $\overline{\bb}_0$ and $\overline{\bb}_\omega$ are restrictions of $\overline{\bb}$ to the corresponding direct summands. 
Recall in Theorem \ref{thm 5.10}, we showed that $\overline{\bb}_0$ is an isomorphism. 
Since $\overline{\bb}_0\circ \tr_0^\omega$ is injective, $\overline{\bb}_\omega$ is an injection. 
Since $\tr_{\sT'}\circ \overline{\bb}_0$ is surjective, $\overline{\bb}_\omega$ is a surjection. 
\end{proof}

\newpage

\appendix 
%%%%%%%%%%%%%%%%%%%%%%%%%
\section{The center of a category}\label{app A} 
%%%%%%%%%%%%%%%%%%%%%%%%%
%---------------------------------------------------------------
\subsection{Center of derived categories} 
%---------------------------------------------------------------
Suppose $\sC$ is an abelian $\k$-linear category. 
Define the ($[1]$-compatible) center of its bounded derived category $\Db\sC$ by 
$$ Z(\Db\sC):=\{ z\in \End(\id_{\Db\sC})\ |\ z_{M[1]}=z_M[1],\ \forall M\in \Db\sC\} .$$ 
Note there is a natural map $Z(\sC) \rightarrow Z(\Db\sC)$, which admits a retraction $Z(\Db\sC) \rightarrow Z(\sC)$ by restriction on the full subcategory $\sC\subset \Db\sC$. 
So the map $Z(\sC)\rightarrow Z(\Db\sC)$ is a direct inclusion. 

%---------------------------------------------------------------
\subsection{Graded center}\label{app A3} 
%---------------------------------------------------------------
A \textit{graded category} $(\sD, \langle 1 \rangle)$ is the data of a category $\sD$ with an auto-functor $\langle 1 \rangle$ of $\sD$. 
Set $\langle d \rangle:=\langle 1 \rangle^{\circ d}$ for any $d\in \Z$. 
Define the \textit{degraded center} of $\sD$ by 
$$Z^\bullet(\sD) :=  
	\left\{ 
	\begin{gathered} 
    z= (z_d)_d \in \prod_{d\in \Z} \Hom( \id_\sD , \langle  d \rangle) 
    \end{gathered}  \left |\ 
    \begin{gathered} 
	\text{for any $M\in \sD$, any $k\in \Z$,\quad $z_{M\langle k \rangle}=z_{M}\langle k \rangle$,} \\ 
	\text{and }{z}_{d,M}\neq0 \text{ only for finitely many $d$ }  
	\end{gathered} \right. \right\} ,$$ 
equipped with the natural ring structure. 
A \textit{degrading functor} $v: (\sD, \langle 1 \rangle) \rightarrow \sC$ is the data of 
\begin{itemize} 
	\item a graded category $(\sD, \langle 1 \rangle)$, and a functor $v: \sD \rightarrow \sC$; 
	\item a natural isomorphism $v\xs v \langle 1 \rangle$; 
\end{itemize} 
such that for any $M, N\in \sD$, the natural map 
\begin{equation}\label{equ A.1'}
\bigoplus_{d\in \Z} \Hom_{\sD}( M , N \langle d \rangle) \xs \Hom_\sC( vM, vN )
\end{equation}
is an isomorphism. 
A \textit{lifting} of an object $M\in \sC$ along $v$ is an object $\tilde{M}\in \sD$ such that $v\tilde{M} \cong M$.  

\ 

Let $v: (\sD, \langle 1 \rangle) \rightarrow \sC$ be a degrading functor. 
\begin{lem}\label{lem A.2} 
There is a natural ring homomorphism 
	$$Z(\sC) \rightarrow Z^\bullet(\sD) , \quad z \mapsto (z_d)_d$$ 
	such that $\sum_d v(z_{d,M})=z_{vM}$ for any $M\in \sD$. 
\end{lem} 
\begin{proof}
For any $z\in Z(\sC)$ and any $M\in \sD$, there is a family $\big(z_{d,M}:M\rightarrow M\langle d\rangle\big)_d$ that is zero except finitely many elements, such that $z_{vM}=\sum_d v(z_{d,M})$. 
We show that $(z_d)_d$ defines an element in $Z^\bullet(\sD)$. 
Indeed, for any morphism $f:M\rightarrow M'$ in $\sD$, we have $z_{vM'}\circ vf=vf\circ z_{vM}$, so by definition 
$$\sum_d v(z_{d,M'}\circ f)= \sum_d vz_{d,M'}\circ vf = z_{vM}\circ vf
=\sum_dvf\circ vz_{d,M}=\sum_d v(f\circ z_{d,M'}).$$ 
By the isomorphism (\ref{equ A.1'}), we deduce that $z_{d,M'}\circ f=f\circ z_{d,M}$ for each $d$. 
\end{proof} 

Suppose that $\sC$ and $\sD$ are abelian categories admitting enough projective objects, and that $v$ and $\langle 1\rangle$ are exact functors. 
Let $\sP$ (resp. $\sQ$) be the full subcategory of projective objects in $\sC$ (resp. $\sD$). 
\begin{lem}\label{lem A.21} 
Suppose that any objects in $\sP$ admits a lifting in $\sD$, and that $\sQ$ coincides with the additive full subcategory of $\sD$ generated by the liftings of objects in $\sP$. 
Then the natural map $Z(\sC)\rightarrow Z^\bullet(\sD)$ is an isomorphism. 
\end{lem} 
\begin{proof} 
Note that the restrictions $Z(\sC)\rightarrow Z(\sP)$ and $Z^\bullet(\sD)\rightarrow Z^\bullet(\sQ)$ are isomorphisms. 
We show that the map $Z(\sP)\rightarrow Z^\bullet(\sQ)$ is an isomorphism. 
Indeed, if $(z_d)_d=0$ for some $z\in Z(\sP)$, then for any object $P\in \sP$ with a lifting $Q\in \sQ$, we have $z_P=\sum_d vz_{d,Q}=0$. 
Hence $z=0$. 
Conversely, for any $(z_d)_d\in Z^\bullet(\sQ)$, the map $z:vQ\in \sQ \mapsto \sum_{d}vz_{d,Q}\in \End_{\sC}(vQ)$  defines an element $z\in Z(\sP)=Z(v\sQ)$. 
It shows that $Z(\sP)\rightarrow Z^\bullet(\sQ)$ is a surjection. 
\end{proof}

%---------------------------------------------------------------
\subsection{Center of category of mixed sheaves} 
%---------------------------------------------------------------
Let $X_0$ be an $\F_p$-variety admits a finite stratification $X_0=\bigsqcup_{s\in \mathscr{S}} X_{s,0}$. 
Set $X:= X_0 \times_{\F_p} \overline{\F}_p$, and $X_{s}:= X_{s,0} \times_{\F_p} \overline{\F}_p$. 
Denote by $D^{\mathrm{b},\mx}_\mathscr{S}(X, \overline{\Q}_\ell)$ the triangulate category of mixed $l$-adic sheaves that are constructible along the stratification $\mathscr{S}$. 
Denote by $\langle 1 \rangle$ the half of the Tate twist. 
Define the ``\textit{pure center}" of $D^{\mx}:=D^{\mathrm{b},\mx}_\mathscr{S}(X, \overline{\Q}_\ell)$ by 
$$Z^\pur_{\mathscr{S}}(X) := 
	\left\{ 
	\begin{gathered} 
    z= (z_d)_d \in \prod_{d\in \Z} \Hom( \id_{D^{\mx}} , \langle  d \rangle[d]) 
    \end{gathered}  \left |\ 
    \begin{gathered} 
	\text{for any $\sF\in D^{\mx}$, any $k,k'\in \Z$,} \\ 
	z_{\sF\langle k \rangle[k']}=z_{M}\langle k \rangle[k'], \quad \text{and } \\ 
	{z}_{d,\sF}\neq0 \text{ only for finitely many $d$ } 
	\end{gathered} \right. \right\}  .$$ 
Denote by 
$$ H^\bullet(X)^\pur:= \bigoplus_{d} \Hom_{D^\mx}( {\overline{\Q}_\ell}_X, {\overline{\Q}_\ell}_X\langle d \rangle [d] ) $$
the subspace consisting of pure elements in $H^\bullet(X)$. 
Note that for any $\sF\in D^{\mx}$, we have $\sF\otimes^L {\overline{\Q}_\ell}_X=\sF$, which yields a map $H^\bullet(X)^\pur\rightarrow Z^\pur_{\mathscr{S}}(X)$. 
It admits a retraction by restriction on the constant sheave $Z^\pur_{\mathscr{S}}(X) \rightarrow H^\bullet(X)^\pur$. 

We also have a equivariant version. 
Suppose $X_0$ is equipped with an action of an algebraic group $\Gamma_0$ over $\F_p$, and set $\Gamma=\Gamma_0\times_{\F_p} \overline{\F}_p$. 
Then one can similarly define the ``pure center" $Z_{\Gamma}^\pur(X)$ of $D^{\mathrm{b},\mx}_{\Gamma}(X, \overline{\Q}_\ell)$, and replace $H^\bullet(X)$ by $H^\bullet_{\Gamma}(X)$. 

The construction also works for ind-varieties, by setting 
	$$ H^\bullet(X)^\pur:= \lim_{s\in \mathscr{S}} H^\bullet(\overline{X_s})^\pur .$$

\newpage 
%%%%%%%%%%%%%%%%%%%%%%%%%
\section{Bernstein's formula}\label{app D} 
%%%%%%%%%%%%%%%%%%%%%%%%%
In \cite{B90}, Bernstein studied trace operators associated with translation functors on $U\g$-modules, and used them to give a proof of Soergel's isomorphism between the center of the principal block of the category $\sO$ for $\g$ and the cohomology of flag variety $\check{G}/\check{B}$. 
A quantum analogue of this construction and the fact that the action of the trace operator associated with translation functor on the center is compatible with push-forward on cohomology was obtained by Peng Shan and Eric Vasserot (unpublished). 
In this appendix, we give details of this construction. 

%---------------------------------------------------------------
\subsection{Bernstein's formula}\label{subsect D.1} 
%---------------------------------------------------------------
%=================================
\subsubsection{Trace of functors} 
%=================================
Following \cite{B90}, we define the traces of functors. 
Let $R$ be a commutative ring. 
Let $\sC$ and $\sD$ be $R$-linear additive categories. 
Let $(E,F,G)$ be a triple of adjoint functors with $E,G:\sC\rightarrow \sD$, $F:\sD\rightarrow \sC$. 
Suppose there is a natural transformation (call the \textit{balancing}) $\delta: E\rightarrow G$.  
Then there is a homomorphism of $R$-modules 
$$\tr_{E,\delta}: Z(\sD)\rightarrow Z(\sC)$$ 
given by 
$$\tr_{E,\delta}(z):\ \id_{\sC}\xrightarrow{\varepsilon} FE\xrightarrow{FzE} FE\xrightarrow{F\delta} FG\xrightarrow{\eta} \id_{\sC}, \quad \forall z\in Z(\sD),$$ 
where $\varepsilon$ is the unit associated with $(E,F)$ and $\eta$ is the counit associated with $(F,G)$. 

\subsubsection{Bernstein's formula}\label{subsect B.1.2} 
Denote by $\rep(U_q)$ the full subcategory of modules in $U_q\mod^{\Lambda}_{\bA}$ that are free of finite rank as $\bA$-modules. 
Then $\rep(U_q)$ consists of integrable $U_q$-modules, and is closed under taking tensor products. 
Recall that $\rep(U_q)$ is a rigid monoidal category with the balancing $\id \xs (-)^{**}$ given by the $K_{2\rho}$-action. 
For any $M\in \rep(U_q)$ and $f\in \End(M)$, the \textit{quantum trace} $\tr_{q,M}(f)$ is defined to be the value of $1$ under the composition 
$$\bA\rightarrow M\otimes M^*\xrightarrow{f\otimes \id}M\otimes M^*\xs M^{**}\otimes M^* \rightarrow \bA,$$ 
with the first and the last maps given by unit and counit. 
Therefore $\tr_{q,M}(f)$ coincides with the usual trace $\Tr(K_{2\rho}f)$. 
Recall the \textit{character} of $M$ is 
$$\ch M=\sum_{\lambda} (\rk_{\bA}M_\lambda)\cdot K_\lambda\in \bA[T/W],$$ 
where we identify the algebras $\bA[T/W]=(\bA\langle K_{\lambda}\rangle_{\lambda\in \Lambda})^W$. 

Let $V\in \rep(U_q)$. 
Let $(E,F,G)$ be the adjoint triple of endo-functors of $\rep(U_q)$ given by tensoring $V,V^*,V^{**}$ from the right. 
Consider the balancing $\delta:E\rightarrow G$ by the isomorphism $V\xs V^{**}$ given above. 
By \textsection \ref{subsect D.1}, there is a homomorphism of $\bA$-modules 
$$\tr_{V}:=\tr_{E,\delta}:\ Z(\rep(U_q))\rightarrow Z(\rep(U_q)).$$ 
The \textit{quantum dimension} of $V$ is $\dim_q V=\tr_{q,V}(1)=(\ch V)(q^{2\rho})$. 
\begin{lem}\label{lem D.1} 
Let $M\in \rep(U_q)$ and $f\in Z(\rep(U_q))$. 
\begin{enumerate} 
\item We have $\tr_{q,M\otimes V}(f|_{M\otimes V})=\tr_{q,M}(\tr_V(f)|_M)$; 
\item The map 
$$\tr_q: \rep(U_q)\rightarrow \Hom_{\bA}(Z(\rep(U_q)),\bA), \quad V \mapsto \tr_{q,V}$$ 
descends to $K_0(\rep(U_q))$, i.e. for any short exact sequence $0\rightarrow V_1\rightarrow V\rightarrow V_2\rightarrow 0$ in $\rep(U_q)$, we have $\tr_{q,V}=\tr_{q,V_1}+\tr_{q,V_2}$. 
Namely, $\tr_{q,V}$ as a $\bA$-linear function on $Z(\rep(U_q))$ only depends on the character $\ch V$. 
\end{enumerate} 
\end{lem} 
\begin{proof}
(1) By definition we have 
\begin{align*}
\tr_{q,M\otimes V}(f|_{M\otimes V})&=
\Tr\big((K_{2\rho}f)|_{M\otimes V}\big)=\Tr\big((K_{2\rho}|_{M}\otimes K_{2\rho}|_{V})\cdot f|_{M\otimes V}\big)\\ 
&=\Tr(K_{2\rho}|_{M}\cdot \tr_V(f)|_{M})=\tr_{q,M}(\tr_V(f)|_M). 
\end{align*} 

Part (2) follows from the equality $\Tr\big((K_{2\rho}f)|_{V}\big)=\Tr\big((K_{2\rho}f)|_{V_1}\big)+\Tr\big((K_{2\rho}f)|_{V_2}\big)$. 
\end{proof} 

Recall the co-induction module $H^i(\lambda)_q\in \rep(U_q)$ introduced in \cite[\textsection 3]{APW91}, for $\lambda\in \Lambda$ and $i\geq 0$. 
Set 
$$\chi_{q,\lambda}=\sum_i (-1)^i \ch H^i(\lambda)_q\in \bA[T/W].$$ 
If $\lambda$ is dominant, we have $\chi_{q,\lambda}=\ch H^0(\lambda)_q=\ch V(\lambda)_q$. 
Recall the Weyl dimension formula (see e.g. \cite[II Prop 5.10]{Jan03}) 
\begin{equation}\label{equ D.2} 
\chi_{q,\mu}(q^{2\rho})=\frac{\bL(q^{2(\mu+\rho)})}{\bL(q^{2\rho})}, \quad \forall \mu\in \Lambda,
\end{equation} 
where $\bL:=K_{\rho}\prod_{\alpha\in \Phi^+}(1-K^{-1}_\alpha)\in \bA[T]$. 

Consider the algebra homomorphism $\bA[T/W]\xrightarrow{\hc^{-1}} Z(U_q)\rightarrow Z(\rep(U_q))$.  
For a $\Lambda$-graded $\bA$-module $M$ that is free of finite rank, we denote by $\bP(M)$ the set of weights in $M$, i.e. it consists of $\lambda\in \Lambda$ appearing with multiplicity $\rk_{\bA} M_\lambda$. 
Recall the algebra automorphism $\tau_\nu$ of $\bA[T]$ for $\nu\in \Lambda$, given by $\tau_\nu(K_\mu)=q^{(\nu,\mu)}K_\mu$, for any $\mu\in \Lambda$. 
\begin{prop}\label{prop D.2} 
There is a unique lifting of $\tr_V$ to a linear map $\tr_V:\bA[T/W]\rightarrow \bA[T/W]$, which is given by 
\begin{equation}\label{equ D.3} 
f\mapsto \frac{\sum\limits_{\nu\in \bP(V)} \tau^{2}_{\nu}(f\cdot \bL)}{\bL}, 
\quad \forall f\in \bA[T/W], 
\end{equation} 
such that the diagram commutes 
$$\begin{tikzcd}
\bA[T/W] \arrow[d]\arrow[r,"\tr_{V}"] & \bA[T/W] \arrow[d] \\ 
Z(\rep(U_q))\arrow[r,"\tr_{V}"] & Z(\rep(U_q)).
\end{tikzcd}$$ 
\end{prop} 

\begin{proof} 
Let $\mu$ be a dominant weight. 
Applying Lemma \ref{lem D.1}(1) to $M=V(\mu)_q$, we have 
\begin{equation}\label{equ D.4} 
\tr_{q,V(\mu)_q\otimes V}(f|_{V(\mu)_q\otimes V})=\tr_{q,V(\mu)_q}(\tr_V(f)|_{V(\mu)_q}).
\end{equation} 
Since $\tr_V(f)$ acts on $V(\mu)_q$ by the scalar $\tr_V(f)(q^{2(\mu+\rho)})$, using the formula (\ref{equ D.2}), we deduce  
\begin{equation}\label{equ D.5} 
\begin{aligned} 
\tr_{q,V(\mu)_q}(\tr_V(f)|_{V(\mu)_q})&=\dim_q V(\mu)_q\cdot \tr_V(f)(q^{2(\mu+\rho)})\\ 
&=\frac{1}{\bL(q^{2\rho})}\big(\bL\cdot \tr_V(f)\big)(q^{2(\mu+\rho)}). 
\end{aligned}
\end{equation} 
By the tensor identity \cite[Prop 2.16]{APW91}, we have $\ch(V(\mu)_q\otimes V)=\sum\limits_{\nu\in \bP(V)}\chi_{\mu+\nu}$. 
Hence Lemma \ref{lem D.1}(2) and (\ref{equ D.2}) show that 
\begin{equation}\label{equ D.6} 
\begin{aligned}
\tr_{q,V(\mu)_q\otimes V}(f|_{V(\mu)_q\otimes V})
&=\sum_{\nu\in \bP(V)}\chi_{\mu+\nu}(q^{2\rho})\cdot f(q^{2(\mu+\nu+\rho)}) \\ 
&=\sum_{\nu\in \bP(V)} \frac{\bL(q^{2(\mu+\nu+\rho)})}{\bL(q^{2\rho})}\cdot f(q^{2(\mu+\nu+\rho)})\\ 
&=\frac{1}{\bL(q^{2\rho})} \sum\limits_{\nu\in \bP(V)} \tau^{2}_{\nu}(f\cdot \bL)(q^{2(\mu+\rho)}). 
\end{aligned}
\end{equation} 
Combining (\ref{equ D.4}), (\ref{equ D.5}) and (\ref{equ D.6}), it follows that 
$$\tr_V(f)(q^{2(\mu+\rho)})=\frac{\sum\limits_{\nu\in \bP(V)} \tau^{2}_{\nu}(f\cdot \bL)}{\bL}(q^{2(\mu+\rho)})$$ 
holds for any dominant weight $\mu$. 
Thus the restriction of $\tr_V$ on $\bA[T/W]$ is induced by the map (\ref{equ D.3}). 
\end{proof} 

%---------------------------------------------------------------
\subsection{Trace map and pushforward}\label{subsect B.2} 
%---------------------------------------------------------------
We denote by $\bA_{\widehat{\zeta_e}}$ be the completion of $\bA=\k[q_e^{\pm 1}]$ at $q_e=\zeta_e$, and $\k[\hbar]_{\widehat{0}}$ the completion of $\k[\hbar]$ at $\hbar=0$. 
There is an identification $\bA_{\widehat{\zeta_e}}\simeq \k[\hbar]_{\widehat{0}}=\k[\![\hbar]\!]$ via $\hbar=q_e-\zeta_e$. 
We identify the graded rings $H_{\Gm}^\bullet(\pt)=\k[\hbar]$. 

Let $U_{\hat{\zeta}}=U_q\otimes_{\bA}\bA_{\widehat{\zeta_e}}$, and let $V(\lambda)_{\hat{\zeta}}=V(\lambda)_q\otimes_{\bA}\bA_{\widehat{\zeta_e}}$ be the Weyl module of $U_{\hat{\zeta}}$. 
Denote by $\rep(U_{\hat{\zeta}})$ the full subcategory of $U_{\hat{\zeta}}\Mod^\Lambda_{\bA_{\widehat{\zeta_e}}}$ of the modules that are finitely generated over $\bA_{\widehat{\zeta_e}}$. 
There is a block decomposition 
\begin{equation}\label{equ D.50} 
\rep(U_{\hat{\zeta}})=\bigoplus_{\omega\in \Xi_\sc} \rep^{\omega}(U_{\hat{\zeta}}), 
\end{equation}
such that the Weyl module $V(\lambda)_{\hat{\zeta}}$ lies in $\rep^{\omega}(U_{\hat{\zeta}})$ if and only if $\lambda\in W_{l,\af}\bullet \omega$.

\begin{prop}[{\cite[Cor 4.10]{BBASV}}] 
There is a $\k[\![\hbar]\!]$-algebra homomorphism 
\begin{equation}\label{equ D.9} 
\bc: H^\bullet_{\Gm}(\Gr^\zeta)_{\widehat{0}}\rightarrow Z(\rep(U_{\hat{\zeta}})),
\end{equation}
compatible with the decompositions (\ref{equ 1.1}) and (\ref{equ D.50}), where ${\widehat{0}}$ refers to the completion of the $H^\bullet_{\Gm}(\pt)$-module $H^\bullet_{\Gm}(\Gr^\zeta)$ at $\hbar=0$. 
\end{prop} 

As in \textsection \ref{subsect 4.0}, there are translation functors in $\rep(U_{\hat{\zeta}})$, 
$$\sfT_{\omega_1}^{\omega_2} : \rep^{\omega_1}(U_{\hat{\zeta}}) \rightarrow \rep^{\omega_2}(U_{\hat{\zeta}}),\quad 
\sfT_{\omega_2}^{\omega_1} : \rep^{\omega_2}(U_{\hat{\zeta}}) \rightarrow \rep^{\omega_1}(U_{\hat{\zeta}}),$$ 
given by the Weyl module $V(\nu)_{\hat{\zeta}}$ with extreme weight $\omega_2-\omega_1$. 
We will use the biadjunction of $(\sfT_{\omega_1}^{\omega_2}, \sfT^{\omega_1}_{\omega_2})$ given by the isomorphism $V(\nu)_{\hat{\zeta}}\xs V(\nu)^{**}_{\hat{\zeta}}$ via $K_{2\rho}$-action (c.f. Remark \ref{rmk 4.2}). 
By \textsection \ref{subsect D.1}, the biadjoint pair $(\sfT_{\omega_1}^{\omega_2}, \sfT^{\omega_1}_{\omega_2})$ yields a linear map $\tr_{\sfT_{\omega_1}^{\omega_2}}:Z(\rep^{\omega_2}(U_{\hat{\zeta}}))\rightarrow Z(\rep^{\omega_1}(U_{\hat{\zeta}}))$. 
Set $\Xi_\sc^-:=\{\omega\in \Xi_\sc|\ 0\leq\langle\omega+\rho,\check{\alpha}\rangle <l,\ \forall \alpha\in \Phi^+ \}$. 
The main result of this subsection is the following. 
\begin{prop}\label{prop D.3} 
There are invertible elements $\lambda_\omega\in H^\bullet_{\Gm}(\Fl^{\omega,\circ})_{\widehat{0}}$ for each $\omega\in \Xi_\sc^-$ such that the following diagram commutes 
$$\begin{tikzcd}[column sep=large]
H^\bullet_{\Gm}(\Fl^{\circ})_{\widehat{0}} \arrow[r,"{\lambda^{-1}_\omega\circ \pi_*\circ \lambda_0}"] \arrow[d,"\bc"'] 
& H^\bullet_{\Gm}(\Fl^{\omega,\circ})_{\widehat{0}} \arrow[d,"\bc"]\\ 
Z(\rep^0(U_{\hat{\zeta}}))\arrow[r,"\tr_{\sfT_\omega^0}"] & Z(\rep^\omega(U_{\hat{\zeta}})). 
\end{tikzcd}$$ 
where $\pi_*$ is the pushforward associated with the natural projection $\pi:\Fl^{\circ}\rightarrow \Fl^{\omega,\circ}$. 
\end{prop}
\begin{rmk}
See in \cite[Prop B.10]{Situ1} a closely related statement. 
\end{rmk}

%=================================
\subsubsection{Construction of the map $\bc$} 
%=================================
We firstly recall the construction of the map $\bc$. 
The Harish-Chandra isomorphism $\bA[T/W]=Z(U_q)$ induces a $\k[\![\hbar]\!]$-algebra homomorphism  
\begin{equation}\label{equ D.60} 
\k[\hbar][T/W]_{\widehat{0}}\rightarrow Z(U_{\hat{\zeta}}). 
\end{equation} 
Since $U_{\hat{\zeta}}$ is torsion free over $\k[\![\hbar]\!]$, there is an inclusion $Z(U_{\hat{\zeta}})/\hbar Z(U_{\hat{\zeta}})\subset Z(U_{\zeta})$. 
By (\ref{equ 2.-1}), the specialization of (\ref{equ D.60}) at $\hbar=0$ induces a chain of maps 
\begin{equation}\label{equ B.9} 
\k[T/W]\rightarrow \k[G^*\times_{T/W}T/W]\twoheadrightarrow \k[1\times_{T/W}T/W]=\k[\Omega]\rightarrow Z(U_{\zeta}).
\end{equation}
Define the \textit{deformation to the normal cone} $\widetilde{N}_\Omega(T/W)$ to be the affine scheme with 
$$\k[\widetilde{N}_\Omega(T/W)]=\k[T/W][\hbar]+\sum_{n>0}\hbar^{-n}I_\Omega^n,$$ 
where $I_\Omega$ is the defining ideal for $\Omega$ in $\k[T/W]$. 
By (\ref{equ B.9}), the map (\ref{equ D.60}) extends to a homomorphism 
$$\k[\widetilde{N}_\Omega(T/W)]_{\widehat{0}}\rightarrow Z(U_{\hat{\zeta}}).$$ 
Consider the composition $\bc':\k[\widetilde{N}_\Omega(T/W)]_{\widehat{0}}\rightarrow Z(U_{\hat{\zeta}})\rightarrow Z(\rep(U_{\hat{\zeta}}))$. 
Similarly as (\ref{equ 3.2}), the category $\rep(U_{\hat{\zeta}})$ admits a $\pi_1$-grading $\rep(U_{\hat{\zeta}})=\bigoplus_{\gamma\in \pi_1}\rep(U_{\hat{\zeta}})^\gamma$, which yields a decomposition 
\begin{equation}\label{equ B.10} 
Z(\rep(U_{\hat{\zeta}}))=\bigoplus_{\gamma\in \pi_1}Z(\rep(U_{\hat{\zeta}})^\gamma). 
\end{equation}
We denote by $p_\gamma$ (resp. $i_\gamma$) the projection to (resp. the embedding from) the $\gamma$-th direct factor in (\ref{equ B.10}), and define 
\begin{equation}\label{equ D.7} 
\bc=\sum_{\gamma\in \pi_1} i_\gamma\circ p_\gamma\circ \bc':\ \k[\widetilde{N}_\Omega(T/W)]_{\widehat{0}}^{\oplus \pi_1}\rightarrow Z(\rep(U_{\hat{\zeta}})). 
\end{equation}
Set $\sigma_\omega$ be the scheme-theoretic fiber at $0\in \t/W$ of $\t/W_\omega\twoheadrightarrow \t/W$, then by the decomposition $\Omega=\bigsqcup_{[\omega]\in \Xi} \Omega_{[\omega]}$, we have 
\begin{equation}\label{equ D.8} 
\k[\widetilde{N}_\Omega(T/W)]_{\widehat{0}}^{\oplus \pi_1}
=\bigsqcup_{[\omega]\in \Xi} \k[\widetilde{N}_{\Omega_{[\omega]}}(T/W)]_{\widehat{0}}^{\oplus \pi_1}
=\bigsqcup_{[\omega]\in \Xi} \k[\widetilde{N}_{\sigma_\omega}(\t/W_{\omega})]_{\widehat{0}}^{\oplus \pi_1}, 
\end{equation} 
where we simplify $W_{\omega}=W_{\zeta^\omega}$. 
Finally, the map (\ref{equ D.9}) is obtained from (\ref{equ D.7}), (\ref{equ D.8}) and the following isomorphism in \cite[(2.16)]{BBASV}, 
\begin{equation}\label{equ D.10} 
H^\bullet_{\Gm}(\Fl^\omega)=\k[\widetilde{N}_{\sigma_{\omega}}(\t/W_{\omega})]^{\oplus \pi_1}.
\end{equation} 

We then explain the isomorphism (\ref{equ D.8}) more precisely. 
We identify the algebras $\k[T]_{\widehat{1}}=\k[\t]_{\widehat{0}}$ via the exponential map $\exp:\t \rightarrow T$. 
For any $\omega\in \Xi_\sc$, we denote by $\bA[T/W]_{\widehat{\zeta^{2(\omega+\rho)}}}$ the completion of the Harish-Chandra center $Z(U_q)=\bA[T/W]$ at $q_e=\zeta_e$ and $W(\zeta^{2(\omega+\rho)})\in T/W$. 
Since $T/W_\omega\rightarrow T/W$ is \'{e}tale at $W_\omega(\zeta^{2(\omega+\rho)})$, there are isomorphisms 
\begin{equation}\label{equ D.11} 
\bA[T/W]_{\widehat{\zeta^{2(\omega+\rho)}}}\xs 
\bA[T/W_\omega]_{\widehat{\zeta^{2(\omega+\rho)}}}
\xrightarrow[\sim]{\tau^{2}_{\omega+\rho}} \bA[T/W_\omega]_{\widehat{1}}\xs \k[\hbar][\t/W_\omega]_{\widehat{(0,0)}}. 
\end{equation} 
Under the isomorphisms above, the ideal of the closed subscheme $\Omega_{[\omega]}$ supported at $W_\omega(\zeta^{2(\omega+\rho)})$ in $\bA[T/W]_{\widehat{\zeta^{2(\omega+\rho)}}}$ corresponds to the one for $\sigma_\omega$ in $\k[\hbar][\t/W_\omega]_{\widehat{(0,0)}}$. 
It yields an isomorphism 
$$\k[\widetilde{N}_{\Omega_{[\omega]}}(T/W)]_{\widehat{0}}=\k[\widetilde{N}_{\sigma_\omega}(\t/W_\omega)]_{\widehat{0}}.$$ 

%=================================
\subsubsection{Push-forward of cohomology} 
%=================================
The isomorphism (\ref{equ D.10}) restricts to an isomorphism 
\begin{equation}\label{equ B.15}
H^\bullet_{\Gm}(\Fl^{\omega,\circ})=\k[\widetilde{N}_{\sigma_{\omega}}(\t/W_{\omega})] 
\end{equation}
on each component. 
Consider a linear map 
$$\pi'_*: \k[\t]\rightarrow \k[\t/W_\omega],\quad f\mapsto \frac{\sum_{x\in W_\omega}(-1)^{\ell(x)}x(f)}{\Lambda_\omega},$$ 
where $\Lambda_\omega:=\prod\limits_{\alpha\in \Phi^+,s_\alpha\in W_\omega} \alpha$. 
Since $\pi'_*$ is $\k[\t/W]$-linear, it extends to a $\k[\hbar]$-linear map 
$$\pi'_*: \k[\widetilde{N}_{\sigma_{0}}(\t)]\rightarrow \k[\widetilde{N}_{\sigma_{\omega}}(\t/W_{\omega})].$$ 
\begin{lem}\label{lem B.5} 
If $\omega\in \Xi_\sc^-$, then the following diagram commutes, 
$$\begin{tikzcd}
\k[\widetilde{N}_{\sigma_0}(\t)] \arrow[d,"\pi'_*"]\arrow[r,"\simeq"] 
& H^\bullet_{\Gm}(\Fl^{\circ}) \arrow[d,"\pi_*"]\\ 
\k[\widetilde{N}_{\sigma_\omega}(\t/W_\omega)] \arrow[r,"\simeq"] 
& H^\bullet_{\Gm}(\Fl^{\omega,\circ}).
\end{tikzcd}$$ 
\end{lem}
\begin{proof} 
We need a $\check{T}\times \Gm$-equivariant version of the isomorphism (\ref{equ B.15}), which is recalled as follows. 
Let $\Delta_\omega:=\t\times_{\t/W}\t/W_\omega$ be a closed subscheme of $\t\times\t/W_\omega$. 
We identify $H^\bullet_{\check{T}\times \Gm}(\pt)=\k[\t][\hbar]$. 
By \cite[(2.16)]{BBASV}, there is an isomorphism of $\k[\t][\hbar]$-algebras 
$$\k[\widetilde{N}_{\Delta_\omega}(\t\times\t/W_\omega)]\xs H^\bullet_{\check{T}\times \Gm}(\Fl^{\omega,\circ}).$$ 
Its composition with the restriction to the $\check{T}$-fixed points 
$$\k[\widetilde{N}_{\Delta_\omega}(\t\times\t/W_\omega)]\rightarrow \Fun(W^{\omega}_{\af},\k[\t][\hbar])$$ 
satisfies 
\begin{equation}\label{equ B.16}
g\otimes f\mapsto \big(g\cdot x(f)\big)_{x\in W^{\omega}_{\af}}, \quad \forall g\in \k[\t][\hbar],\ \forall f\in \k[\t/W_\omega]. 
\end{equation} 
Since $\omega\in \Xi_\sc^-$, we have $W_\omega=W_{l,\omega}$, and the set $J_\omega$ of the corresponding $l$-affine simple coroots is contained in $\check{\Sigma}$. 
So the parahoric subgroup $P^{J_\omega}$ of $\check{G}(\!(t^l)\!)$ is the preimage of a parabolic subgroup $P_\omega$ in $\check{G}$ under the evaluation map $\check{G}[\![t^l]\!]\rightarrow \check{G}$ via $t\mapsto 0$. 
Now $\Fl^\circ\rightarrow \Fl^{\omega,\circ}$ is a locally trivial $P_\omega/\check{B}$-fibration. 
The $\check{T}$-equivariant Euler class of the normal bundle of $x\check{B}/\check{B}$ in $P_\omega/\check{B}$ is given by $(-1)^{\ell(x)}\Lambda_\omega$, for any $x\in W_\omega$. 
We obtain a commutative diagram 
\begin{equation}\label{equ B.17} 
\begin{tikzcd}
H^\bullet_{\check{T}\times \Gm}(\Fl^{\circ}) \arrow[d,"\pi_*"] \arrow[r,hook] 
& \Fun(W_{\af},\k(\t)[\hbar]) \arrow[d,"\pi''_*"]\\ 
H^\bullet_{\check{T}\times \Gm}(\Fl^{\omega,\circ}) \arrow[r,hook] 
& \Fun(W^{\omega}_{\af},\k(\t)[\hbar]), 
\end{tikzcd}
\end{equation} 
where $\pi''_*$ is given by 
$$\pi''_*: (f_y)_{y\in W_{\af}} \mapsto \big(\frac{\sum_{x\in W_\omega}(-1)^{\ell(x)} f_{yx}}{y(\Lambda_\omega)}\big)_{y\in W^{\omega}_{\af}}.$$ 
The map $\pi'_*$ extends $\k[\t]$-linearly to a map $\pi'_*: \k[\t\times\t]\rightarrow \k[\t\times\t/W_\omega]$, which further yields a $\k[\t][\hbar]$-linear map 
$$\pi'_*: \k[\widetilde{N}_{\Delta_0}(\t\times\t)]\rightarrow 
\k[\widetilde{N}_{\Delta_\omega}(\t\times\t/W_\omega)].$$ 
By (\ref{equ B.16}) and the diagram (\ref{equ B.17}), we have a commutative diagram 
$$\begin{tikzcd}
\k[\widetilde{N}_{\Delta_0}(\t\times\t)] \arrow[d,"\pi'_*"]\arrow[r,"\simeq"] 
& H^\bullet_{\check{T}\times \Gm}(\Fl^{\circ}) \arrow[d,"\pi_*"]\\ 
\k[\widetilde{N}_{\Delta_\omega}(\t\times\t/W_\omega)] \arrow[r,"\simeq"] 
& H^\bullet_{\check{T}\times \Gm}(\Fl^{\omega,\circ}). 
\end{tikzcd}$$ 
Specializing $H^\bullet_{\check{T}}(\pt)=\k[\t]$ at $0\in \t$, we get the desired commutative diagram from the one above. 
\end{proof}

\begin{proof}[Proof of Proposition \ref{prop D.3}] 
We adopt the notations $p_{[0]}$ and $n_{[0]}$ from \textsection \ref{subsect 5.3.4}, and recall that $\bL=\bL_\omega\cdot \bL'_\omega$ with $\bL_\omega$ and $\bL'_\omega$ defined in (\ref{equ 4.e15}). 
By \cite[Lem 5.13]{APW91}, any projective module in $ \rep(U_{\hat{\zeta}})$ admits a finite filtration with composition factors given by Weyl modules of $U_{\hat{\zeta}}$. 
Let $Q$ be a projective module in $\rep^\omega(U_{\hat{\zeta}})$, then any quotient $Q_n=Q/\hbar^nQ$ is an extension by finitely many Weyl modules $V(\lambda)_\k$ with $\lambda\in W_{l,\af}\bullet \omega$. 
Similar arguments as in the proof of Lemma \ref{lem 5.14} show that, there is $p_{[0]}\in \bA[T/W]$ with $n_{[0]}$ large enough, such that for any $f\in \bA[T/W]_{\widehat{\zeta^{2\rho}}}$, we have 
\begin{align*}
\tr_{\sfT_\omega^0}(f)|_{Q_n}&=\tr_{V}(f\cdot p_{[0]})|_{Q_n}\\ 
&=\frac{\sum_{\nu\in \bP(V)}\tau^{2}_\nu(f\cdot p_{[0]}\cdot \bL)}{\bL}|_{Q_n}\\ 
&=|\Stab_{W}(\omega)|^{-1} \frac{\sum_{x\in W}\tau^{2}_{-x\omega}(f\cdot \bL)}{\bL}|_{Q_n}, 
\end{align*} 
where for any $g\in Z(\rep(U_{\hat{\zeta}}))$, $g|_{Q_n}$ refers to its image in $Z(\rep(U_{\hat{\zeta}}))\rightarrow \End(Q_n)$. 
Since $\tau_{-x\omega}=x\tau_{-\omega}x^{-1}$ and $\bL'_\omega$ is invertible in $\bA[T/W_\omega]_{\widehat{\zeta^{2(\omega+\rho)}}}$, we have 
\begin{align*} 
\frac{\sum_{x\in W}\tau^{2}_{-x\omega}(f\cdot \bL)}{\bL}|_{Q_n}
&=\frac{\sum_{x\in W}(-1)^{\ell(x)}x \tau^{-2}_{\omega}(f\cdot \bL)}{\bL}|_{Q_n}\\ 
&=\sum_{x\in W/W_\omega}x\big[
\frac{1}{\bL'_\omega}\cdot \frac{\sum_{y\in W_\omega}(-1)^{\ell(y)}y \tau^{-2}_{\omega}(f\cdot \bL)}{\bL_\omega}\big]|_{Q_n}\\ 
&=\sum_{x\in W/W_\omega}x\tau^{-2}_{\omega+\rho}\big[ 
\frac{1}{\tau^{2}_{\omega+\rho}(\bL'_\omega)}\cdot 
\frac{\sum_{y\in W_\omega}(-1)^{\ell(y)}y\big(\tau^{2}_{\rho}(f)\cdot \tau^{2}_{\rho}(\bL)\big)}{\bL_\omega}\big]|_{Q_n}.
\end{align*} 
We define an invertible element in $\k[\hbar][\t/W_\omega]_{\widehat{(0,0)}}$ by 
$$\lambda_\omega=\frac{|\Stab_{W}(\omega)|\cdot |W_\omega|}{|W|}\cdot \frac{\bL_\omega}{\Lambda_\omega}\cdot \tau^2_{\omega+\rho}(\bL'_\omega).$$ 
By the isomorphisms (\ref{equ D.11}), and the observation that the inverse of the first map in (\ref{equ D.11}) is given by 
$$\bA[T/W_\omega]_{\widehat{\zeta^{2(\omega+\rho)}}}\xs \bA[T/W]_{\widehat{\zeta^{2(\omega+\rho)}}}, \quad f\mapsto \frac{|W_\omega|}{W}\sum_{x\in W/W_\omega}x(f),$$ 
we deduce that for any $f\in \k[\hbar][\t]_{\widehat{(0,0)}}$, there is an equality 
$$\tr_{\sfT^0_\omega}(f)|_{Q_n}=\lambda_\omega^{-1}\pi'_*(f\lambda_0)|_{Q_n}.$$ 
Since $Q=\Lim{n}Q_n$, the equality above holds if $Q_n$ is replaced by $Q$. 
Hence there is a commutative diagram 
$$\begin{tikzcd}[column sep=large]
\k[\hbar][\t]_{\widehat{(0,0)}}\arrow[d] \arrow[r,"\lambda^{-1}_\omega\circ \pi'_*\circ \lambda_0"] 
& \k[\hbar][\t/W_\omega]_{\widehat{(0,0)}} \arrow[d]\\ 
Z(\rep^0(U_{\hat{\zeta}}))\arrow[r,"\tr_{\sfT_\omega^0}"] & Z(\rep^\omega(U_{\hat{\zeta}})).
\end{tikzcd}$$
Since $Z(\rep(U_{\hat{\zeta}}))$ is torsion free over $\k[\![\hbar]\!]$, it extends to a commutative diagram 
$$\begin{tikzcd}[column sep=large]
\k[\widetilde{N}_{\sigma_0}(\t)]_{\widehat{0}} \arrow[r,"{\lambda^{-1}_\omega\circ \pi'_*\circ \lambda_0}"] \arrow[d,"\bc"'] 
& \k[\widetilde{N}_{\sigma_\omega}(\t/W_\omega)]_{\widehat{0}} \arrow[d,"\bc"]\\ 
Z(\rep^0(U_{\hat{\zeta}}))\arrow[r,"\tr_{\sfT_\omega^0}"] & Z(\rep^\omega(U_{\hat{\zeta}})). 
\end{tikzcd}$$ 
Now our conclusion follows from Lemma \ref{lem B.5}. 
\end{proof}

\

\nocite{*} 

\bibliographystyle{plain} 
\bibliography{MyBibtex2}

\end{document}